\documentclass[12pt]{article}
\usepackage[dvips]{graphicx}
\usepackage{amsmath}
\usepackage{amsthm}
\usepackage{amssymb}

\theoremstyle{break}
\newtheorem{theo}{Theorem}[section]
\newtheorem{defi}{Definition}[section]
\newtheorem{prop}{Proposition}[section]
\newtheorem{lemm}{Lemma}[section]
\newtheorem{cor}{Corollary}[section]
\newtheorem{rem}{Remark}[section]
\newtheorem{conj}{Conjecture}[section]

\makeatletter

\@addtoreset{equation}{section}
\makeatother

\makeatletter
\@addtoreset{figure}{section}

\makeatother

\title{Voros coefficients of the third Painelv\'e equation 
and parametric Stokes phenomena}
 
\author{Kohei Iwaki (RIMS, Kyoto University)}
\date{\today}

\begin{document}
\maketitle


\begin{abstract}
We compute all ``Voros coefficients'' of 
the third Painlev\'e equation 
$(P_{\rm III'})_{D_{6}}$ of the type $D_{6}$ 
(in the sense of \cite{OKSO}) and discuss 
the ``parametric Stokes phenomena'' 
occurring to formal transseries solutions of 
$(P_{\rm III'})_{D_{6}}$. 
We derive connection formulas for 
parametric Stokes phenomena under 
an assumption for Borel summability 
of transseries solutions. 
Furthermore, we also compute the Voros coefficient 
of the degenerate third Painlev\'e equation 
of the type $D_{7}$ in 
Appendix \ref{Appendix:P3-D7}. 

\vspace{+.5em}

{\it Key Wards} : The third Painlev\'e equation, 
exact WKB analysis, Voros coeffieints, 
parametric Stokes phenomena.

\vspace{+.5em}


\end{abstract}

\section{Introduction}
\label{section:Introduction}

In this article we study the third Painlev\'e equation 
with a large parameter $\eta > 0$
\[
(P_{\rm III'})_{D_{6}} : \frac{d^{2}\lambda}{dt^{2}} = 
\frac{1}{\lambda} \Bigl( \frac{d \lambda}{dt} \Bigr)^{2}
-
\frac{1}{t} \frac{d \lambda}{dt} 
+
\eta^{2} \Bigl( 
\frac{\lambda^{3}}{t^{2}} - \frac{c_{\infty} \lambda^{2}}{t^{2}}
+ \frac{c_{0}}{t} - \frac{1}{\lambda}
\Bigr)
\]
from the view point of the exact WKB analysis 
(cf.\ \cite{KT iwanami, Voros}); that is, 
an asymptotic analysis for large $\eta$
based on the Borel resummation method. 
(See \cite{KT iwanami} and series of papers 
\cite{KT WKB Painleve I, AKT Painleve WKB, 
KT WKB Painleve III} for WKB analysis of Painlev\'e equations)
Here $c_{\infty}$ and $c_{0}$ are non-zero complex parameters, 
and the equation is the type $D_{6}$ 
in the sense of \cite{OKSO}. 
We are interested in (one-parameter family of) 
formal transseries solutions of $(P_{\rm III'})_{D_{6}}$: 
\[
\lambda(t,\eta; \alpha) = \lambda^{(0)}(t,\eta) + 
\alpha \eta^{-1/2} \lambda^{(1)}(t,\eta) e^{\eta \phi(t)} + 
(\alpha \eta^{-1/2})^{2} \lambda^{(2)}(t,\eta) e^{2 \eta \phi(t)} + 
\cdots,
\]
where $\alpha$ is a free parameter, 
$\lambda^{(k)}(t,\eta) = \lambda_{0}^{(k)}(t) + 
\eta^{-1} \lambda_{1}^{(k)}(t) + 
\eta^{-2} \lambda_{2}^{(k)}(t) + \cdots$  
($k \ge 0$) is a formal 
power series in $\eta^{-1}$ and $\phi(t)$ is a certain 
function. 

What we discuss in this article is so-called 
``{\it parametric Stokes phenomena}'' occurring to the 
transseries solutions of $(P_{\rm III'})_{D_{6}}$. 
These are kinds of Stokes phenomena concerning with  
continuous variations of the complex parameters $c_{\infty}$ 
and $c_{0}$. In \cite{Iwaki, Iwaki-Bessatsu} we discussed 
the parametric Stokes phenomena occurring to 
transseries solutions of the second Painlev\'e
equation with the large parameter $\eta > 0$
\[
(P_{\rm II}) : \frac{d^{2}\lambda}{dt^{2}} = 
\eta^{2} (2 \lambda^{3} + t \lambda + c),
\] 
where $c$ is a complex parameter. 
As is shown in \cite{Iwaki, Iwaki-Bessatsu}, 
when $\arg c = \pi/2$, transseries solutions 
may not be Borel summable (as a formal series in $\eta^{-1}$), 
and the Borel sums defined when $\arg c = \pi/2 - \varepsilon$ 
and $\arg c = \pi/2 + \varepsilon$ give different 
analytic solutions (after analytic continuation with 
respect to the parameter $c$ across the imaginary axis 
$\arg c = \pi/2$) of $(P_{\rm II})$. 
This kind of phenomenon was firstly observed in \cite{SS} 
through the analysis of a linear ordinary differential equation, 
and the analysis in terms of the exact WKB analysis 
was given by \cite{Takei Sato conjecture}. 
(See also \cite{Koike-Takei, Aoki-Tanda}.) 

As discussed in \cite{Takei Sato conjecture, Iwaki}, 
the parametric Stokes phenomena are closely related to 
the ``{\it degeneration of the Stokes geometry}''
(i.e., two turning points are connected by a Stokes curve), 
and the ``{\it Voros coefficients}'' 
(cf.\ \cite{DDP, Takei Sato conjecture}) play a central 
role when we discuss the connection problem for parametric 
Stokes phenomena. That is, formal solutions may not 
be Borel summable when the Stokes geometry degenerates, 
and an explicit connection formula describing 
the parametric Stokes phenomena can be read off the 
``jump property'' of the Voros coefficients. 

  \begin{figure}[h]
  \begin{minipage}{0.5\hsize}
  \begin{center} 
  \includegraphics[width=50mm,clip]
  {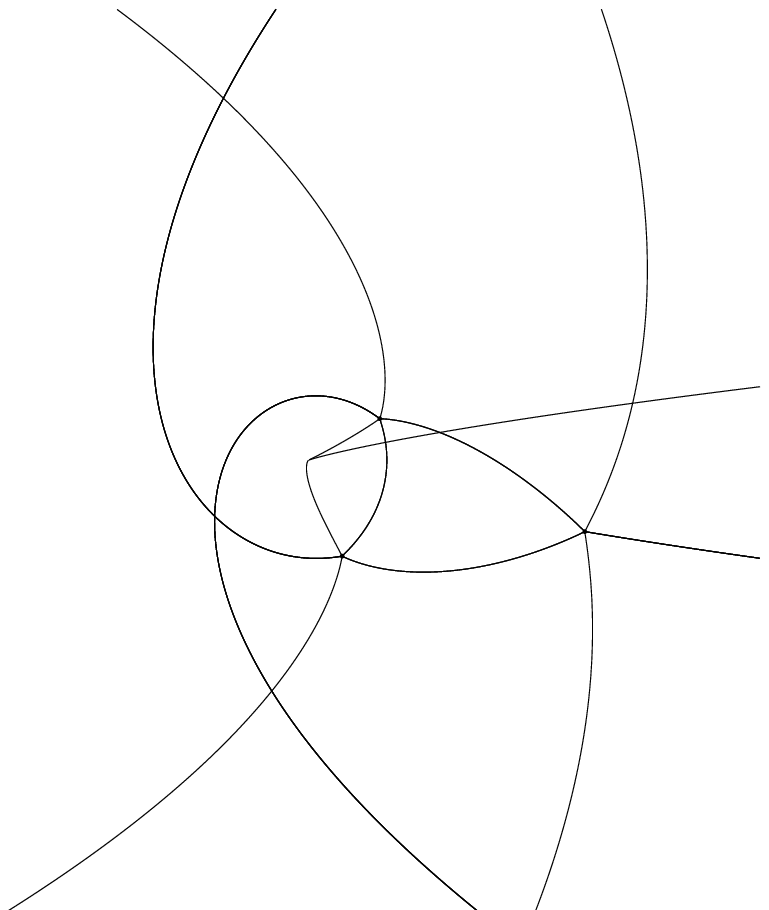}    
  \end{center} 
  \caption{\small{$(c_{\infty},c_{0}) = (3,3-i)$.}}
  \label{fig:PIII'-triangle-intro}
  \end{minipage} \hspace{-.3em}
  \begin{minipage}{0.5\hsize}
  \begin{center}
  \includegraphics[width=50mm, clip]{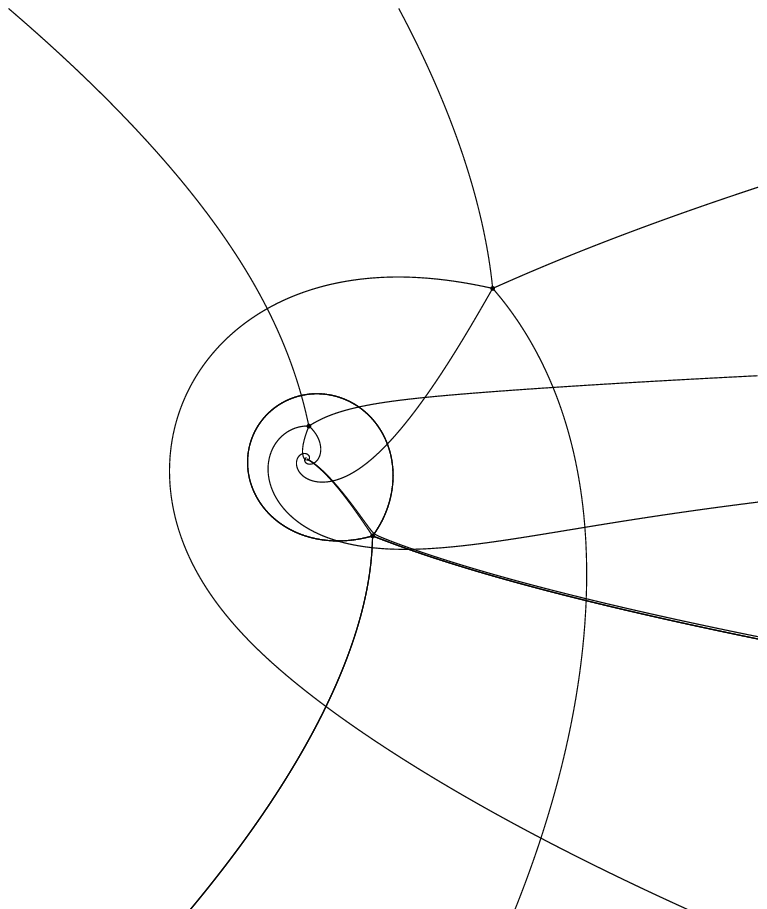}
  \end{center} 
  \caption{\small{$(c_{\infty},c_{0}) = (3i,1-2i)$}.}
  \label{fig:PIII'-loop-intro}
  \end{minipage}
  \end{figure}

The degeneration of the Stokes geometry is also 
observed for $(P_{\rm III'})_{D_{6}}$. For example, 
Figure \ref{fig:PIII'-triangle-intro} 
and \ref{fig:PIII'-loop-intro} 
describe some Stokes geometry of $(P_{\rm III'})_{D_{6}}$ 
for some values of parameters $(c_{\infty}, c_{0})$, and 
we can observe that the Stokes geometry degenerate. 
We observe the following two kinds of degenerations 
in the Stokes geometry of $(P_{\rm III'})_{D_{6}}$. 
The first one is ``{\it triangle-type degeneration}''; 
that is, there are three pairs of turning points 
connected by a Stokes curve simultaneously. 
Figure \ref{fig:PIII'-triangle-intro} shows 
an example of the triangle-type degeneration. 
The second one is ``{\it loop-type degeneration}''; 
that is, a Stokes curve form a loop around the singular 
point $t = 0$. An example the loop-type degeneration is shown 
in Figure \ref{fig:PIII'-loop-intro}. Intriguingly enough, 
as far as we checked, only these two kinds of degeneration 
can be observed for the Stokes geometry of $(P_{\rm III'})_{D_{6}}$. 
(See Appendix \ref{Appendix:examples of P-Stokes geometry}.)

As similar to \cite{SS, Takei Sato conjecture, Iwaki}, 
we can expect that parametric Stokes phenomena 
also occur for transseries solutions of $(P_{\rm III'})_{D_{6}}$ 
relevant to these degenerations of Stokes geometry. 
The purpose of this article is to compute the all 
Voros coefficients of $(P_{\rm III'})_{D_{6}}$, 
and to analyze the parametric Stokes phenomena 
occurring to the transseries solutions.
Voros coefficients of $(P_{\rm III'})_{D_{6}}$ 
(cf.\ \cite{Iwaki}) are formal power series defined 
by the integral 
\[
W(c_{\infty},c_{0},\eta) = 
\int_{\tau}^{*} \bigl( R_{\rm odd}(t,\eta) - 
\eta R_{-1}(t) \bigr) \hspace{+.1em} dt, 
\]
where 
$R_{\rm  odd}(t,\eta) = \sum_{n = 0}^{\infty} 
\eta^{1-2n} R_{2n-1}(t)$ given by \eqref{eq:R-odd} 
is the odd part of a formal solution of the Riccati equation 
\eqref{eq:Riccati equation} associated with 
the Fr\'echet derivative \eqref{eq:Frechet bibun} 
of $(P_{\rm III'})_{D_{6}}$, 
$\tau$ is a turning point and $*$ $(= 0$ or $\infty)$ 
is a singular point of $(P_{\rm III'})_{D_{6}}$. 
There are several Voros coefficients 
depending on the choice of the singular point $*$, 
and of the path of integration. 
One of Voros coefficients is represented 
as follows (Theorem \ref{Main Theorem}):
\begin{theo} 
The Voros coefficient for an appropriate path 
from a turning point to the singular point 
$* = \infty$ is given by 
\begin{equation}
W(c_{\infty},c_{0},\eta) = \int_{\tau}^{\infty} 
\bigl( R_{\rm odd}(t,\eta) - 
\eta R_{-1}(t) \bigr) \hspace{+.1em} dt = 
\sum_{n = 1}^{\infty} \frac{2^{1-2n} - 1}{2n(2n-1)} 
B_{2n} \biggl( \frac{c_{\infty} - c_{0}}{2} 
\hspace{+.1em} \eta \biggr)^{1-2n}.  
\label{eq:Voros coeff in introduction}
\end{equation}
Here $B_{2n}$ is the $2n$-th Bernoulli number 
defined by 
\begin{equation}
\frac{w}{e^{w}-1} = 1 - \frac{1}{2} w 
+ \sum_{n=1}^{\infty} \frac{B_{2n}}{(2n)!} w^{2n}. 
\label{eq:Bernoulli number}
\end{equation}
\end{theo}
All other Voros coefficients can also be expressed 
in terms of the Bernoulli numbers. They are summarized in 
Theorem \ref{Main Theorem3} and Theorem \ref{Main Theorem2}. 
Moreover, from above explicit representation 
we can compute its Borel sum and check the 
following jump property of the Voros coefficients 
(Proposition \ref{Prop:Borel sum of F and G} and 
Corollary \ref{corollary:jump property}).
\begin{prop}
Let $W(c_{\infty}, c_{0},\eta)$ be the formal power series 
\eqref{eq:Voros coeff in introduction}. \vspace{+.2em}

\noindent
(i) $W(c_{\infty}, c_{0},\eta)$ is not Borel summable 
(as a formal power series in $\eta^{-1}$) when 
$c_{\infty} - c_{0} \in i {\mathbb R}$, 
and Borel summable otherwise. \vspace{+.2em}

\noindent
(ii) Let ${\cal S}_{\pm}[e^{W(c_{\infty}, c_{0},\eta)}]$ 
be the Borel sum of the formal power series 
$e^{W(c_{\infty}, c_{0},\eta)}$ when 
$\arg (c_{\infty}-c_{0}) = \pi/2 \pm \varepsilon$ 
for a sufficiently small $\varepsilon > 0$. 
After the analytic continuation across 
the axis $\{ \arg (c_{\infty}-c_{0}) = \pi/2 \}$, the following 
relation holds:
\begin{equation}
{\cal S}_{+}[e^{W(c_{\infty}, c_{0},\eta)}] = 
(1 + e^{\pi i (c_{\infty} - c_{0})\eta}) \hspace{+.1em}
{\cal S}_{-}[e^{W(c_{\infty}, c_{0},\eta)}].
\end{equation}
\end{prop}
From this jump property and a certain assumption for 
the Borel summability of the transseries solution 
(Conjecture \ref{conjecture:Kamimoto}), 
when the independent variable $t$ is fixed at a point, 
we derive an explicit connection formula 
describing the parametric Stokes phenomenon relevant to 
the triangle-type degeneration in Figure \ref{fig:PIII',0} 
(Section \ref{section:connection formulas}). \\[+.5em]
\textbf{Connection formula for the 
parametric Stokes phenomenon.} \hspace{+.5em}
\textit{Assume that Conjecture \ref{conjecture:Kamimoto}
holds. Let ${\cal S}_{II}[\lambda(t,\eta;\alpha)]$
(resp., ${\cal S}_{I}[\lambda(t,\eta;\tilde{\alpha})]$)
be the Borel sum of a transseries solution 
$\lambda(t,\eta;\alpha)$ when 
$(c_{\infty},c_{0}) = (2 - \varepsilon, 2 -i)$ 
(resp., $\lambda(t,\eta;\tilde{\alpha})$
when $(c_{\infty},c_{0}) = (2 + \varepsilon, 2 -i)$).
If they represent the same analytic solution of 
$(P_{\rm III'})_{D_{6}}$ after the analytic continuation 
with respect to the parameter $(c_{\infty},c_{0})$ 
across $\{ \arg (c_{\infty}-c_{0}) = \pi/2 \}$, 
then the parameters related as 
\begin{equation}
\tilde{\alpha} = (1 + e^{\pi i (c_{\infty} - c_{0}) \eta}) 
\hspace{+.1em} \alpha.
\end{equation}
That is, the following relation holds:
\begin{equation}
{\mathcal S}_{II}\bigl[\lambda_{\tau_{1}}
(t,\eta;\alpha) \bigr]
= {\mathcal S}_{I}\bigl[\lambda_{\tau_{1}}
(t,\eta;\tilde{\alpha}) 
\bigr]\Bigl|_{\tilde{\alpha} = 
(1 + e^{\pi i (c_{\infty} - c_{0}) \eta}) 
\hspace{+.1em} \alpha}. 
\label{eq:THE connection formula in introduction}
\end{equation}
}

This formula describes the parametric Stokes phenomenon; 
that is, \eqref{eq:THE connection formula in introduction}
gives the explicit relationship between the Borel sums of 
the transseries solutions of $(P_{\rm III'})_{D_{6}}$
in different region in the parameter space of 
$(c_{\infty},c_{0})$. We also discuss another type connection 
formula in Section \ref{section:connection formulas}.
These are our main results. 
However, we have succeeded to derive 
the explicit connection formulas only in certain cases; 
a crucial case (when the loop-type degeneration happens and 
$t$ lies inside the loop) remains to be analyzed. 
(See Section \ref{subsection:connection of loop-type}.) 

This article is organized as follows.  In Section 
\ref{section:transseries solutions and P-Stokes geometry} 
we recall the fundamental notions in the WKB theory of 
Painlev\'e equations. In Section 
\ref{section:P-Voros coefficients} 
we recall the definition of the Voros coefficients 
for non-linear differential equations, 
and state the main results about the explicit 
representations of the Voros coefficients 
(Theorem \ref{Main Theorem3} and 
Theorem \ref{Main Theorem2}). 
Section \ref{section:proof of the Main theorems}
is devoted to the proof of these main theorems.
Based on these results, we discuss the connection 
problems for parametric Stokes phenomena 
in Section \ref{section:connection formulas}.
Furthermore, we also compute the all Voros coefficients 
of the degenerate third Painlev\'e equation 
\[
(P_{\rm III'})_{D_{7}} : 
\frac{d^{2}\lambda}{dt^{2}} = 
\frac{1}{\lambda} \Bigl( \frac{d \lambda}{dt} \Bigr)^{2}
-
\frac{1}{t} \frac{d \lambda}{dt} 
+
\eta^{2} \Bigl( - \frac{2 \lambda^{2}}{t^{2}}
+ \frac{c}{t} - \frac{1}{\lambda}
\Bigr)
\]
of the type $D_{7}$ in Appendix \ref{Appendix:P3-D7}
(Theorem \ref{theorem:D7-Voros}). Note that a part of 
this result is announced in \cite{Iwaki-Bessatsu}. 

\subsection*{Acknowledgment}
The author is very grateful to 
Yoshitsugu Takei, Takahiro Kawai, Takashi Aoki, 
Tatsuya Koike, Yousuke Ohyama, Shingo Kamimoto, 
Shinji Sasaki and Mika Tanda for helpful advices.

\section{Transseries solutions and Stokes geometry of 
$(P_{\rm III'})_{D_{6}}$}
\label{section:transseries solutions and P-Stokes geometry}

In this section we briefly review a construction of 
{transseries solutions} of the third Painelv\'e equation 
of the type $D_{6}$ with a large parameter 
$\eta > 0$ of the following form:
\begin{equation}
(P_{\rm III'})_{D_{6}} : \frac{d^{2}\lambda}{dt^{2}} = 
\frac{1}{\lambda} \Bigl( \frac{d \lambda}{dt} \Bigr)^{2}
-
\frac{1}{t} \frac{d \lambda}{dt} 
+
\eta^{2} \Bigl( 
\frac{\lambda^{3}}{t^{2}} - \frac{c_{\infty} \lambda^{2}}{t^{2}}
+ \frac{c_{0}}{t} - \frac{1}{\lambda}
\Bigr). \label{eq:P3-D6}
\end{equation}
Here $c_{\infty}$ and $c_{0}$ are complex parameters.
In this paper we write ${\bf c} = (c_{\infty},c_{0})$ 
and always assume 
\begin{equation}
{\bf c} \in {\bf S} = 
\{(c_{\infty}, c_{0}) \in {\mathbb C}^{2}; 
c_{\infty}, c_{0}, c_{\infty}^{2} - c_{0}^{2}, 
c_{\infty}^{2} + c_{0}^{2} \ne 0 \}
\label{eq:THE condition}
\end{equation}
for genelicity. 
We will explain the meaning of the assumptions 
\eqref{eq:THE condition} in 
Remark \ref{Remark:R-turning points}.
Moreover, we also recall the definitions of 
{turning points} and {Stokes curves}
(\cite{KT WKB Painleve I}) in this section. 

Note that the equation $(P_{\rm III'})_{D_{6}}$ 
is obtained from the ``original'' one 
(cf.\cite{Okamoto Painleve, OKSO})
\[
\frac{d^{2}y}{dz^{2}} = 
\frac{1}{y} \Bigl( \frac{dy}{dz} \Bigr)^{2}
-
\frac{1}{z} \frac{dy}{dz} 
+
\frac{\gamma y^{3}}{4 z^{2}} 
+ \frac{\alpha y^{2}}{4 z^{2}}
+ \frac{\beta}{4z} + \frac{\delta}{4 y}
\]
throgh the following scalings; 
$y = \eta \hspace{+.1em} \lambda, 
 z = \eta^{2} \hspace{+.1em} t, 
 \alpha = -4 \eta \hspace{+.1em} c_{\infty},
 \beta = 4 \eta \hspace{+.1em} c_{0},
 \gamma = 4$ and  
$\delta = -4$.
Therefore, quantities appearing in this paper have a 
homogeneity. They are summarized in 
Appendix \ref{Appendix:Homogeneity}.

\subsection{Transseries solutions}
\label{section:1-parameter solutions}

Our main interest consists in the analysis of 
{\it transseries solutions}; that is, 
formal solutions of $(P_{\rm III'})_{D_{6}}$ of the form
\begin{equation}
\lambda(t,\textbf{c},\eta;\alpha) = 
\sum_{k \ge 0}(\alpha \eta^{-{1}/{2}})^{k} 
\lambda^{(k)}(t,\textbf{c},\eta) 
e^{k \eta \phi},
\label{eq:1-parameter solution}
\end{equation}
where $\alpha$ is a free parameter, 
$\lambda^{(k)}(t, {\bf c},\eta)$ 
($k \ge 0$)
is a formal power series in $\eta^{-1}$ 
of the form
\[
\lambda^{(k)}(t,\textbf{c},\eta) =  
\sum_{\ell \ge 0} \eta^{-\ell}
\lambda^{(k)}_{\ell}(t,\textbf{c}),
\]
and 
$\phi = \phi(t, {\bf c})$ is some function 
defined as follows.
Let $F(\lambda,t,{\bf c})$ be the coefficient of $\eta^{2}$
in the right-hand side of \eqref{eq:P3-D6}:
\begin{equation}
F(\lambda,t,{\bf c}) = 
\frac{{\lambda}^{3}}{t^{2}} - 
\frac{c_{\infty} {\lambda}^{2}}{t^{2}}
+ \frac{c_{0}}{t} - \frac{1}{{\lambda}},
\end{equation}
and $\lambda_{0} = \lambda_{0}(t, {\bf c})$ be 
an algebraic function defined by 
\begin{equation}
F(\lambda_{0},t,{\bf c}) = 
\frac{{\lambda_{0}}^{3}}{t^{2}} - 
\frac{c_{\infty} {\lambda_{0}}^{2}}{t^{2}}
+ \frac{c_{0}}{t} - \frac{1}{{\lambda_{0}}} = 0, 
\label{eq:lambda(0)0}
\end{equation}
which is nothing but the leading term 
$\lambda^{(0)}_{0}(t,{\bf c})$
of the formal power series $\lambda^{(0)}(t,{\bf c},\eta)$
in \eqref{eq:1-parameter solution}. Then the phase function 
$\phi(t,{\bf c})$ is defined by 
\begin{equation}
\phi(t,\textbf{c})=
\int^{t} \sqrt{ \Delta(t,\textbf{c}) } 
\hspace{+.2em} dt ,
\hspace{+1.em}
\Delta(t,\textbf{c}) 
= \frac{\partial F}{\partial \lambda}
\bigl( \lambda_{0}(t,{\bf c}),t,{\bf c} \bigr). 
\label{eq:phi-3}
\end{equation}
Transseries solutions are considered 
in a domain where the real part of $\phi$ is negative; 
i.e., ${\rm exp}(\eta \phi)$ is exponentially small when 
$\eta \rightarrow +\infty$. See \cite{Costin, Iwaki} 
for a construction of such formal solutions.
The (generalized) Borel resummation method 
for transsries are discussed in \cite{Costin}.

The formal expansion of the form 
\eqref{eq:1-parameter solution} 
is called ``{\it instanton-type expansion}'' 
(or ``{\it non-perturbative expansion}'') 
in physical literatures. 
We follow this terminology and call 
the $k$-th formal series $(\alpha \eta^{-1/2})^{k} 
\lambda^{(k)}(t,{\bf c},\eta) e^{k \eta \phi(t)}$ 
``{\it $k$-instanton part}'' of 
\eqref{eq:1-parameter solution} (for $k \ge 0$).  
We note that the $0$-instanton part
$\lambda^{(0)}(t,\textbf{c},\eta)$ is itself 
a formal power series solution of $(P_{\rm III'})_{D_{6}}$ 
called {\it a 0-parameter solution} (\cite{KT WKB Painleve I}).
The coefficients of 0-parameter solution are determined recursively 
once we fix the branch of the algebraic function $\lambda_{0}$. 
For example, the coefficients of first few terms 
are given as follows:
\begin{eqnarray*}
\lambda_{1}^{(0)}(t,\textbf{c}) & = & 0. \\[+.3em]
\lambda_{2}^{(0)}(t,\textbf{c}) & = & 
\frac{1}{\Delta} \hspace{+.1em}
\Bigl( \frac{d^{2}\lambda_{0}}{dt^{2}} 
- \frac{1}{\lambda_{0}} 
\Bigl( \frac{d \lambda_{0}}{dt} \Bigr)^{2}
+ \frac{1}{t} \frac{d \lambda_{0}}{dt} \Bigr). \\[+.2em]
\lambda_{3}^{(0)}(t,\textbf{c}) & = & 0. \\[+.3em]
\lambda_{4}^{(0)}(t,\textbf{c}) & = & 
\frac{1}{\Delta} \hspace{+.1em}
\biggl( \frac{d^{2}\lambda^{(0)}_{2}}{dt^{2}} 
- \frac{2}{\lambda_{0}}
\frac{d \lambda_{0}}{dt}
\frac{d \lambda^{(0)}_{2}}{dt} 
+ \frac{\lambda^{(0)}_{2}}{\lambda_{0}^{2}} 
\Bigl( \frac{d \lambda_{0}}{dt} \Bigr)^{2}
+ \frac{1}{t} \frac{d \lambda^{(0)}_{2}}{dt} \\[+.3em]
&  &  \hspace{+2.em}
- \frac{3 \lambda_{0} {\lambda^{(0)}_{2}}^{2}}{t^{2}}
+ \frac{c_{\infty} {\lambda^{(0)}_{2}}^{2}}{t^{2}}
+ \frac{{\lambda^{(0)}_{2}}^{2}}{\lambda_{0}^{3}}
\biggr). \\ 
 & \vdots & 
\end{eqnarray*}
We can verify that 
$\lambda_{\ell}^{(0)}(t,{\bf c}) = 0$ 
for each odd number $\ell$. The Borel summability 
of 0-parameter solution is discussed in 
\cite{Kamimoto-Koike}.

Next we discuss normalizations of transseries solutions. 
We can easily confirm that, if we denote 
the $1$-instanton part of 
\eqref{eq:1-parameter solution} by 
\[
\tilde{\lambda}^{(1)} = \alpha \eta^{-{1}/{2}} 
\lambda^{(1)}(t,\textbf{c},\eta) e^{\eta \phi}, 
\]
then $\tilde{\lambda}^{(1)}$ satisfies the following 
second order linear differential equation
\begin{equation}
\frac{d^{2}\tilde{\lambda}^{(1)}}{dt^{2}} = 
\Bigl( \frac{2}{\lambda^{(0)}} 
\frac{d \lambda^{(0)}}{dt} 
- \frac{1}{t} \Bigr) 
\frac{d \tilde{\lambda}^{(1)}}{dt} 
+ \eta^{2} \biggl\{ 
\frac{\partial F}{\partial \lambda}
\bigl( \lambda^{(0)},t,{\bf c} \bigr) 
- \eta^{-2}
\Bigl( \frac{1}{\lambda^{(0)}} 
\frac{d \lambda^{(0)}}{dt} 
\Bigr)^{2}
\biggr\} \hspace{+.2em} \tilde{\lambda}^{(1)}, 
\label{eq:Frechet bibun}
\end{equation}
which is a Fr\'echet derivative 
of $(P_{\rm III'})_{D_{6}}$ 
at $\lambda = \lambda^{(0)}$. 
Thus we can take a WKB solution 
\cite[$\S 2$]{KT iwanami} 
of \eqref{eq:Frechet bibun} for the 
$1$-instanton part $\tilde{\lambda}^{(1)}$; that is, 
\begin{equation}
\tilde{\lambda}^{(1)} = \alpha \hspace{+.1em} {\rm exp} 
\biggl( \int^{t} R(t,\textbf{c},\eta) dt \biggr),
\label{eq:tilde lambda (1)}
\end{equation}
where $\alpha$ is a free parameter, and 
\[
R(t,\textbf{c},\eta) = \sum_{\ell \ge -1}
\eta^{-\ell} R_{\ell}(t,\text{c}) = 
\eta R_{-1}(t,{\bf c}) + R_{0}(t,{\bf c}) + 
\eta^{-1} R_{1}(t,{\bf c}) + \cdots
\]
is a formal solution of the following Riccati 
equation associated with 
\eqref{eq:Frechet bibun}:
\begin{equation}
R^{2} + \frac{d R}{d t} = 
\Bigl( \frac{2}{\lambda^{(0)}} 
\frac{d \lambda^{(0)}}{dt} 
- \frac{1}{t} \Bigr) R 
+ \eta^{2} \biggl\{ 
\frac{\partial F}{\partial \lambda}
\bigl( \lambda^{(0)},t,{\bf c} \bigr) 
- \eta^{-2}
\Bigl( \frac{1}{\lambda^{(0)}} 
\frac{d \lambda^{(0)}}{dt} 
\Bigr)^{2}
\biggr\}. 
\label{eq:Riccati equation}
\end{equation}

\begin{rem} \normalfont
\label{Remark:odd part}
We can easily confirm that, 
once we fix the square root   
\begin{equation}
R_{-1}(t,\textbf{c}) = 
\sqrt{ \Delta (t,\textbf{c})},
\end{equation}
then coefficients $R_{\ell}(t,{\bf c})$ 
($\ell \ge 0$) are determined recursively. 
For example, the coefficients of first few terms 
are given by
\begin{eqnarray*}
R_{0}(t,{\bf c}) & = & -\frac{1}{2 R_{-1}}\frac{dR_{-1}}{dt} 
 + \frac{1}{\lambda_{0}} \frac{d\lambda_{0}}{dt} 
 - \frac{1}{2t}. \\[+.5em]
R_{1}(t,{\bf c}) & = & \frac{1}{2R_{-1}} 
 \biggl( - R_{0}^{2} - \frac{dR_{0}}{dt} 
 + \Bigl(\frac{2}{\lambda_{0}}\frac{d\lambda_{0}}{dt} 
 - \frac{1}{t} \Bigr) R_{0}   \\[+.3em]
&   & + \Bigl(  \frac{6\lambda_{0}}{t^{2}}
 - \frac{2c_{\infty}}{t^{2}} 
 - \frac{2}{\lambda_{0}^{3}} \Bigr) \lambda_{2}^{(0)} 
 - \Bigl(\frac{1}{\lambda_{0}} 
 \frac{d \lambda_{0}}{dt} \Bigr)^{2}
 \biggr). \\
& \vdots & 
\end{eqnarray*}
Thus we have two formal solutions $R_{+}$ and $R_{-}$ of 
\eqref{eq:Riccati equation} with 
$R_{\pm}(t,\textbf{c},\eta) = \pm \eta \hspace{+.1em} 
\sqrt{\Delta(t,\textbf{c})} + \cdots$.
We define $R_{\rm odd}$ and $R_{\rm even}$ by
\begin{eqnarray}
R_{\rm odd}(t,\textbf{c},\eta) & = & 
\frac{1}{2} \bigl( R_{+}(t,\textbf{c},\eta) 
- R_{-}(t,\textbf{c},\eta) \bigr), 
\label{eq:R-odd}
 \\[+.3em]
R_{\rm even}(t,\textbf{c},\eta) & = & 
\frac{1}{2} \bigl( R_{+}(t,\textbf{c},\eta) 
+ R_{-}(t,\textbf{c},\eta) \bigr). 
\label{eq:R-even}
\end{eqnarray}
We can easily confirm that all the coefficients of 
$\eta^{2n}$ in $R_{\rm odd}(t,{\bf c},\eta)$ 
vanish for $n \ge 0$, and 
$R_{\rm odd}(t,{\bf c},\eta)$ has the form
\begin{equation}
R_{\rm odd}(t,{\bf c},\eta) = \sum_{n \ge 0}
\eta^{1-2n} R_{2n-1}(t,{\bf c}) = 
\eta R_{-1}(t,{\bf c}) + \eta^{-1} R_{1}(t,{\bf c}) + 
\cdots . 
\label{eq:expression of Rodd}
\end{equation}
Moreover, since  
\begin{equation}
R_{\rm even} = 
- \frac{1}{2} \frac{1}{R_{\rm odd}} 
\frac{d R_{\rm odd}}{dt} + 
\Bigl( \frac{1}{\lambda^{(0)}} 
\frac{d \lambda^{(0)}}{dt} 
- \frac{1}{2t} \Bigr), 
\label{eq:relation of Rodd and Reven}
\end{equation}
holds by \eqref{eq:Riccati equation}, 
the WKB solution $\tilde{\lambda}^{(1)}$ 
of \eqref{eq:Frechet bibun} can be written as 
\begin{eqnarray}
\tilde{\lambda}^{(1)} & = & 
\alpha \hspace{+.1em} 
\frac{{\lambda^{(0)}(t,\textbf{c},\eta)}}
{\sqrt{ t \hspace{+.1em} R_{\rm odd}(t,\textbf{c},\eta)}} 
\hspace{+.1em} {\rm exp} 
\biggl( \int^{t} 
R_{\rm odd}(t,\textbf{c},\eta) dt \biggr) 
\label{eq:first part of 1-parameter solution} \\[+.5em]
& = &
\alpha \hspace{+.1em} \eta^{-{1}/{2}} \bigl( 
\lambda^{(1)}_{0}(t,\textbf{c}) + 
\eta^{-1} \lambda^{(1)}_{1}(t,\textbf{c}) + 
\eta^{-2} \lambda^{(1)}_{2}(t,\textbf{c}) + \cdots
\bigr) e^{\eta \phi}. 
\label{eq:tilde lambda (1) expansion}
\end{eqnarray}
We mainly use the expression 
\eqref{eq:first part of 1-parameter solution}
rather than \eqref{eq:tilde lambda (1) expansion}
for the convenience of descriptions of normalizations.
\end{rem}

We note that, once the normalization 
of the $1$-instanton part $\tilde{\lambda}^{(1)}$ 
(i.e.\ the end point and the path of the 
integral of $R_{\rm odd}$ 
in \eqref{eq:first part of 1-parameter solution}) 
is fixed, then the formal power series 
$\lambda^{(k)}(t,{\bf c},\eta)$ ($k \ge 2$) 
are determined uniquely in a recursive manner. 
Then the all coefficients $\lambda^{(k)}_{\ell}(t,{\bf c})$ 
appears in \eqref{eq:1-parameter solution} are holomorphic 
in $t$ on the universal covering of 
\begin{equation}
\Omega_{D_{6}} = 
{\mathbb C}^{\times}/\{ \text{turning points} \}. 
\label{eq:the domain of P3D6}
\end{equation}
Here turning points of $(P_{\rm III'})_{D_{6}}$ 
are defined in Definition \ref{def:P-Stokes geometry} below.
In Section \ref{section:P-Voros coefficients}
we will introduce some special normalizations of 
\eqref{eq:first part of 1-parameter solution}, 
and analyze parametric Stokes phenomena for those 
transseries solutions in 
Section \ref{section:connection formulas}. 

\subsection{Stokes geometry}
\label{section:P-Stokes geometry}

Now we recall the definitions of turning points and 
Stokes curves of $(P_{\rm III'})_{D_{6}}$. 

\begin{defi} \normalfont
{\cite[Definition 4.5]{KT iwanami}} 
\label{def:P-Stokes geometry} 
(i) A point $t = \tau$ is called 
{\it a turning point}
if $\tau \ne 0$ and it is a zero of the function 
$\Delta(t,\textbf{c})$.  \vspace{+.3em}

\noindent
(ii) For a turning point $t = \tau$, 
real one-dimensional curves defined by 
\begin{equation}
{\rm Im} \int_{\tau}^{t} 
\sqrt{\Delta(t,\textbf{c})} \hspace{+.2em} dt = 0
\end{equation}
are called {\it Stokes curves} (emanating from $\tau$).
\end{defi}

In Definition \ref{def:P-Stokes geometry} the branch of 
$\lambda_{0}$ is fixed. But later we consider all branches 
of $\lambda_{0}$ and lift turning points and 
Stokes curves onto the Riemann surface of $\lambda_{0}$, 
which is a four-fold covering of ${\mathbb C}^{\times}$ 
branching at turning points.
See Remark \ref{Remark:uniformization} for the lift.

\begin{rem} \label{Remark:R-turning points} \normalfont
Turning points are, by definition, 
zeros of the discriminant 
of the following algebraic equation for $\lambda$: 
\begin{equation}
(t^{2} \lambda F(\lambda,t,{\bf c}) = )\hspace{+.2em}
\lambda^{4} - c_{\infty} \lambda^{3}
+ c_{0} t \lambda - t^{2} = 0. 
\label{eq:lambda0}
\end{equation}
The discriminant ${Disc}(t,{\bf c})$ of 
the above algebraic equation is given by 
\begin{equation}
{Disc}(t,{\bf c}) = 
t^{3} ( -256 t^{3} + 192 t^{2} c_{\infty} c_{0}
+ 6 t c_{\infty}^{2} c_{0}^{2} - 27 t c_{\infty}^{4} 
- 27 t c_{0}^{4} + 4 c_{\infty}^{3} c_{0}^{3}).
\end{equation}
Moreover, the discriminant of the cubic equation
$Disc(t,{\bf c})/t^{3} = 0$ for $t$ is factorized as  
\begin{equation}
-20155392 (c_{\infty}^{2} - c_{0}^{2})^{4}
(c_{\infty}^{2} + c_{0}^{2})^{2}. 
\end{equation}
Therefore, under the assumptions \eqref{eq:THE condition}, 
there are three turning points on 
${\mathbb C}^{\times}$, and all of them are simple 
(in the sense of \cite[Definition 4.5]{KT iwanami}). 
We always assume \eqref{eq:THE condition} in this paper.
\end{rem}

Unfortunately, Definition \ref{def:P-Stokes geometry} 
is not enough to describe the 
``complete'' Stokes geometry 
because the singular point $t = 0$ of 
$(P_{\rm III'})_{D_{6}}$ 
may play the same role of turning point 
(\cite{Takei-Wakako, Takei turning point problem, Kamimoto-Koike}). 
We know that there are following four asymptotic 
behaviors of $\lambda_{0}$ as $t \rightarrow 0$:
\begin{eqnarray}
\lambda_{0}(t,{\bf c}) 
& = & 
c_{\infty} + O(t), 
\label{eq:double-pole type infty} \\[+.5em]
\displaystyle \lambda_{0}(t,{\bf c}) & = & 
{t}/{c_{0}} + O(t^{2}), 
\label{eq:double-pole type zero} \\[+.5em]
\displaystyle \lambda_{0}(t,{\bf c}) & = & 
\pm \sqrt{{c_{0}}/{c_{\infty}}} 
\hspace{+.2em} t^{1/2} (1 + O(t^{1/2})). 
\label{eq:simple-pole type}
\end{eqnarray}
Then, as shown in \cite{Takei-Wakako, Kamimoto-Koike}, 
for the branch \eqref{eq:simple-pole type} 
of $\lambda_{0}$, the singular point $t = 0$ 
plays the same role as turning points; 
that is, there exist a Stokes curve 
emanating from $t = 0$ on which transseries solutions 
may not be Borel summable. 
(A similar phenomenon occurs for WKB solutions of 
Schr\"odinger equations whose potential function 
has a simple-pole \cite{Koike}.)
Thus we define the following:

\begin{defi} \normalfont
{\cite{Takei-Wakako, Kamimoto-Koike}} 
\label{def:P-Stokes geometry 2} 
(i) The singular point $t=0$ for the branch 
\eqref{eq:simple-pole type} of $\lambda_{0}$ 
is called {\it the simple-pole} of $(P_{\rm III'})_{D_{6}}$. 
(See \eqref{eq:quadratic differential in u-plane} below 
for the reason why we call \eqref{eq:simple-pole type} 
simple-pole.) The simple-pole is denoted by $\tau_{sp}$.

\noindent
(ii) For the above branch of $\lambda_{0}$, 
the real one-dimensional curve defined by 
\begin{equation}
{\rm Im} \int_{\tau_{sp}}^{t} 
\sqrt{\Delta(t,\textbf{c})} \hspace{+.2em} dt = 0
\end{equation}
is also called {\it a Stokes curve} (emanating 
from the simple-pole $t = \tau_{sp}$).
\end{defi}

  \begin{figure}[h]
  \begin{minipage}{0.5\hsize}
  \begin{center}
  \includegraphics[width=50mm, clip]
  {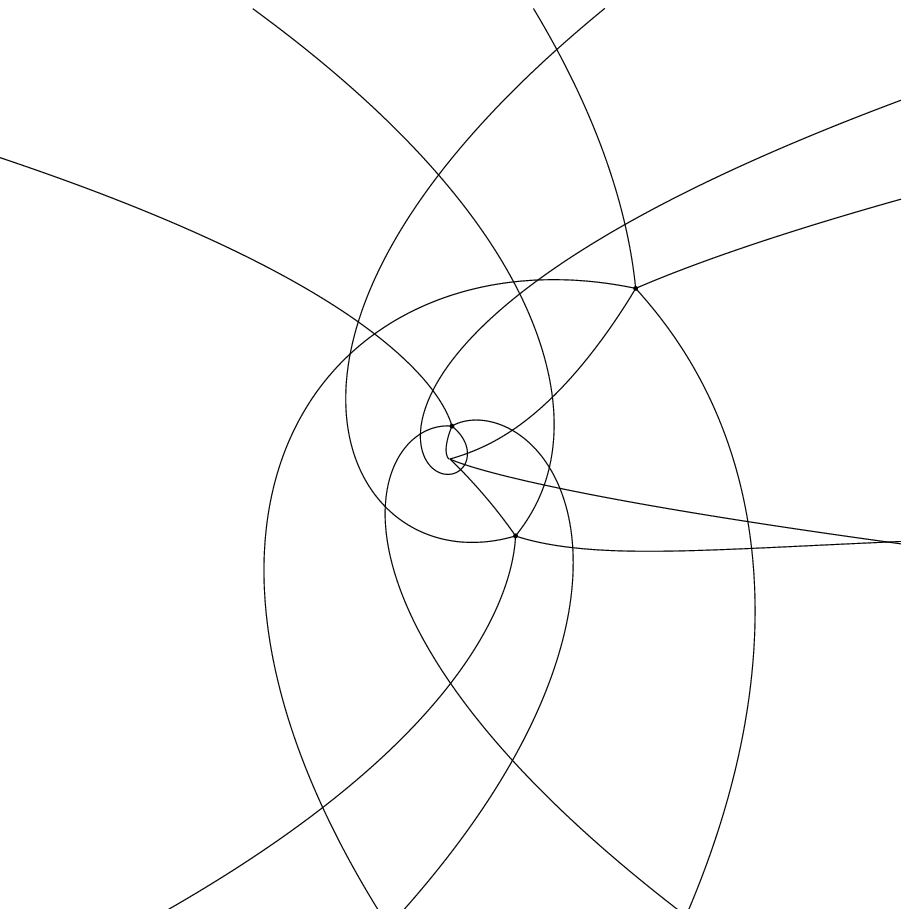}    
  \end{center} 
  \caption{\small{$\textbf{c} = (2+i,3)$.}}
  \label{fig:P3D6-generic}
  \end{minipage} \hspace{+.3em}
  \begin{minipage}{0.5\hsize}
  \begin{center}
  \includegraphics[width=50mm, clip]
  {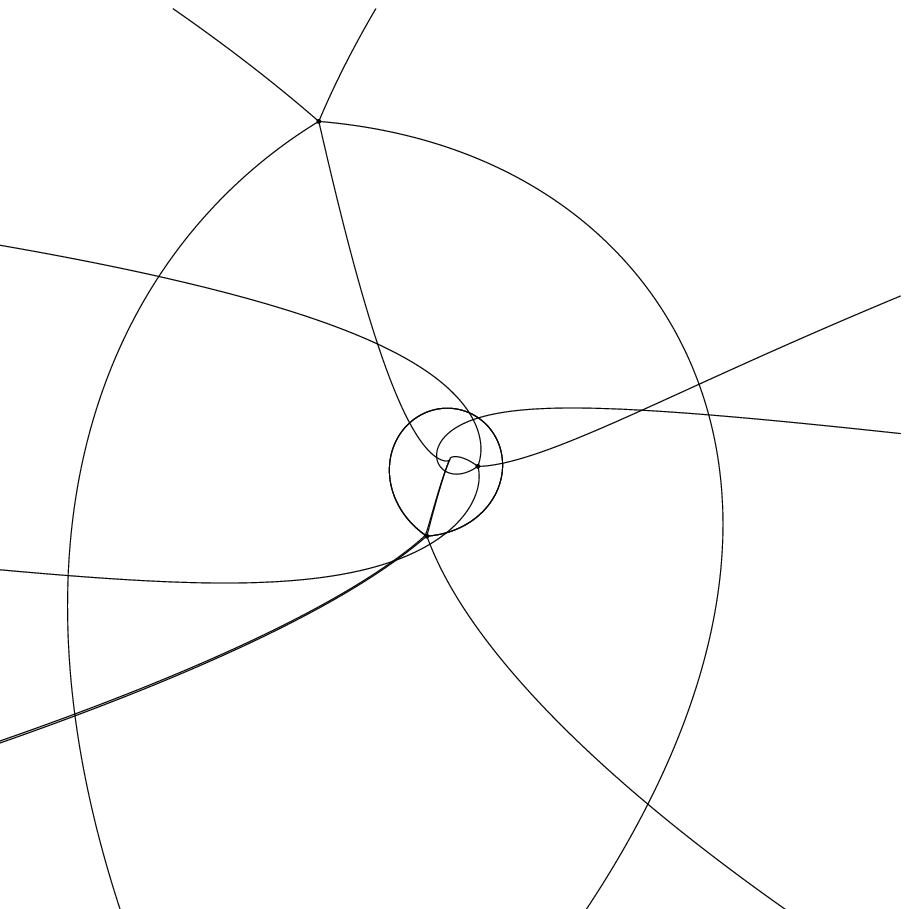}    
  \end{center} 
  \caption{\small{$\textbf{c} = (2+i,3i)$.}}
  \label{fig:P3D6-loop}
  \end{minipage} \hspace{+.3em}
  \end{figure}
  \begin{figure}[h]
  \begin{minipage}{0.31\hsize}
  \begin{center}
  \includegraphics[width=52mm, clip]
  {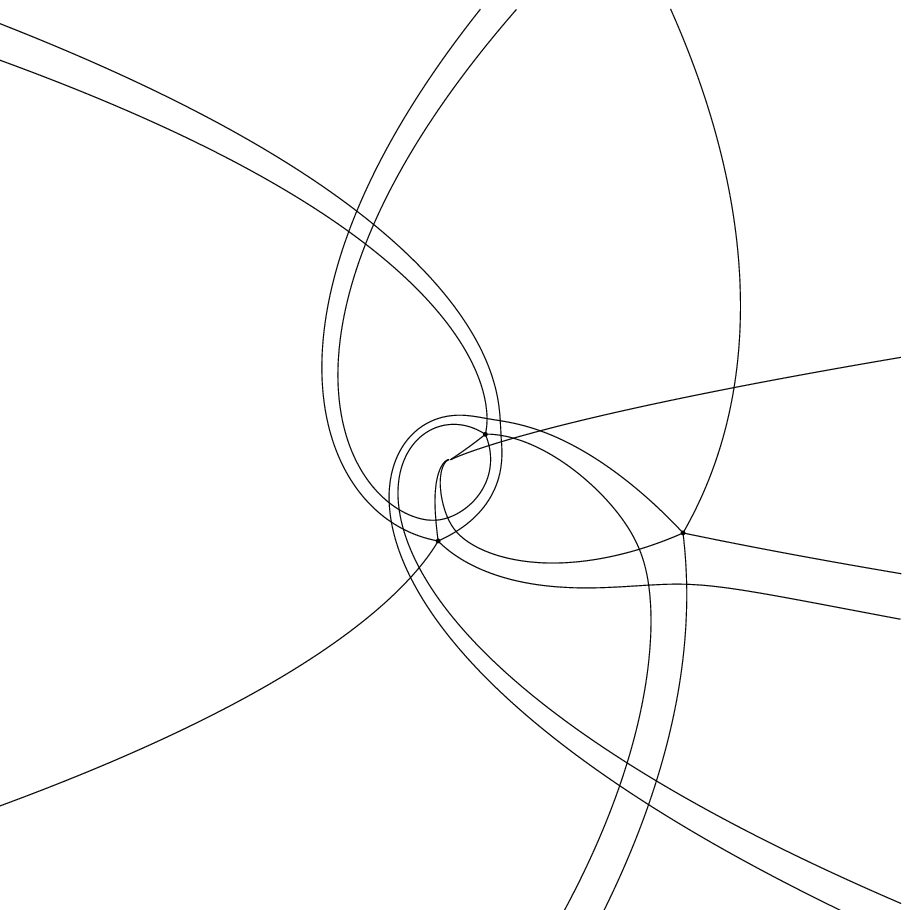}    
  \end{center} 
  \caption{\small{$\textbf{c} = (1.9,2-i)$.}}
  \label{fig:PIII',-epsilon}
  \end{minipage} \hspace{+.3em}
  \begin{minipage}{0.31\hsize}
  \begin{center}
  \includegraphics[width=52mm,  clip]
  {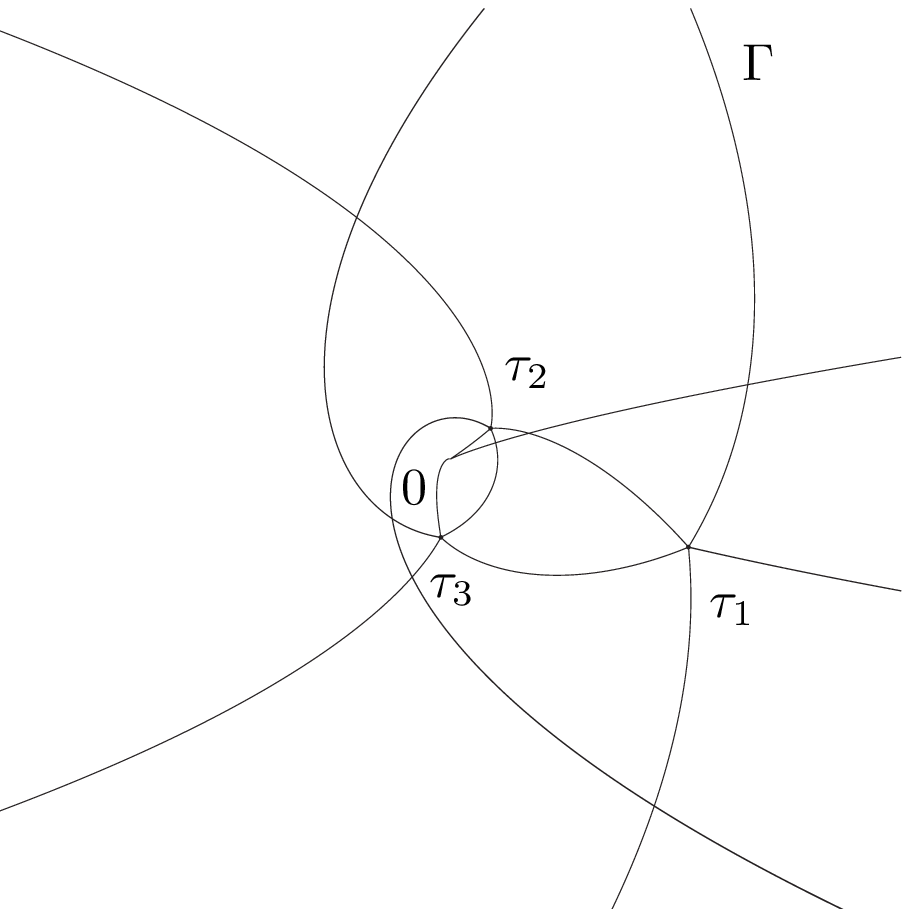}  
  \end{center} 
  \caption{\small{$\textbf{c} = (2,2-i)$.}}
  \label{fig:PIII',0}
  \end{minipage} \hspace{+.3em}
  \begin{minipage}{0.31\hsize}
  \begin{center}
  \includegraphics[width=52mm, clip]
  {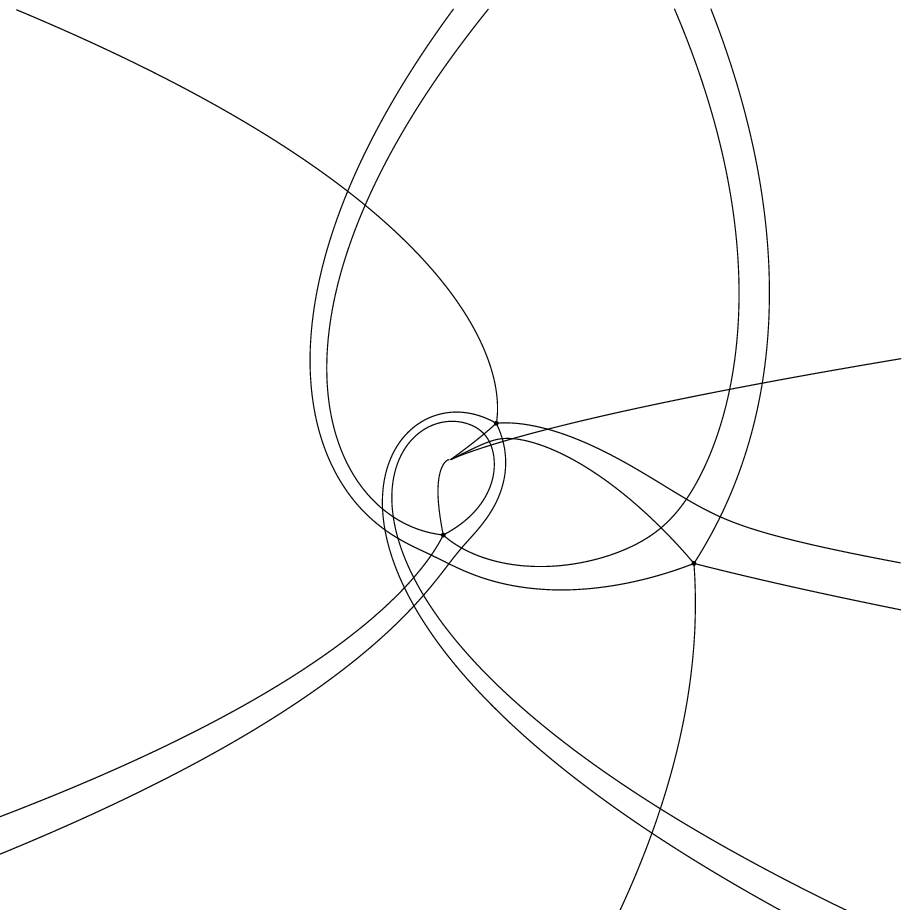}  
  \end{center} 
  \caption{\small{$\textbf{c} = (2.1,2-i)$}.}
  \label{fig:PIII',+epsilon}
  \end{minipage}
  \end{figure}

Figures \ref{fig:P3D6-generic} $\sim$ 
\ref{fig:PIII',+epsilon} describe the Stokes curves of 
$(P_{\rm III'})_{D_{6}}$ on the $t$-plane for several values 
of ${\bf c}$. Note that the quadratic differential 
$\Delta(t,{\bf c}) \hspace{+.1em} dt^{2}$ 
which defines the Stokes curves behaves as 
\begin{eqnarray}
\Delta(t,{\bf c}) \hspace{+.1em} dt^{2} 
& = & 
C_{\tau} \hspace{+.1em} (t-\tau)^{1/2} 
\bigl( 1 + O((t-\tau)^{1/2}) \bigr) dt^{2} 
\hspace{+1.em} \text{as $t \rightarrow \tau 
\hspace{+.2em} (\ne \tau_{sp}$)}, \\
\Delta(t,{\bf c}) \hspace{+.1em} dt^{2}
& = &
C_{{sp}} \hspace{+.1em} t^{-3/2}
\bigl( 1 + O(t^{1/2}) \bigr) dt^{2} 
\hspace{+1.em} \text{as $t \rightarrow \tau_{sp}$}, 
\end{eqnarray}
with some constants $C_{\tau}$ and $C_{{sp}}$ which are 
non-zero under the assumption \eqref{eq:THE condition}.
Hence, five Stokes curves emanate from 
each turning points, and one Stokes curve 
emanates from the simple-pole. 
Connection problems on Stokes curves of Painlev\'e equations 
in terms of the exact WKB analysis are discussed 
in \cite{KT iwanami, Takei Painleve, 
Takei turning point problem, Takei-Wakako}.

In Figure \ref{fig:PIII'-triangle-intro}, 
\ref{fig:PIII'-loop-intro}, 
\ref{fig:P3D6-loop} and 
\ref{fig:PIII',0} we can observe 
that there exists a ``{bounded Stokes curve}'', 
which connect two different turning points,
or form a loop around the singular point $t = 0$. 
We call such situation 
``{\it degeneration of the Stokes geometry}'', 
and it is closely related to parametric Stokes phenomena 
for transseries solutions, 
as we pointed in Section \ref{section:Introduction}. 
Connection problems for the parametric Stokes phenomena 
are discussed in the subsequent sections. 

  \begin{figure}[h]
  \begin{minipage}{0.5\hsize}
  \begin{center}
  \includegraphics[width=60mm]
  {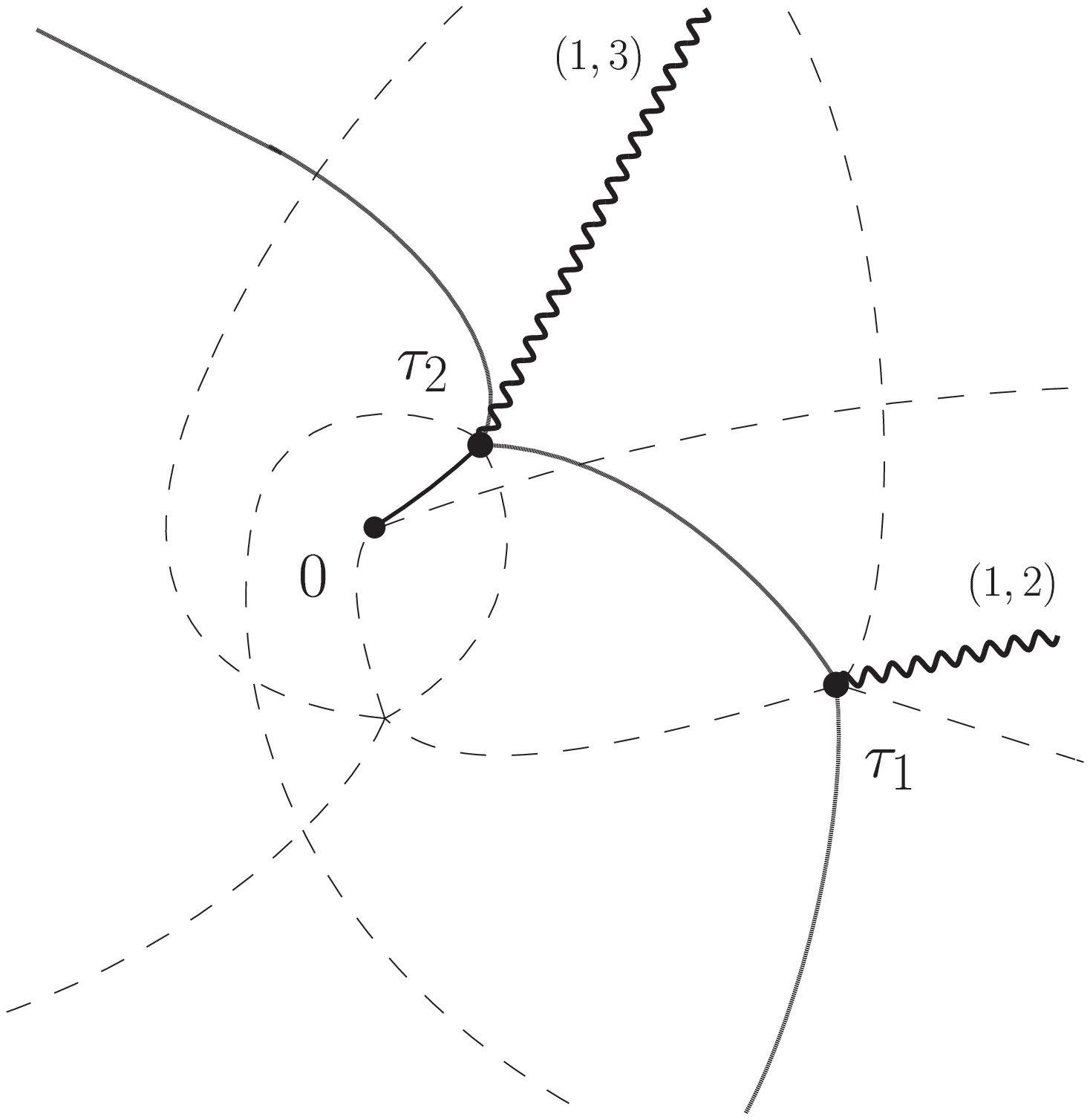}
  \end{center}
  \caption{\small{Sheet 1.}}
  \label{fig:P3D6Stokes-sheet-1}
  \end{minipage}
  \begin{minipage}{0.5\hsize}
  \begin{center}
  \includegraphics[width=60mm]
  {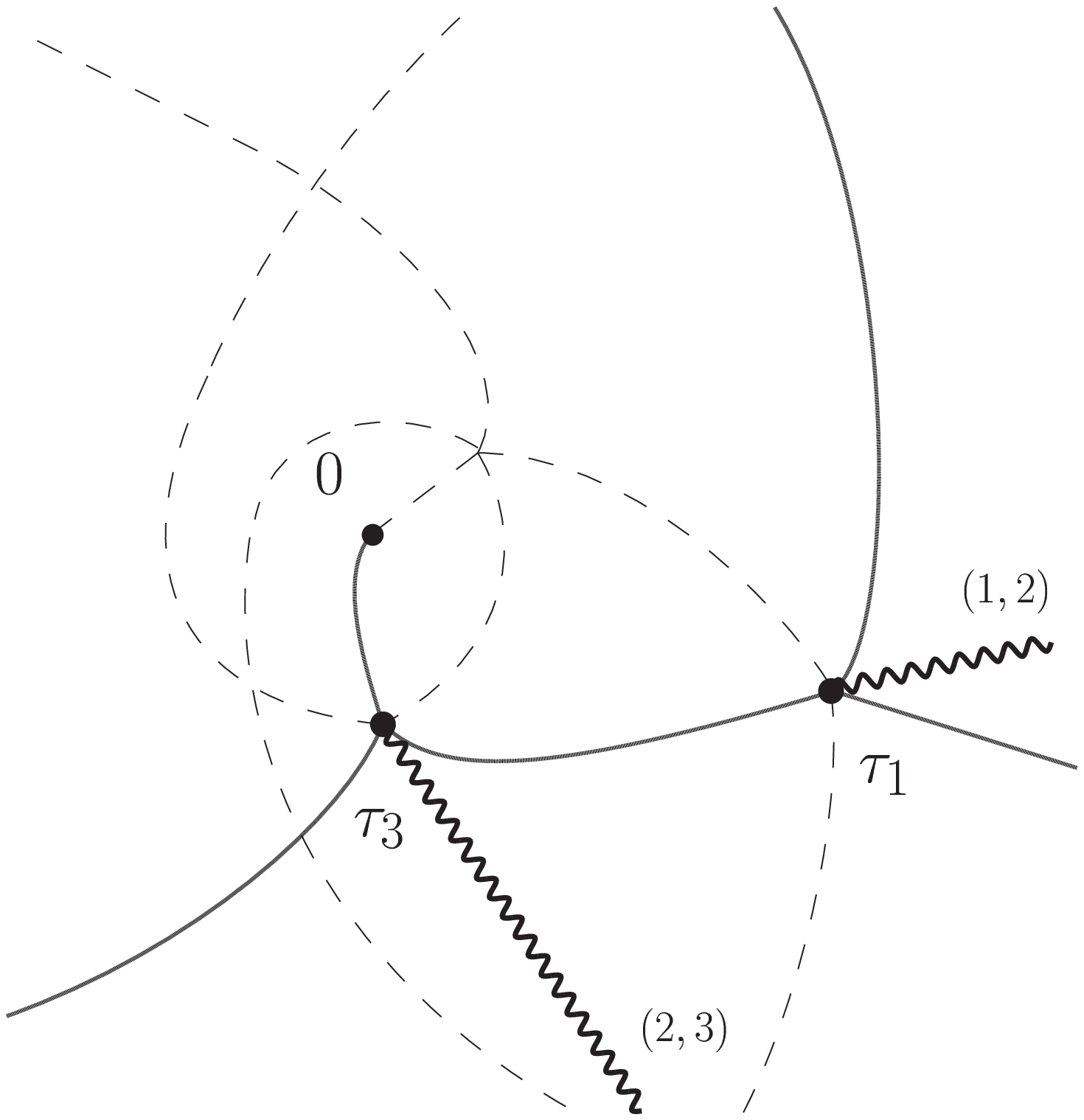}
  \end{center}
  \caption{\small{Sheet 2.}}
  \label{fig:P3D6Stokes-sheet-2}
  \end{minipage}
  \end{figure}
  \begin{figure}[h]
  \begin{minipage}{0.5\hsize}
  \begin{center}
  \includegraphics[width=60mm]
  {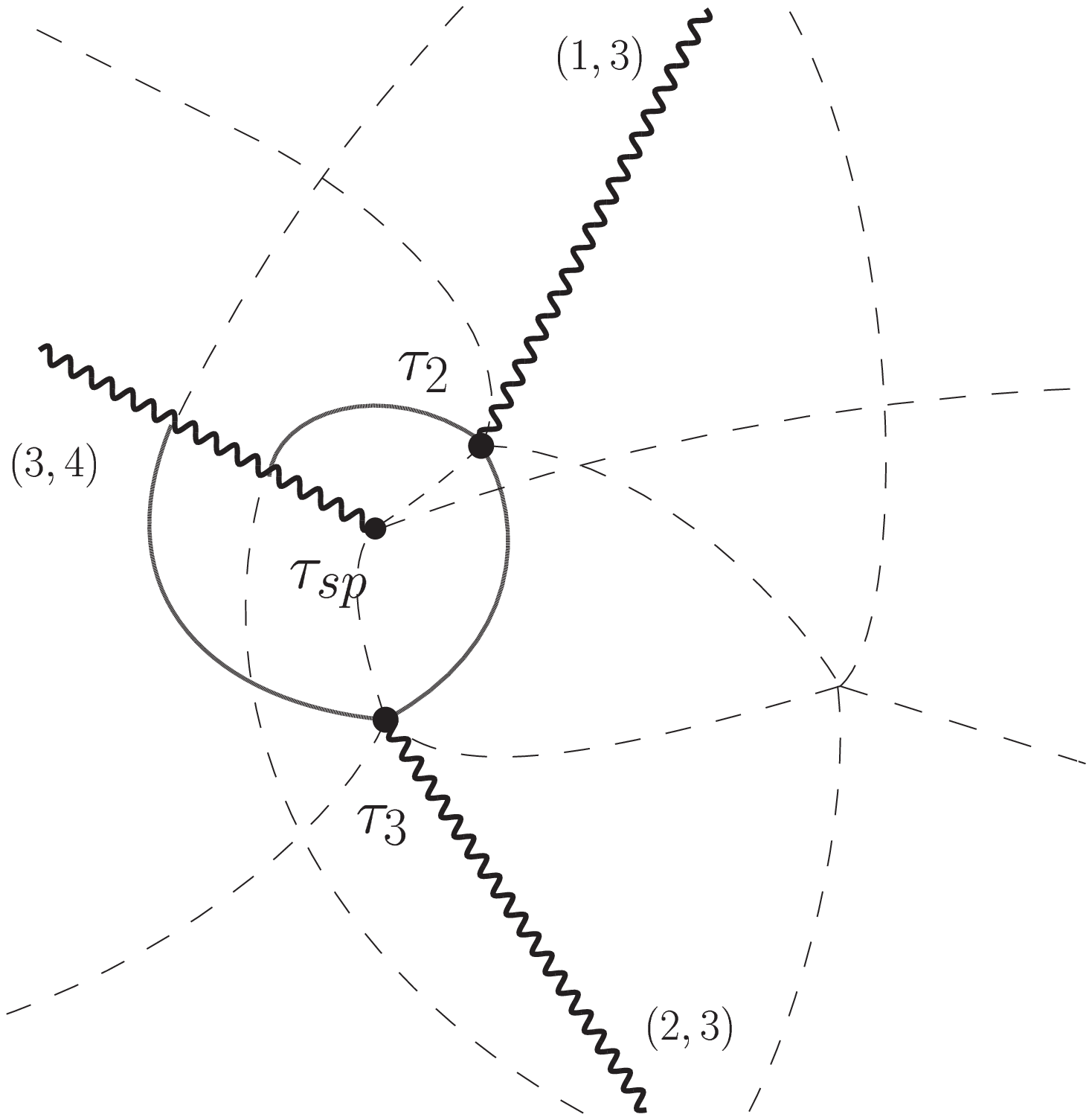}
  \end{center}
  \caption{\small{Sheet 3.}}
  \label{fig:P3D6Stokes-sheet-3}
  \end{minipage}
  \begin{minipage}{0.5\hsize}
  \begin{center}
  \includegraphics[width=60mm]
  {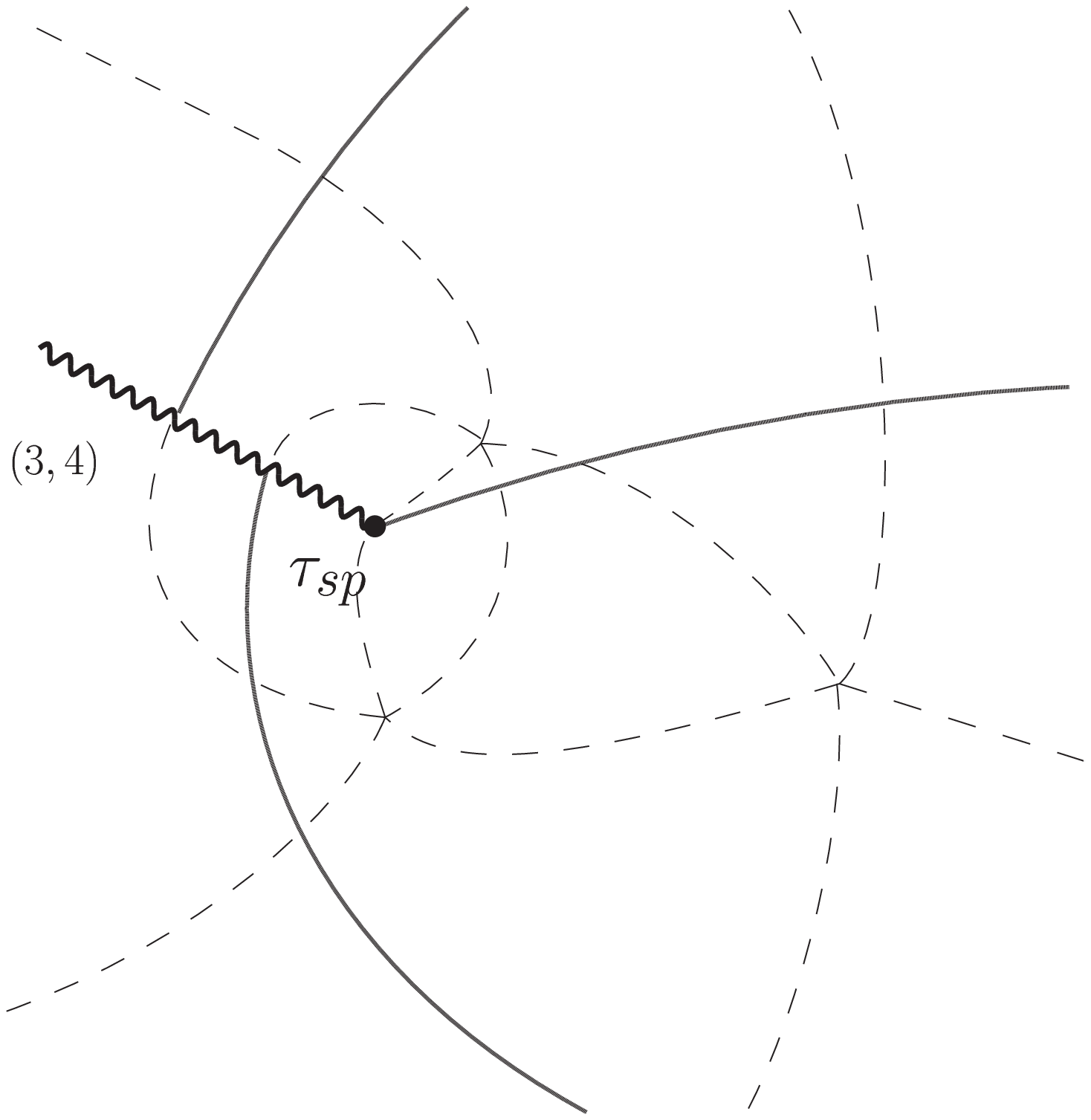}
  \end{center}
  \caption{\small{Sheet 4.}}
  \label{fig:P3D6Stokes-sheet-4}
  \end{minipage}
  \end{figure}

\begin{rem} \label{Remark:uniformization}
\normalfont
Since $\lambda_{0}(t,{\bf c})$ is multivalued function of $t$, 
we should consider the lift of Stokes geometry  
onto the Riemann surface of $\lambda_{0}$ 
(i.e., the Riemann surface defined by \eqref{eq:lambda(0)0}).
Since $\lambda_{0}$ satisfies the algebraic equation 
\eqref{eq:lambda0} of degree 4, we need four sheets 
(Sheet 1 $\sim$ Sheet 4) to consider the lift.
For example, Figure \ref{fig:P3D6Stokes-sheet-1} $\sim$
Figure \ref{fig:P3D6Stokes-sheet-4} describe the lift 
of Stokes geometry when ${\bf c} = (2, 2-i)$.
The wiggly line of the type $(j, k)$ ($1 \le j,k \le 4$),  
solid lines and dotted lines in these figures 
represent the branch cut for between the Sheet $j$ and 
Sheet $k$, Stokes curves on the sheet under consideration
and Stokes curves on the other sheets, respectively. 
The origin of the type \eqref{eq:double-pole type infty} 
and \eqref{eq:double-pole type zero} correspond to 
the origin of the Sheet 1 and Sheet 2 respectively. 
Thus the Riemann surface of $\lambda_{0}$ has genus $0$.
\end{rem}

\begin{figure}[h]
  \begin{minipage}{0.31\hsize}
  \begin{center}
  \includegraphics[width=46mm]
  {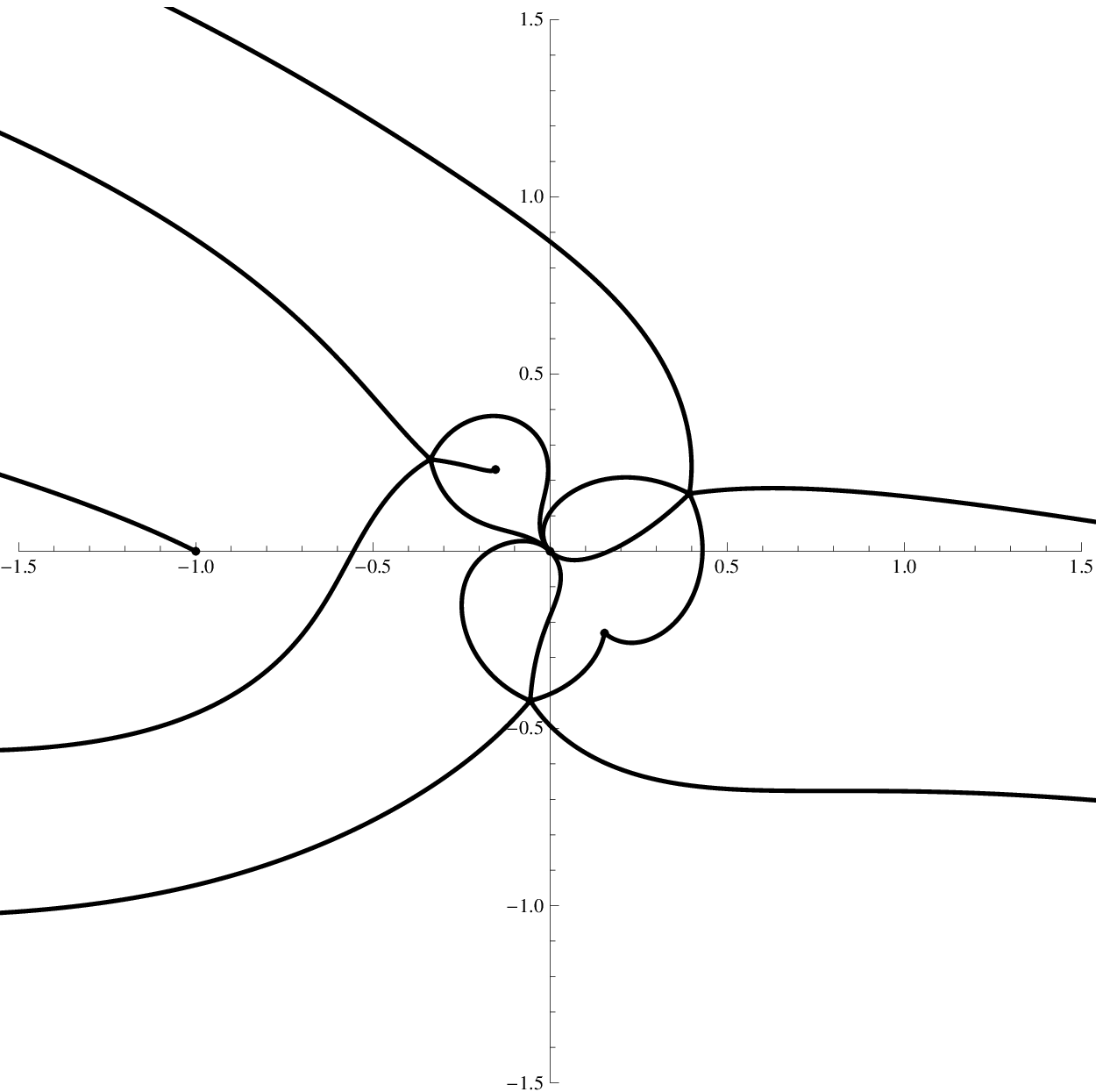}
  \end{center}
  \caption{\small{${\bf c}=(2+i, 3)$.}}
  \label{fig:P3D6-u-plane-generic}
  \end{minipage} 
\begin{minipage}{0.31\hsize}
\begin{center}
\includegraphics[width=46mm]{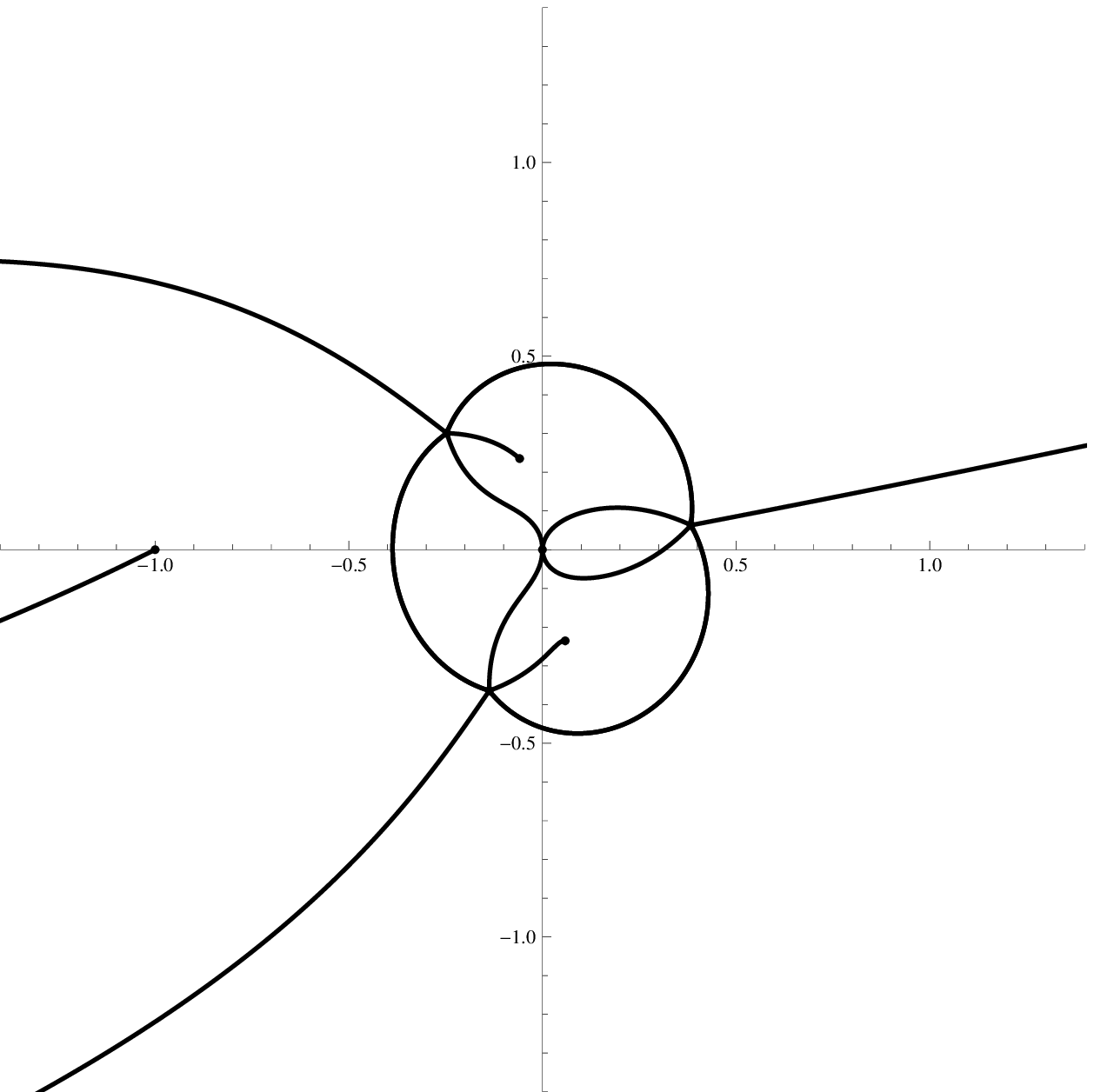}
\end{center}
\caption{\small{${\bf c} = (2, 2-i)$. \hspace{+.2em}}}
\label{fig:P3D6-u-plane-triangle-0}
\end{minipage} \hspace{+.3em}
  \begin{minipage}{0.31\hsize}
  \begin{center}
  \includegraphics[width=46mm]
  {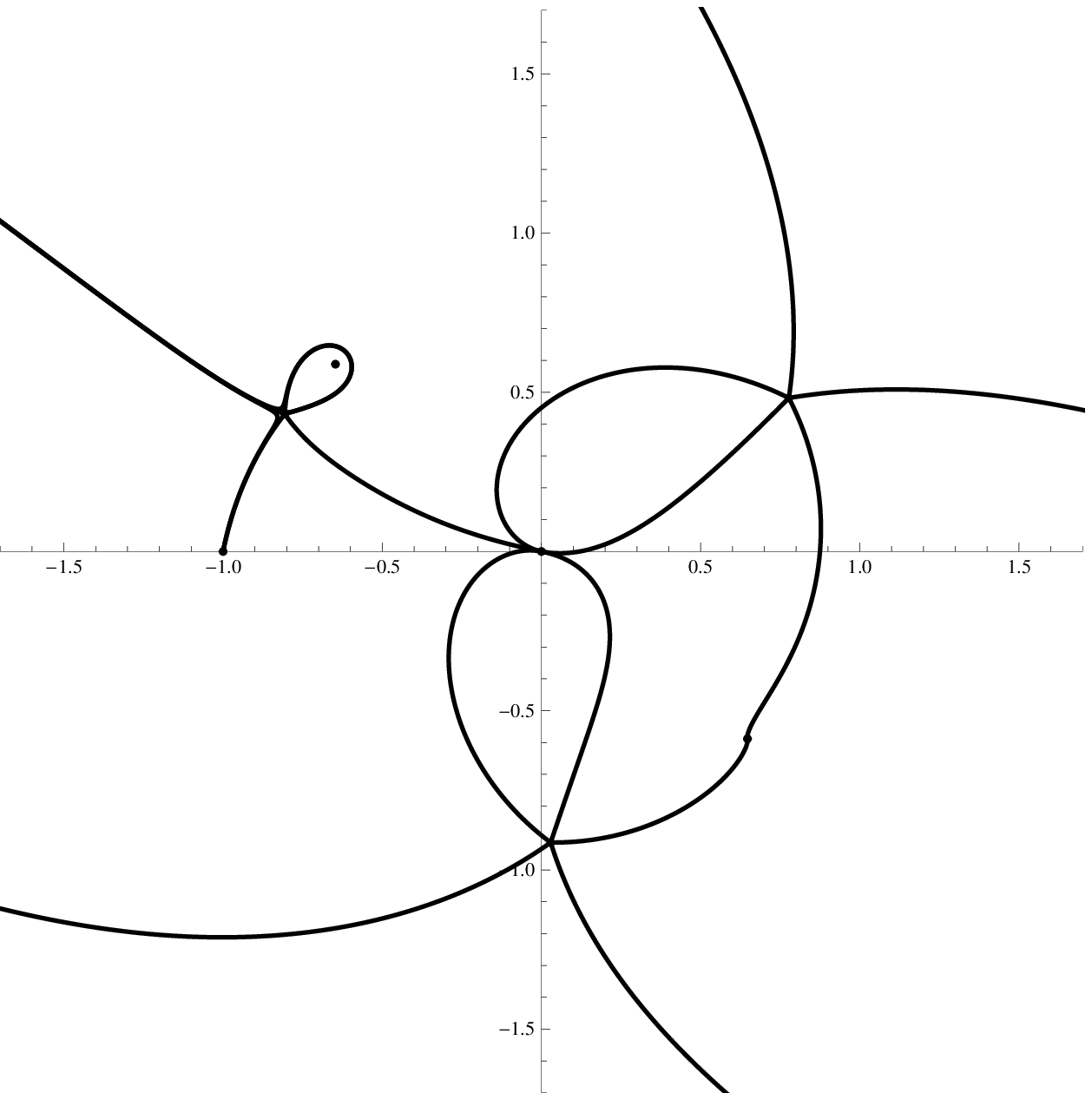}
  \end{center}
  \caption{\small{${\bf c} = (5+i, 2i)$}.}
  \label{fig:P3D6-u-plane-loop}
  \end{minipage}
\end{figure}

It is convenient to change the variable from $t$ 
to $u$ which is given by 
\begin{equation}
u = \frac{1-\mu_{0}}{\mu_{0}}, 
\label{eq:u-coordinate}
\end{equation}
where $\mu_{0} = \mu_{0}(t,{\bf c})$ is given by \eqref{eq:mu0} 
in Section \ref{section:proof of the Main theorems}. 
Both $t$ and $\lambda_{0}(t,{\bf c})$ can be written 
in terms of $u$ as
\begin{eqnarray*}
\lambda_{0} & = & \frac{u + 1}{4u} 
\bigl( (c_{\infty}+c_{0})u - (c_{\infty}-c_{0})  \bigr), 
\\[+.3em] 
t & = & \frac{(u+1)^{2}}{16 u^{2}}
\bigl( (c_{\infty}+c_{0})^{2} u^{2} 
- (c_{\infty}-c_{0})^{2}  \bigr).
\end{eqnarray*}
Stokes geometry can be described without any intersections 
on the $u$-plane. In fact, the quadratic differential 
$\Delta(t,{\bf c}) \hspace{+.1em} dt^{2}$ 
which determine the Stokes geometry is written as 
\[
\Delta(t,{\bf c}) \hspace{+.1em} dt^{2} = 
q(u,{\bf c}) \hspace{+.1em} du^{2},
\]
\begin{equation}
q(u,{\bf c}) = 
\frac{\bigl\{ (c_{\infty} + c_{0})^{2} u^{3} 
+ (c_{\infty} - c_{0})^{2} \bigr\}^{3}}
{(u+1) u^{4} \bigl( 
(c_{\infty} + c_{0})^{2} u^{2} - (c_{\infty} - c_{0})^{2} 
\bigr)^{2} }.
\label{eq:quadratic differential in u-plane}
\end{equation}
The turning points correspond to zeros 
\[
\frac{(c_{\infty} - c_{0})^{2/3}}
{(c_{\infty} + c_{0})^{2/3}} \hspace{+.1em}
(-1)^{1/3}  \omega^{j} \hspace{+1.em}
(\omega = e^{2\pi i/3}, \hspace{+.3em} j=0,1,2)
\] 
of \eqref{eq:quadratic differential in u-plane}, 
and the simple-pole $t = \tau_{sp}$ of \eqref{eq:simple-pole type} 
corresponds to $u = -1$; i.e., the simple-pole of 
\eqref{eq:quadratic differential in u-plane}.
This is the reason why we call 
\eqref{eq:simple-pole type} simple-pole. 
Similarly, we will call the singular points $t = 0$ of 
\eqref{eq:double-pole type infty} 
and \eqref{eq:double-pole type zero} are double-poles 
since they correspond to double-poles 
$u = (c_{\infty}-c_{0})/(c_{\infty}+c_{0})$ 
and $u = -(c_{\infty}-c_{0})/(c_{\infty}+c_{0})$ 
of \eqref{eq:quadratic differential in u-plane}
respectively.
Moreover, \eqref{eq:u-coordinate} gives a coordinate 
on the Riemann surface of $\lambda_{0}$ 
because 
\[
\frac{du}{dt} = \frac{8 u^{3}}{1+u} \hspace{+.2em}
\frac{1}{(c_{\infty}+c_{0})^{2} u^{3} + 
(c_{\infty} - c_{0})^{2}}
\]
never vanish when $t \in \Omega_{D_{6}}$. 
Figure \ref{fig:P3D6-u-plane-generic} 
$\sim$ \ref{fig:P3D6-u-plane-loop} 
are the Stokes geometry on the $u$-plane 
for several values of ${\bf c}$. 
In Appendix \ref{Appendix:examples of P-Stokes geometry} 
we show more examples of the Stokes geometry on the $u$-plane. 
The residues of the 1-form 
$\sqrt{q(u,{\bf c})} \hspace{+.1em} du$
at singular points are summarized as follows
(the sign $\pm$ depends on the choice of the square root
of $q(u,{\bf c})$):
\begin{eqnarray}
{\rm Res}_{u = \infty} \sqrt{q(u,{\bf c})} \hspace{+.1em} du 
& = & \pm (c_{\infty} + c_{0})/2, \\[+.3em]
{\rm Res}_{u = 0} \sqrt{q(u,{\bf c})} \hspace{+.1em} du 
& = & \pm (c_{\infty} - c_{0})/2, \\[+.3em]
{\rm Res}_{u = (c_{\infty}-c_{0})/(c_{\infty}+c_{0})} 
\sqrt{q(u,{\bf c})} \hspace{+.1em} du 
& = & \pm c_{\infty}, 
\label{eq:residue at double-pole infty} \\[+.3em]
{\rm Res}_{u = - (c_{\infty}-c_{0})/(c_{\infty}+c_{0})} 
\sqrt{q(u,{\bf c})} \hspace{+.1em} du  
& = & \pm c_{0} .
\label{eq:residue at double-pole zero}
\end{eqnarray}

In the figures above and in Appendix 
\ref{Appendix:examples of P-Stokes geometry}
we observe some interesting characteristic 
features of the Stokes geometry of $(P_{\rm III'})_{D_{6}}$. 
We find two kinds of degenerations. The first one is 
``{\it triangle-type degeneration}''; that is, there are 
three pairs of turning points connected by Stokes 
curves simultaneously. Figure \ref{fig:P3D6-u-plane-triangle-0} 
shows an example of the triangle-type degeneration. 
This kind of degeneration is also observed 
in the case of the second, forth and sixth Painlev\'e equation 
(see \cite{Iwaki} for the second Painlev\'e equation). 
The second one is ``{\it loop-type degeneration}''; that is, 
a Stokes curve form a loop around the double-pole type 
singular point t = 0 and, at the same time, the turning point 
which is the end point of the loop and the simple-pole 
are connected by a Stokes curve. 
An example the loop-type degeneration is shown in 
Figure \ref{fig:P3D6-u-plane-loop}. 
It is known that appearances of these loops 
are caused by the fact that the residue of the 1-form 
$\sqrt{\Delta(t,{\bf c})} \hspace{+.2em} dt$ at 
the double-pole inside the loop takes a pure imaginary 
number \cite[$\S$7]{Strebel}. 
See \eqref{eq:residue at double-pole infty} and 
\eqref{eq:residue at double-pole zero}.
The same kind of degeneration is also observed for 
the degenerate third Painlev\'e equation of the type $D_{7}$:
\begin{equation}
\frac{d^{2}\lambda}{dt^{2}} = 
\frac{1}{\lambda} \Bigl( \frac{d \lambda}{dt} \Bigr)^{2}
-
\frac{1}{t} \frac{d \lambda}{dt} 
+
\eta^{2} \Bigl( - \frac{2 \lambda^{2}}{t^{2}}
+ \frac{c}{t} - \frac{1}{\lambda}
\Bigr), \label{eq:P3-D7}
\end{equation}
when $c$ is pure imaginary 
(see Figure \ref{fig:P3D7,0} in 
Appendix \ref{Appendix:P3-D7}).
Intriguingly enough, as far as we checked, 
only these two kinds of degeneration can be observed. 
(See Appendix A.) We expect that these are common 
characteristic features of Stokes geomertries of 
all Painlev\'e equations. 
But we have not understood the mechanism of 
these phenomena yet.

\section{Voros coefficients of $(P_{\rm III'})_{D_{6}}$}
\label{section:P-Voros coefficients}

In this section we compute  
{Voros coefficients of $(P_{\rm III'})_{D_{6}}$}. 
Those are formal power series in $\eta^{-1}$ 
which describe differences of normalizations of 
transseries solutions. 
Voros coefficients play a essential role 
when we discuss the parametric Stokes phenomena 
in Section \ref{section:connection formulas}.

\subsection{Voros coefficients $W_{\infty}$}
\label{subsection:P-Voros coefficients}

We recall the definition of Voros coefficients 
of $(P_{\rm III'})_{D_{6}}$ 
(cf. \cite{Iwaki, Iwaki-Bessatsu}).

\begin{defi} \normalfont
For a path $\Gamma(\tau,\infty)$ from a turning point 
(or the simple-pole) $\tau$ to $\infty$, 
{\it the Voros coefficient
for the path $\Gamma(\tau,\infty)$} is 
a formal power series in $\eta^{-1}$ defined by 
the integral 
\begin{equation}
W_{\infty}(\textbf{c},\eta) 
= \int_{\Gamma(\tau,\infty)} 
\bigl( R_{\rm odd}(t,\textbf{c},\eta) - 
\eta R_{-1}(t,\textbf{c}) \bigr) \hspace{+.1em} dt.
\label{eq:P-Voros coeff}
\end{equation}
\end{defi}

\begin{rem} \normalfont
A priori, the Voros coefficients may depend on the 
choice of the path $\Gamma(\tau,\infty)$, 
hence we should use the notation $W_{\Gamma(\tau,\infty)}$. 
However, as is shown in the proof of Theorem \ref{Main Theorem3}, 
Voros coefficients depends only on the asymptotic behavior of 
$\lambda_{0}$ as $t \rightarrow \infty$, 
and the choice of the square root $R_{-1} = \sqrt{\Delta}$. 
They are independent of the lower end point $\tau$.
For this reason, we write them simply as \eqref{eq:P-Voros coeff}.
\end{rem}

The Voros coefficients represent a difference 
between several normalizations of transseries solutions.
To see this, we introduce the following 
two special normalizations. 
Note that, as we mentioned in the end of Section 
\ref{section:1-parameter solutions}, 
giving a normalization of transseries solution 
is equivalent to giving that of its $1$-instanton part 
\eqref{eq:first part of 1-parameter solution}.
(In this section we do not discuss the precise location of 
the independent variable $t$, and the choice of 
the path of the integration of 
\eqref{eq:first part of 1-parameter solution}. It will be 
specified in Section \ref{section:connection formulas} 
when we describe the connection formula concretely.)

The first one is {\it the normalization at $t = \tau$}, 
($\tau$ is a turning point or the simple-pole):
\begin{equation}
\lambda_{\tau}(t,\textbf{c},\eta;\alpha) = 
\lambda^{(0)}(t,\textbf{c},\eta) + 
\alpha \eta^{-{1}/{2}} 
\lambda_{\tau}^{(1)}(t,\textbf{c},\eta) e^{\eta \phi} 
+ (\alpha \eta^{-{1}/{2}})^{2} 
\lambda_{\tau}^{(2)}(t,\textbf{c},\eta) 
e^{2 \eta \phi} + \cdots,
\label{eq:1-parameter solution at tau}
\end{equation}
where its $1$-instanton part is normalized as 
\begin{equation}
\tilde{\lambda}_{\tau}^{(1)} 
(t,\textbf{c},\eta;\alpha) =
\alpha \hspace{+.1em} 
\frac{{\lambda^{(0)}}}
{\sqrt{t \hspace{+.1em} R_{\rm odd}}} 
\hspace{+.1em} {\rm exp} 
\biggl( \int_{\tau}^{t} 
R_{\rm odd}(t,\textbf{c},\eta) dt \biggr).
\label{eq:normalization at tau}
\end{equation}
Since the each coefficient $R_{2n-1}$ 
of $R_{\rm odd}$ has a branch point at $\tau$, the integral 
\eqref{eq:normalization at tau} should be 
considered as a contour integral.
(See Remark \ref{Remark:integral from P-turning points}.)
The second one is {\it the normalization at $t = \infty$}:
\begin{equation}
\lambda_{\infty}(t,\textbf{c},\eta;\alpha) = 
\lambda^{(0)}(t,\textbf{c},\eta) + 
\alpha \eta^{-{1}/{2}} 
\lambda_{\infty}^{(1)}(t,\textbf{c},\eta) e^{\eta \phi} 
+ (\alpha \eta^{-{1}/{2}})^{2} 
\lambda_{\infty}^{(2)}(t,\textbf{c},\eta) 
e^{2 \eta \phi} + \cdots,
\label{eq:1-parameter solution at infinity}
\end{equation}
where its 1-instanton part is normalized as 
\begin{equation}
\tilde{\lambda}_{\infty}^{(1)} 
(t,\textbf{c},\eta;\alpha) =
\alpha \hspace{+.1em} \frac{{\lambda^{(0)}}}
{\sqrt{t \hspace{+.1em} R_{\rm odd}}} 
\hspace{+.1em} {\rm exp} 
\biggl( \eta \int_{\tau}^{t} 
R_{-1}(t,\textbf{c}) \hspace{+.1em} dt + 
\int_{\infty}^{t} 
\bigl(R_{\rm odd}(t,\textbf{c},\eta) 
- \eta R_{-1}(t,\textbf{c}) \bigr) dt \biggr).
\label{eq:normalization at infty}
\end{equation}
Since the all coefficients of $R_{\rm odd}$ 
are integrable at $t = \infty$ 
except for the leading term $R_{-1}$ 
(see Appendix \ref{Appendix:asymptotics}), 
the integral in \eqref{eq:normalization at infty}
is well-defined. These two normalizations 
\eqref{eq:1-parameter solution at tau} and 
\eqref{eq:1-parameter solution at infinity} 
are related as
\begin{eqnarray}
\tilde{\lambda}^{(1)}_{\tau}
(t,\textbf{c},\eta;\alpha) & = & 
e^{W_{\infty}} \hspace{+.2em}
\tilde{\lambda}^{(1)}_{\infty}
(t,\textbf{c},\eta;\alpha), 
\label{eq:definition of the P-Voros coeff} \\[+.3em]
\lambda_{\tau}
(t,\textbf{c},\eta;\alpha) & = & 
\lambda_{\infty}
(t,\textbf{c},\eta;\alpha \hspace{+.1em} 
e^{W_{\infty}}), 
\label{eq:relation of two normalized 1-parameter solutions}
\end{eqnarray}
by the Voros coefficient $W_{\infty}$ for 
a suitable path $\Gamma(\tau,\infty)$ from $\tau$ to $\infty$.

\begin{rem} \normalfont
\label{Remark:integral from P-turning points}

Here we give a remark for integrals of $R_{\rm odd}$ 
from turning points or the simple-pole. 
Since we can easily verify that $R_{2n-1}$ 
has a Puiseux expansion  
\begin{equation}
R_{2n-1} = (t - \tau)^{N/4} 
\sum_{k \ge 0} r_{k} (t - \tau)^{k/2}
\label{eq:Puiseux}
\end{equation}
near $t = \tau$ 
(where $\tau$ is a turning point or the simple-pole)
with some odd integer $N$, we can define the integral 
from $\tau$ in terms of contour integral on the 
Riemann surface of $\sqrt{\Delta}$ 
(i.e., the double cover of the Riemann surface of 
$\lambda_{0}$ branching at turning points 
and the simple-pole). 
In addition to Figure 
\ref{fig:P3D6Stokes-sheet-1} $\sim$ 
\ref{fig:P3D6Stokes-sheet-4}, 
we need further branch cuts to determine 
the square root $\sqrt{\Delta}$ as in 
Figure \ref{fig:P3D6-Gamma-t-sheet-1} 
$\sim$ \ref{fig:P3D6-Gamma-t-sheet-4}, 
which describe the case ${\bf c} = (2, 2-i)$.
In these figures the spiral lines designate 
the new branch cuts, and we fix the square root 
so that the real part of the integral 
\[
\int_{\tau}^{t} \sqrt{\Delta(t,{\bf c})} 
\hspace{+.1em} dt
\]
along a Stokes curve is positive 
if the integral is taken 
in parallel with the arrow on the Stokes curve.
Then, for example, if $t$ is fixed at the point 
in Figure \ref{fig:P3D6-Gamma-t-sheet-2}, 
we define the integral from $\tau_{1}$ to $t$ 
by the following contour integral: 
\begin{equation}
\int_{\tau_{1}}^{t} R_{\rm odd}(t,{\bf c},\eta)
\hspace{+.1em} dt  = \frac{1}{2} 
\int_{{\Gamma}_{t}} R_{\rm odd}(t,{\bf c},\eta) 
\hspace{+.1em} dt, 
\label{eq:definition of contour integral}
\end{equation}
where ${\Gamma}_{t}$ is a path on the Riemann surface 
of $\sqrt{\Delta}$ described in 
Figure \ref{fig:P3D6-Gamma-t-sheet-1} 
and \ref{fig:P3D6-Gamma-t-sheet-2}.  
Here $\check{t}$ in Figure \ref{fig:P3D6-Gamma-t-sheet-2}
represents the point on the other sheet of square root 
$\sqrt{\Delta}$ satisfying 
$\lambda_{0}(\check{t},{\bf c}) = 
\lambda_{0}({t},{\bf c})$ and  
\[
R_{\ell}(\check{t},{\bf c}) = 
(-1)^{\ell} R_{\ell}(t,{\bf c}) 
\hspace{+1.em} (\ell \ge -1)
\] 
holds. The dotted part of $\Gamma_{t}$ 
represents the path on the other sheet of 
square root $\sqrt{\Delta}$. 
We can also define integrals from other turning points 
or the simple-pole in the same manner.
\end{rem}

  \begin{figure}[h]
  \begin{minipage}{0.5\hsize}
  \begin{center}
  \includegraphics[width=55mm]
  {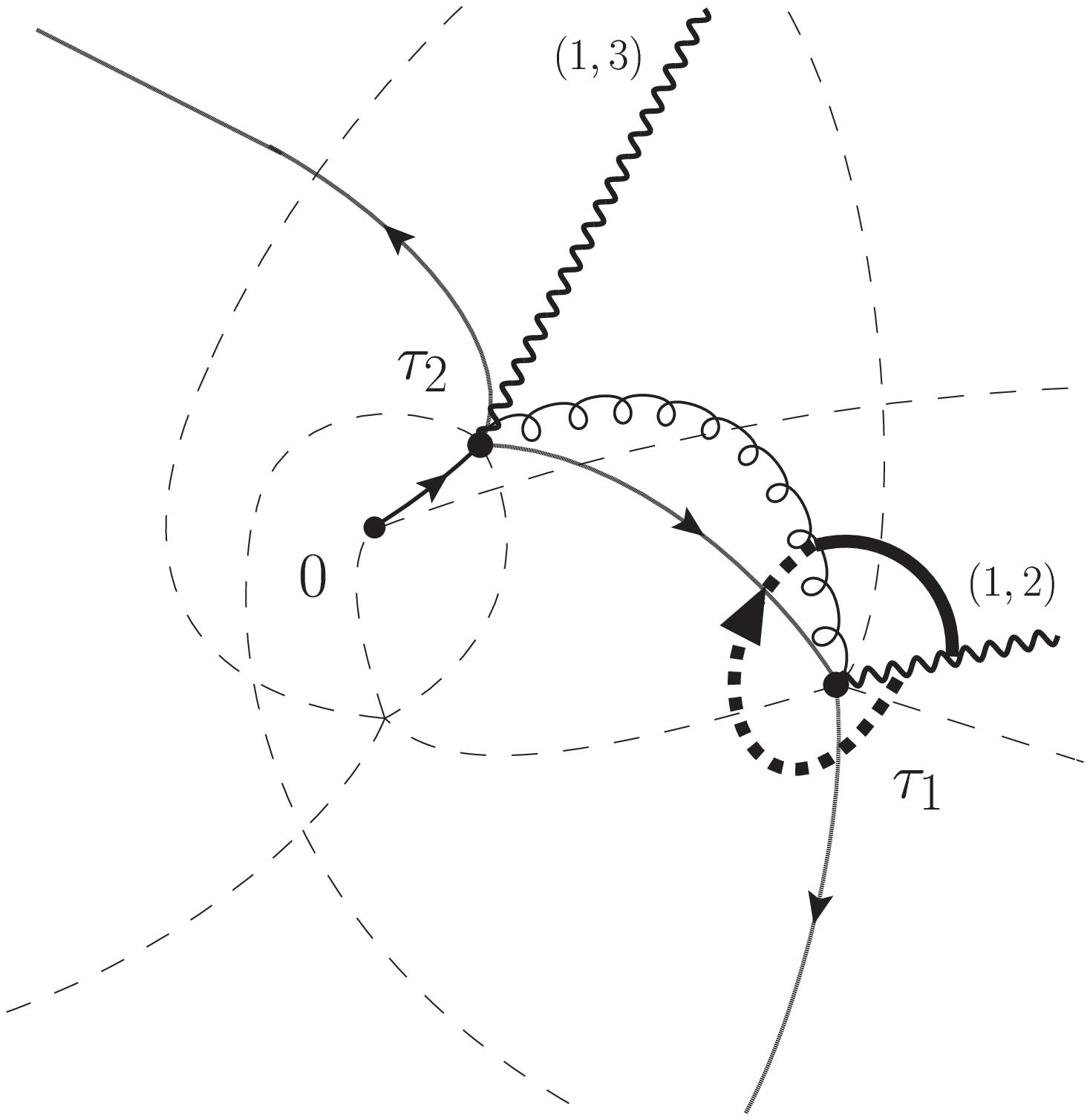}
  \end{center}
  \caption{\small{Sheet 1.}}
  \label{fig:P3D6-Gamma-t-sheet-1}
  \end{minipage}
  \begin{minipage}{0.5\hsize}
  \begin{center}
  \includegraphics[width=55mm]
  {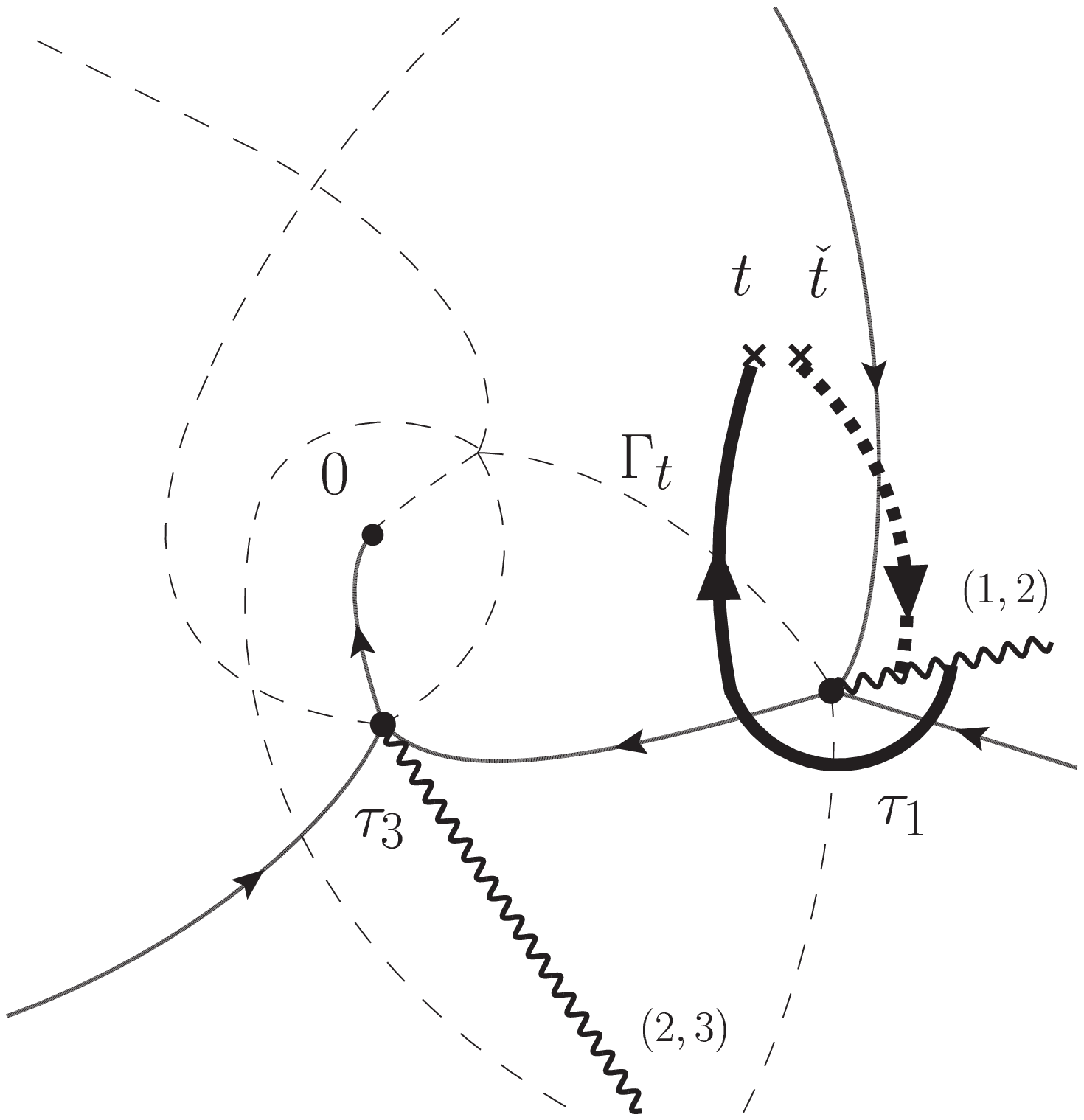}
  \end{center}
  \caption{\small{Sheet 2.}}
  \label{fig:P3D6-Gamma-t-sheet-2}
  \end{minipage}
  \end{figure}
  \begin{figure}[h]
  \begin{minipage}{0.5\hsize}
  \begin{center}
  \includegraphics[width=55mm]
  {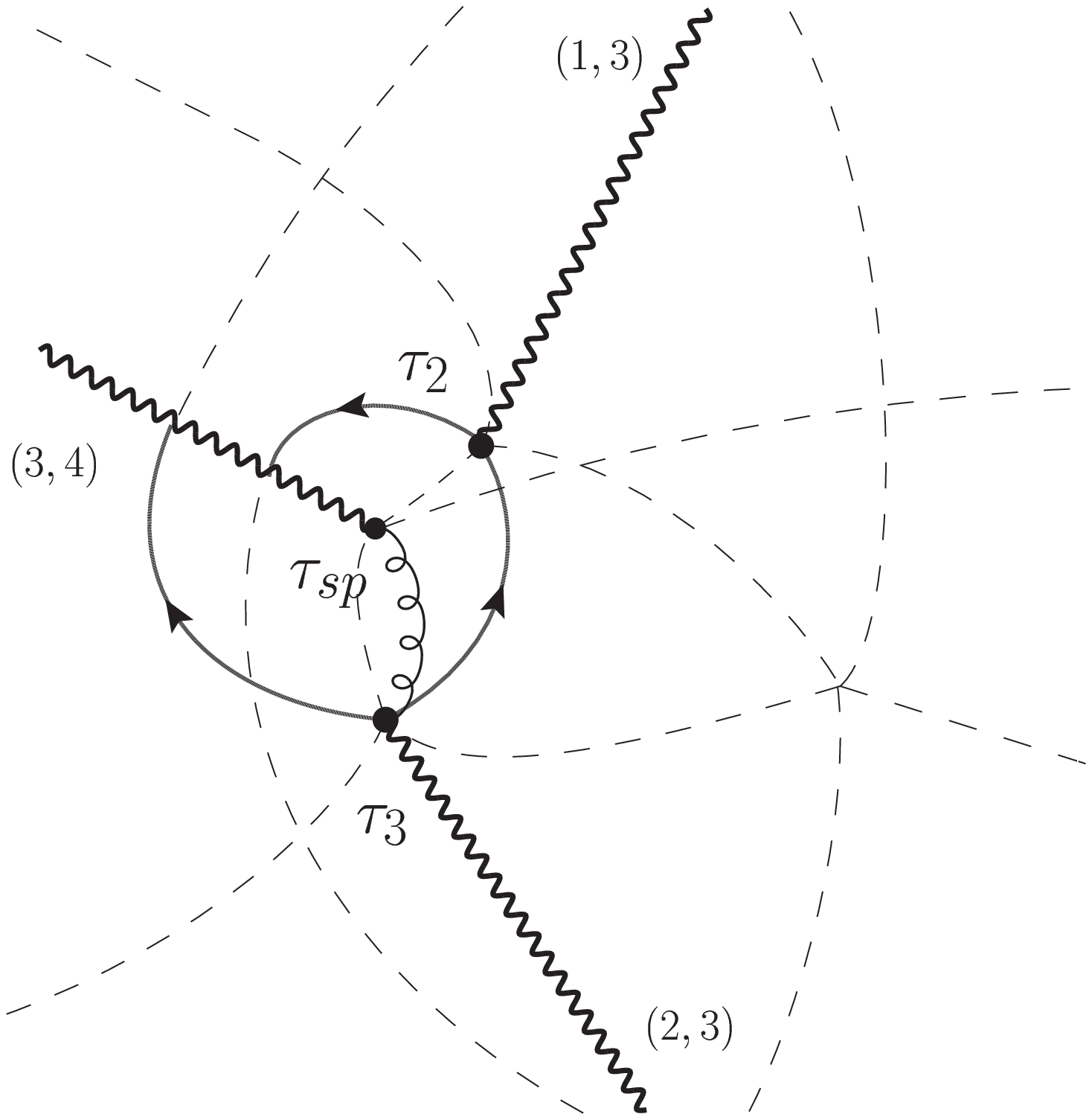}
  \end{center}
  \caption{\small{Sheet 3.}}
  \label{fig:P3D6-Gamma-t-sheet-3}
  \end{minipage}
  \begin{minipage}{0.5\hsize}
  \begin{center}
  \includegraphics[width=55mm]
  {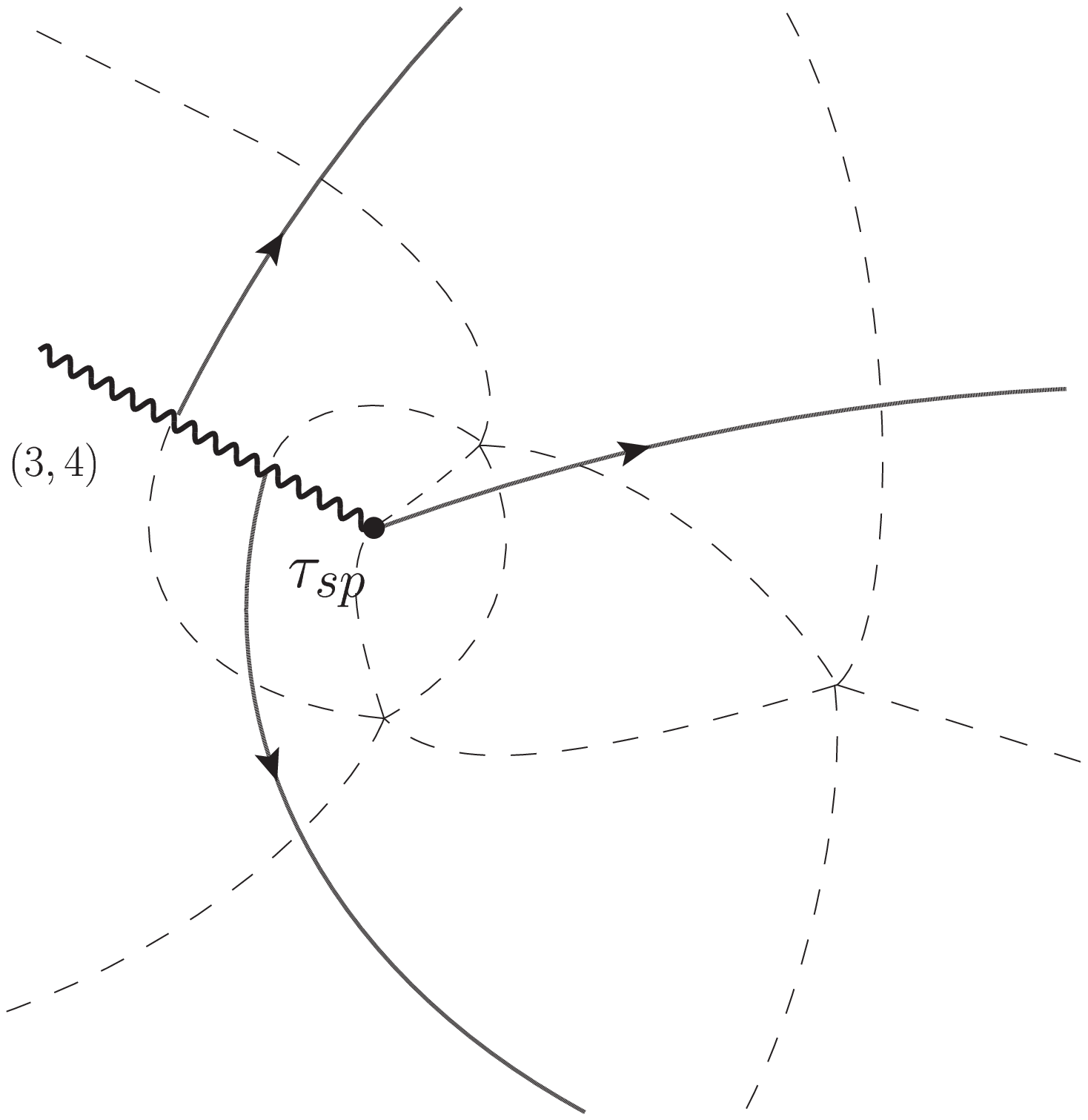}
  \end{center}
  \caption{\small{Sheet 4.}}
  \label{fig:P3D6-Gamma-t-sheet-4}
  \end{minipage}
  \end{figure}

There are several Voros coefficients in 
accordance with choices of a path $\Gamma(\tau,\infty)$ 
of the integration \eqref{eq:P-Voros coeff}, and  we should 
distinguish the branch of $\lambda_{0}$ near $t = \infty$. 
We know from the algebraic equation 
\eqref{eq:lambda(0)0} that $\lambda_{0}$ 
has following four asymptotic behaviors as $t$ 
tends to infinity:
\begin{eqnarray}
\lambda_{0} & = & + t^{1/2} 
\bigl( 1 + O(t^{-1/2}) \bigr)  
\hspace{+1.4em} \text{as $t \rightarrow \infty_{1}$}, 
\label{eq:behavior of lambda0 infinity-1} \\
\lambda_{0} & = & - t^{1/2} 
\bigl( 1 + O(t^{-1/2}) \bigr) 
\hspace{+1.4em} \text{as $t \rightarrow \infty_{2}$}, 
\label{eq:behavior of lambda0 infinity-2} \\
\lambda_{0} & = & + i \hspace{+.1em} t^{1/2} 
\bigl( 1 + O(t^{-1/2}) \bigr) 
\hspace{+1.em} \text{as $t \rightarrow \infty_{3}$}, 
\label{eq:behavior of lambda0 infinity-3} \\
\lambda_{0} & = & - i \hspace{+.1em} t^{1/2} 
\bigl( 1 + O(t^{-1/2}) \bigr) 
\hspace{+1.em} \text{as $t \rightarrow \infty_{4}$}. 
\label{eq:behavior of lambda0 infinity-4}
\end{eqnarray}
Here we use the symbols $\infty_{j}$ 
($1 \le j \le 4$) to distinguish the 
asymptotic behaviors of $\lambda_{0}$. 
Note that both $t = \infty_{1}$ 
and $\infty_{2}$ correspond to 
$u = \infty$, while both $t = \infty_{3}$ 
and $\infty_{4}$ correspond 
to $u = 0$ in the coordinate \eqref{eq:u-coordinate}. 
We use another symbols $\infty_{j, \pm}$ 
($1 \le j \le 4$) to specify the choice of the 
square root $R_{-1} = \sqrt{\Delta(t,{\bf c})}$ 
as follows:

\begin{eqnarray}
\lambda_{0} & = & + t^{1/2} 
\bigl( 1 + O(t^{-1/2}) \bigr), \hspace{+1.em}
R_{-1} = \pm 2 \hspace{+.1em} t^{-1/2} 
\bigl( 1 + O(t^{-1/2}) \bigr) 
\hspace{+1.3em} \text{as $t \rightarrow \infty_{1, \pm}$}, 
\nonumber \\ & & 
\label{eq:behavior of lambda0 infinity-1-pm} \\
\lambda_{0} & = & - t^{1/2} 
\bigl( 1 + O(t^{-1/2}) \bigr), \hspace{+1.em}
R_{-1} = \mp 2 \hspace{+.1em} t^{-1/2} 
\bigl( 1 + O(t^{-1/2}) \bigr) 
\hspace{+1.3em} \text{as $t \rightarrow \infty_{2, \pm}$}, 
\nonumber \\ & & 
\label{eq:behavior of lambda0 infinity-2-pm} \\
\lambda_{0} & = & + i \hspace{+.1em} t^{1/2} 
\bigl( 1 + O(t^{-1/2}) \bigr) , \hspace{+.6em}
R_{-1} = \pm 2 i \hspace{+.1em} t^{-1/2}  
\bigl( 1 + O(t^{-1/2}) \bigr) 
\hspace{+1.em} \text{as $t \rightarrow \infty_{3, \pm}$}, 
\nonumber \\ & & 
\label{eq:behavior of lambda0 infinity-3-pm} \\
\lambda_{0} & = & - i \hspace{+.1em} t^{1/2} 
\bigl( 1 + O(t^{-1/2}) \bigr) , \hspace{+.6em}
R_{-1} = \mp 2 i \hspace{+.1em} t^{-1/2}  
\bigl( 1 + O(t^{-1/2}) \bigr) 
\hspace{+1.em} \text{as $t \rightarrow \infty_{4, \pm}$}. 
\nonumber \\ & & 
\label{eq:behavior of lambda0 infinity-4-pm}
\end{eqnarray}
See Appendix \ref{Appendix:asymptotics} for higher 
order asymptotic behaviors. 
Then, our first main result are stated as follows.

\begin{theo} \label{Main Theorem3}
Let ${\cal F}(c,\eta)$ be a formal power series 
in $\eta^{-1}$ defined by 
\begin{equation}
{\cal F}(c,\eta) = \sum_{n=1}^{\infty} 
\frac{2^{1-2n} - 1}{2n (2n - 1)} \hspace{+.1em} B_{2n} 
\hspace{+.1em} (c \hspace{+.1em} \eta )^{1-2n},
\label{eq:Voros-coeff-F}
\end{equation}
where $B_{2n}$ is the $2n$-th Bernoulli number defined by 
\eqref{eq:Bernoulli number}.
Then, the Voros coefficient $W_{\infty}({\bf c},\eta)$ for 
a path $\Gamma(\tau, \infty)$ 
is represented explicitly as follows: 
\begin{eqnarray}
W_{\infty_{1, \pm}}(\textbf{c},\eta) =
W_{\infty_{2, \pm}}(\textbf{c},\eta) 
& = & \pm \hspace{+.1em} {\cal F}(c_{p},\eta).
\label{eq:Main Theorem1+} \\[+.3em]
W_{\infty_{3,\pm}}(\textbf{c},\eta) = 
W_{\infty_{4,\pm}}(\textbf{c},\eta) 
& = & \pm \hspace{+.1em} {\cal F}(c_{m},\eta) .
\label{eq:Main Theorem1-} 
\end{eqnarray}
Here $c_{p}$ and $c_{m}$ are given by
\begin{equation}
c_{p} = \frac{c_{\infty} + c_{0}}{2} , 
\hspace{+1.em}
c_{m} = \frac{c_{\infty} - c_{0}}{2}.
\label{eq:c-plus-minus}
\end{equation}
\end{theo}

The proof of Theorem \ref{Main Theorem3} 
will be given in Section \ref{section:proof of the Main theorems}
together with that of Theorem \ref{Main Theorem2} below.

\subsection{Voros coefficients $W_{0_{c_{\infty}}}$ 
and $W_{0_{c_{0}}}$}
\label{section:P-Voros coefficients for t = 0}

In the previous subsection we defined 
the Voros coefficient for $t = \infty$. 
However, since $(P_{\rm III'})_{D_{6}}$ also has double-pole type 
singular points \eqref{eq:double-pole type infty}
and \eqref{eq:double-pole type zero} at $t = 0$
(cf.\ Remark \ref{Remark:uniformization}), 
we can consider Voros coefficients 
relevant to these double-poles.

First, we specify the branch of $\lambda_{0}$ 
near the double poles. 
We use symbols $0_{c_{\infty}}$ and $0_{c_{0}}$ 
depending on the asymptotic behaviors 
\eqref{eq:double-pole type infty} and 
\eqref{eq:double-pole type zero} of $\lambda_{0}$; 
that is, if $t$ tends to $0_{c_{\infty}}$ 
(resp., to $0_{c_{0}}$), then $\lambda_{0}$ behaves as 
\eqref{eq:double-pole type infty} 
(resp., \eqref{eq:double-pole type zero}). 
Then, the Voros coefficients 
for double-poles are defined as follows:

\begin{defi} \normalfont
For a path $\Gamma(\tau, 0_{c_{\ast}})$ from 
a turning point (or the simple-pole) $\tau$ to 
$0_{c_{\ast}}$ ($\ast = \infty$ or $0$), 
{\it the Voros coefficient  
for the path $\Gamma(\tau, 0_{c_{\ast}})$} 
is defined by 
\begin{equation}
W_{0_{c_{\ast}}}(\textbf{c},\eta) = 
\int_{\Gamma(\tau, 0_{c_{\ast}})} 
\bigl( R_{\rm odd}(t,\textbf{c},\eta) - 
\eta R_{-1}(t,\textbf{c}) \bigr) \hspace{+.1em} dt.
\label{eq:P-Voros coeff for 0}
\end{equation}
\end{defi}

It turns out to be that the right-hand side of 
\eqref{eq:P-Voros coeff for 0} is independent 
of the choice of the path $\Gamma(\tau, 0_{c_{\ast}})$
(see Theorem \ref{Main Theorem2}). As well as 
\eqref{eq:relation of two normalized 1-parameter solutions}, 
the Voros coefficients $W_{0_{c_{\ast}}}$ 
also describe a difference between 
$\lambda_{\tau}(t,c,\eta;\alpha)$ and the 
transseries solution 
$\lambda_{0_{c_{\ast}}}(t,c,\eta;\alpha)$
{\it normalized at double-poles}; 
\begin{equation}
\lambda_{0_{c_{\infty}}}(t,\textbf{c},\eta;\alpha) = 
\lambda^{(0)}(t,\textbf{c},\eta) + 
\alpha \eta^{-{1}/{2}} 
\lambda_{0_{c_{\infty}}}^{(1)}(t,\textbf{c},\eta) e^{\eta \phi} 
+ (\alpha \eta^{-{1}/{2}})^{2} 
\lambda_{0_{c_{\infty}}}^{(2)}(t,\textbf{c},\eta) 
e^{2 \eta \phi} + \cdots,
\label{eq:1-parameter solution at 0 : c-infty} \\[+.3em]
\end{equation}
\begin{equation}
\lambda_{0_{c_{0}}}(t,\textbf{c},\eta;\alpha) = 
\lambda^{(0)}(t,\textbf{c},\eta) + 
\alpha \eta^{-{1}/{2}} 
\lambda_{0_{c_{0}}}^{(1)}(t,\textbf{c},\eta) e^{\eta \phi} 
+ (\alpha \eta^{-{1}/{2}})^{2} 
\lambda_{0_{c_{0}}}^{(2)}(t,\textbf{c},\eta) 
e^{2 \eta \phi} + \cdots, 
\label{eq:1-parameter solution at 0 : c-0}
\end{equation}
where the 1-instanton part is normalized as 
\begin{equation}
\tilde{\lambda}_{0_{c_{\infty}}}^{(1)} 
(t,\textbf{c},\eta;\alpha) =
\alpha \hspace{+.1em} \frac{{\lambda^{(0)}}}
{\sqrt{t \hspace{+.1em} R_{\rm odd}}} 
\hspace{+.1em} {\rm exp} 
\biggl( \eta \int_{\tau}^{t} R_{-1}(t,\textbf{c}) 
\hspace{+.1em} dt + 
\int_{0_{c_{\infty}}}^{t} \bigl(R_{\rm odd}(t,\textbf{c},\eta) 
- \eta R_{-1}(t,\textbf{c}) \bigr) dt \biggr),
\label{eq:normalization at 0 : c-infty}
\end{equation}
\begin{equation}
\tilde{\lambda}_{0_{c_{0}}}^{(1)} 
(t,\textbf{c},\eta;\alpha) =
\alpha \hspace{+.1em} \frac{{\lambda^{(0)}}}
{\sqrt{t \hspace{+.1em} R_{\rm odd}}} 
\hspace{+.1em} {\rm exp} 
\biggl( \eta \int_{\tau}^{t} R_{-1}(t,\textbf{c}) 
\hspace{+.1em} dt + 
\int_{0_{c_{0}}}^{t} \bigl(R_{\rm odd}(t,\textbf{c},\eta) 
- \eta R_{-1}(t,\textbf{c}) \bigr) dt \biggr).
\label{eq:normalization at 0 : c-0}
\end{equation}
Since the all coefficient $R_{2n-1}$ of $R_{\rm odd}$ 
are integrable at these double-poles except 
for the leading term $R_{-1}$ 
(see Appendix \ref{Appendix:asymptotics}), 
the above integrals are well-defined.

To state our second main result, we introduce another symbols 
in order to specify the choice of the square root  
$R_{-1} = \sqrt{\Delta}$ at double-poles:
\begin{eqnarray}
\lambda_{0} & = & c_{\infty}
\bigl( 1 + O(t) \bigr), \hspace{+1.em}
R_{-1} = \pm \frac{c_{\infty}}{t} \bigl( 1 + O(t) \bigr) 
\hspace{+1.em} \text{as $t \rightarrow 0_{c_{\infty}, \pm}$},
\label{eq:behavior of lambda0 0-infinity-pm} \\
\lambda_{0} & = & \frac{t}{c_{0}}
\bigl( 1 + O(t) \bigr), \hspace{+1.2em}
R_{-1} = \pm \frac{c_{0}}{t} \bigl( 1 + O(t) \bigr) 
\hspace{+1.3em} \text{as $t \rightarrow 0_{c_{0}, \pm}$}.
\label{eq:behavior of lambda0 0-0-pm} 
\end{eqnarray}
Then our second main result which is shown 
in the next section are stated as follows.

\begin{theo} \label{Main Theorem2}
Let ${\cal F}(c,\eta)$ and ${\cal G}(c,\eta)$ 
be the formal power series in $\eta^{-1}$ 
defined by \eqref{eq:Voros-coeff-F} and 
\begin{equation}
{\cal G}(c,\eta) = \sum_{n=1}^{\infty} 
\frac{B_{2n}}{2n (2n - 1)} \hspace{+.1em}  
\hspace{+.1em} (c \hspace{+.1em} \eta )^{1-2n}, 
\label{eq:Voros-coeff-G}
\end{equation}
where $B_{2n}$ is the $2n$-th Bernoulli number 
defined by \eqref{eq:Bernoulli number}.  
Then, the Voros coefficient $W_{0_{c_{\ast}}}({\bf c},\eta)$
for a path $\Gamma(\tau, 0_{c_{\ast}})$ $(\ast = \infty$, $0)$
is represented explicitly as follows:
\begin{eqnarray}
W_{0_{c_{\infty}, \pm}}(\textbf{c},\eta) 
& = & \pm \Bigl\{
{\cal F}( c_{p},\eta ) + {\cal F}(c_{m} ,\eta ) 
 - 3 \hspace{+.1em} {\cal G} ( c_{\infty},\eta ) \Bigr\}, 
\label{eq:P-Voros for 0-c-infinity} \\[+.5em]
W_{0_{c_{0}, \pm}}(\textbf{c},\eta) 
& = & \pm \Bigl\{
{\cal F}( c_{p},\eta) - {\cal F}(c_{m},\eta) 
 - 3 \hspace{+.1em} {\cal G} ( c_{0},\eta ) \Bigr\}.
\label{eq:P-Voros for 0-c-0}
\end{eqnarray}
Here $c_{p}$ and $c_{m}$ are given by \eqref{eq:c-plus-minus}.
\end{theo}

Theorem \ref{Main Theorem2} is also proved in 
in Section \ref{section:proof of the Main theorems}.

\begin{rem} \normalfont
\label{Remark:F and G in Voros coeff}

As you see in the list of Theorem \ref{Main Theorem3} 
and Theorem \ref{Main Theorem2}, 
the formal power series ${\cal F}(c,\eta)$ and 
${\cal G}(c,\eta)$ are fundamental peaces of 
the Voros coefficients. 
These formal power series also appear as 
Voros coefficients of other equations.
For example, ${\cal F}(c,\eta)$ appears as the 
Voros coefficient of the Weber equation 
(\cite{Takei Sato conjecture}) which is linear, 
the second Painlev\'e equation (\cite{Iwaki}), 
and a fourth order analogue of 
the second Painlev\'e equation (\cite{Iwaki in prep})
\begin{equation}
\frac{d^{4}\lambda}{dt^{4}} = 
 \eta^{2} \Bigl(10 \lambda^{2} \frac{d^{2}\lambda}{dt^{2}}
+ 10 \lambda \Bigl( \frac{d \lambda}{dt} \Bigr)^{2} \Bigr)
+ \eta^{4} ( -6 \lambda^{5} + t \lambda + c) 
\label{eq:higher P2}
\end{equation}
which is considered in \cite{KT WKB higher Painleve}. 
On the other hand, ${\cal G}(c,\eta)$ appears in the Voros 
coefficients of the hypergeometric equation (\cite{Aoki-Tanda}) 
and the Bessel equation (\cite{Iwaki in prep2}). 
Note also that ${\cal G}(c,\eta)$ coincides with 
the Voros coefficient of the degenerate third 
Painlev\'e equation of the type $D_{7}$. 
See Theorem \ref{theorem:D7-Voros} 
in Appendix \ref{Appendix:P3-D7}.
\end{rem}

\begin{rem} \normalfont 
\label{rem:degeneration and summability}
Note that, when $c$ is pure imaginary, 
${\cal F}(c,\eta)$ and ${\cal G}(c,\eta)$ are not 
Borel summable as a formal power series in $\eta^{-1}$ 
(see Proposition \ref{Prop:Borel sum of F and G}). 
Intriguingly, when one of Voros coefficients 
is not Borel summable 
(i.e., one of $c_{p}$, $c_{m}$, $c_{\infty}$  
or $c_{0}$ is pure imaginary), 
then the Stokes geometry degenerates, 
as far as we have checked; 
see Appendix \ref{Appendix:examples of P-Stokes geometry}.
Moreover, the Borel sums of ${\cal F}(c,\eta)$ and 
${\cal G}(c,\eta)$ jump at the imaginary axis of $c$-plane. 
Since the Borel sum of ${\cal F}(c,\eta)$ 
and ${\cal G}(c,\eta)$ are computed 
explicitly as in Proposition \ref{Prop:Borel sum of F and G}, 
we can describe connection formulas for parametric Stokes 
phenomena exactly. See Section 
\ref{section:connection formulas} for details.
\end{rem}

\section{Proof of the Main theorems}
\label{section:proof of the Main theorems}

In this section we give proofs of the main theorems 
about explicit representations of the 
Voros coefficients (Theorem \ref{Main Theorem3} 
and Theorem \ref{Main Theorem2}). 
To compute the Voros coefficients, 
we adopt a similar method used in 
\cite{Takei Sato conjecture} and \cite{Koike-Takei}. 
Especially, here we consider the Voros coefficient
$W_{\Gamma}(\textbf{c},\eta)$ defined by 
the integral \eqref{eq:P-Voros coeff} along 
the Stokes curve $\Gamma$ in  Figure \ref{fig:PIII',0}, 
which emanates from $\tau_{1}$ and flows to $\infty_{3,+}$. 
That is, $W_{\Gamma}$ is defined by the following contour 
integral
\begin{equation}
W_{\Gamma}({\bf c},\eta) = \frac{1}{2} 
\int_{\Gamma_{\rm contour}} 
\bigl( R_{\rm odd}(t,\textbf{c},\eta) - 
\eta R_{-1}(t,\textbf{c}) \bigr) \hspace{+.1em} dt, 
\label{eq:def of P-Voros coeff W-Gamma}
\end{equation}
where the path $\Gamma_{\rm contour}$ is 
a path on the Riemann surface 
of $\sqrt{\Delta}$ taken as in 
Figure \ref{fig:P3D6-Gamma-contour-sheet-1} and 
\ref{fig:P3D6-Gamma-contour-sheet-2}.
For this Voros coefficient, we will show the following.

  \begin{figure}[h]
  \begin{minipage}{0.5\hsize}
  \begin{center}
  \includegraphics[width=55mm]
  {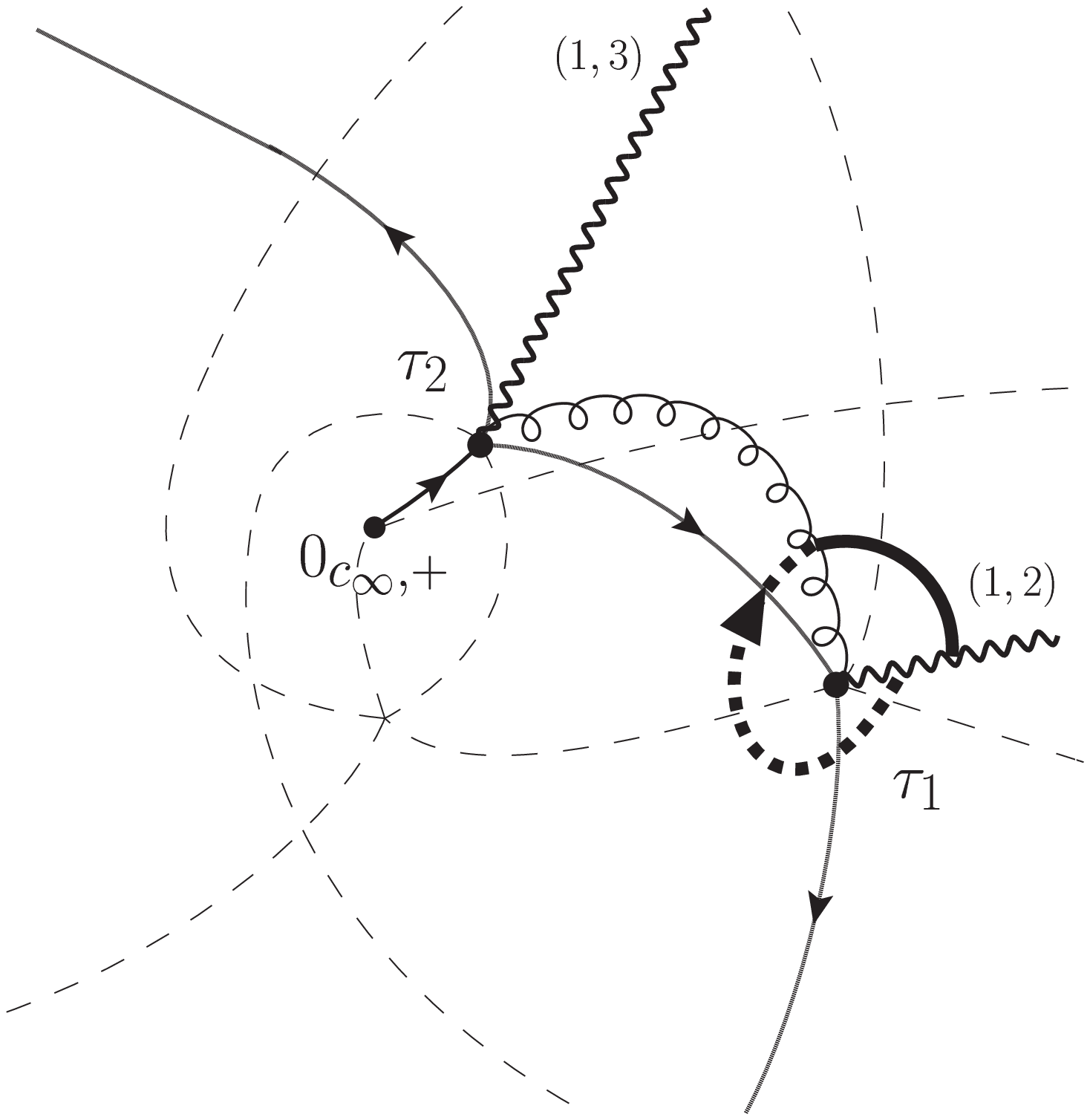}
  \end{center}
  \caption{\small{$\Gamma_{\rm contour}$ on Sheet 1.}}
  \label{fig:P3D6-Gamma-contour-sheet-1}
  \end{minipage}
  \begin{minipage}{0.5\hsize}
  \begin{center}
  \includegraphics[width=55mm]
  {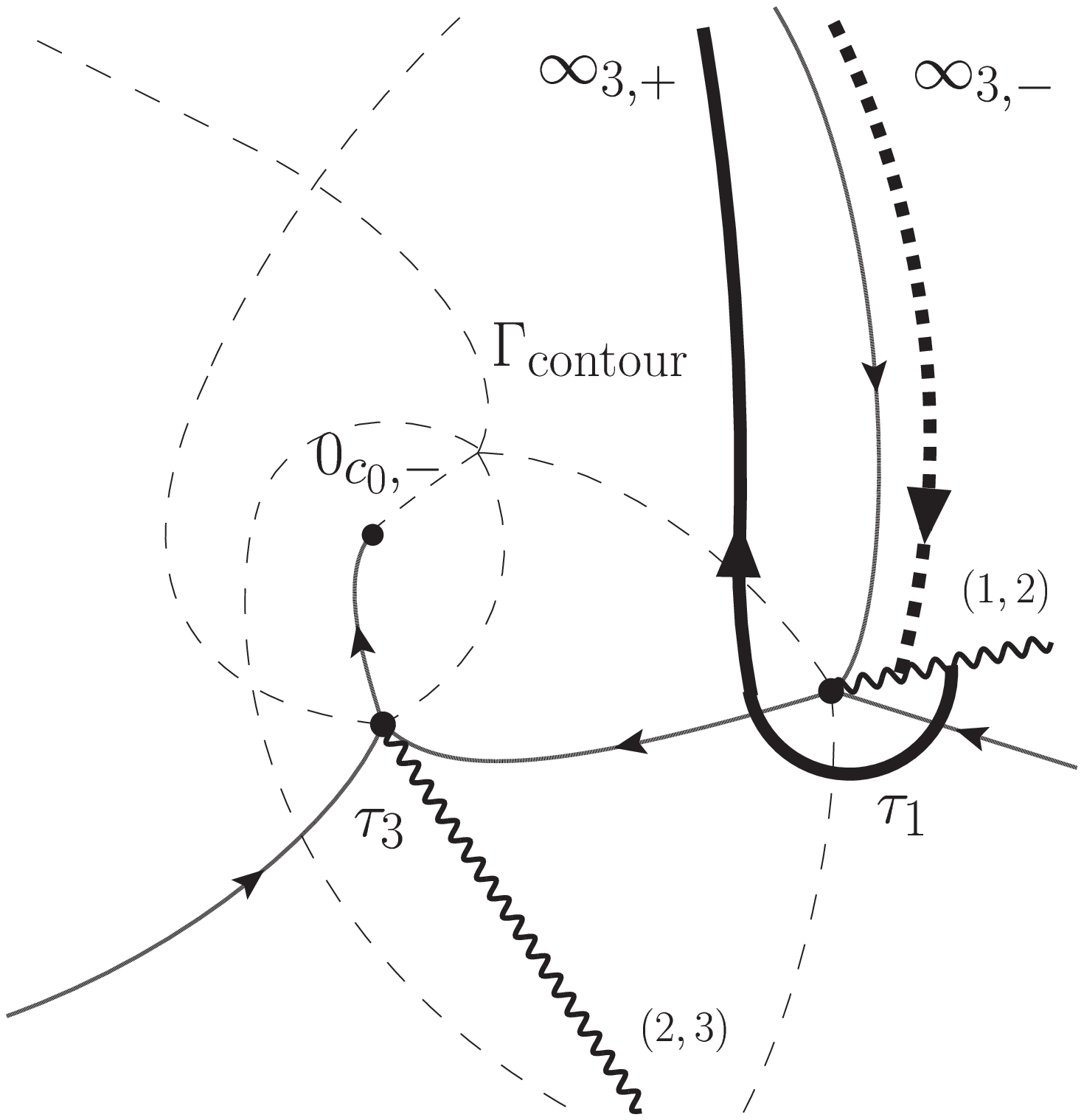}
  \end{center}
  \caption{\small{$\Gamma_{\rm contour}$ on Sheet 2.}}
  \label{fig:P3D6-Gamma-contour-sheet-2}
  \end{minipage}
  \end{figure}

\begin{theo} \label{Main Theorem}
The Voros coefficient $W_{\Gamma}(\textbf{c},\eta)$ 
is represented explicitly as follows:
\begin{equation}
W_{\Gamma}(\textbf{c},\eta) 
= \sum_{n=1}^{\infty} 
\frac{2^{1-2n} - 1}{2n (2n - 1)} \hspace{+.1em} B_{2n} 
\hspace{+.1em} \Bigl( \frac{c_{\infty} - c_{0}}{2} 
\hspace{+.1em} \eta \Bigr)^{1-2n} .
\label{eq:Main Theorem}
\end{equation}
Here $B_{2n}$ is the $2n$-th Bernoulli number defined by 
\eqref{eq:Bernoulli number}.
\end{theo}

This statement is included in the statement 
of Theorem \ref{Main Theorem3}. The other statements 
in Theorem \ref{Main Theorem3} and \ref{Main Theorem2} 
can be proved in the same manner presented here.

\subsection{Proof of Theorem \ref{Main Theorem}
(and Theorem \ref{Main Theorem3})}
\label{subsection:Proof of Main Theorem1}

The following lemma is a key for the proof 
of our main theorems.

\begin{lemm}[cf.\ 
{\cite[Lemma 1.2]{Takei Sato conjecture}} and 
{\cite[$\S$3]{Koike-Takei}}] 
\label{key lemma}
(i) The formal power series ${\cal F}(c,\eta)$ 
and ${\cal G}(c,\eta)$ defined in \eqref{eq:Voros-coeff-F} 
and \eqref{eq:Voros-coeff-G} satisfy the following 
difference equations:
\begin{eqnarray}
{\cal F}(c + \eta^{-1},\eta) - {\cal F}(c,\eta) & = & 
1 - (c \hspace{+.1em} \eta + 1)\hspace{+.1em} {\rm{log}} 
\Bigl( 1+\frac{1}{c \hspace{+.1em} \eta} \Bigr) 
+ {\rm{log}}
\Bigl( 1+\frac{1}{2 c \hspace{+.1em} \eta} \Bigr), 
\nonumber \\
\label{eq:difference equation for F} \\
{\cal G}(c + \eta^{-1},\eta) - {\cal G}(c,\eta) & = & 
1 - \Bigl( c \hspace{+.1em} \eta + \frac{1}{2} \Bigr) 
\hspace{+.1em} {\rm{log}} 
\Bigl( 1+\frac{1}{c \hspace{+.1em} \eta} \Bigr).
\label{eq:difference equation for G}
\end{eqnarray}

\noindent
(ii) Conversely, if there exist a formal power series solution 
of the equation \eqref{eq:difference equation for F} 
or \eqref{eq:difference equation for G} of the form 
\[
\sum_{\ell \ge 1} a_{\ell} \hspace{+.1em}
(c \hspace{+.1em} \eta)^{-\ell} 
\]
with some $a_{\ell} \in {\mathbb C}$ which is independent 
of both $c$ and $\eta$ $(\ell \ge 1)$, then it coincides 
with ${\cal F}(c,\eta)$ or ${\cal G}(c,\eta)$, respectively.
\end{lemm}

To apply Lemma \ref{key lemma} for the proof 
of our main theorems, 
we derive difference equations satisfied by 
Voros coefficients. For the purpose, 
we use the following two {\it B\"acklund transformations} 
$T_{1}$ and $T_{2}$, where $T_{1}$ induces the shift of 
the parameters of $(P_{\rm III'})_{D_{6}}$ as 
$(c_{\infty},c_{0}) \mapsto 
(c_{\infty} + \eta^{-1}, c_{0} + \eta^{-1})$, 
while $T_{2}$ induces $(c_{\infty},c_{0}) \mapsto 
(c_{\infty} + \eta^{-1}, c_{0} - \eta^{-1})$.
It is convenient to use the following 
Hamiltonian system $(H_{\rm III'})_{D_{6}}$ 
which is equivalent to $(P_{\rm III'})_{D_{6}}$ 
(cf.\ \cite{Okamoto Painleve}): 
\begin{equation}
(H_{\rm III'})_{D_{6}}  : \hspace{+.5em}
\frac{d\lambda}{dt} = \eta 
\frac{\partial {\cal H}}{\partial \mu}, 
\hspace{+1.em}
\frac{d\mu}{dt} = - \eta 
\frac{\partial {\cal H}}{\partial \lambda}, 
\label{eq:H3D6}
\end{equation}
where the Hamiltonian ${\cal H} = {\cal H}(\lambda,\mu,t)$ 
is defined by 
\begin{equation}
t \hspace{+.1em} {\cal H} = \lambda^{2} \mu^{2} 
- \bigl( \lambda^{2} + (c_{0} - \eta^{-1}) \lambda - t \bigr) \mu 
+ \frac{1}{2} (c_{\infty} + c_{0} - \eta^{-1}) \lambda. 
\label{eq:hamiltonian of H3D6}
\end{equation}
Then the explicit form of B\"acklund transformations 
$T_{1}$ and $T_{2}$ are given by the following:

\begin{lemm}[e.g., \cite{Jimbo-Miwa, Okamoto Painleve}] 
\label{Lemma:Backlund}
Let $(\lambda,\nu)$ be a solution of 
$(H_{\rm III'})_{D_{6}}$. 
Then, $(\Lambda, M) = 
(\Lambda_{j}(\lambda,\nu), M_{j}(\lambda,\nu))$
$(j = 1,2)$ defined by 
\begin{eqnarray}
\hspace{-.5em} &  & 
\begin{cases}
\displaystyle \Lambda_{1} = - \frac{t}{\lambda} 
+ \frac{(c_{\infty} + c_{0} + \eta^{-1}) t}{2\lambda^{2}(\mu - 1) 
+ (c_{\infty} - c_{0} + \eta^{-1}) \lambda + 2t}, \\[+1.5em]
\displaystyle M_{1} = \frac{\lambda^{2}(\mu - 1)}{t} 
+ \frac{(c_{\infty} - c_{0} + \eta^{-1}) \lambda }{2t} 
+ 1 ,
\end{cases} \label{eq:Backlund1} \\[+.7em]
\hspace{-.5em} &  &
\begin{cases}
\displaystyle \Lambda_{2} = \frac{2t(\mu - 1)}
{2\lambda(\mu - 1) + (c_{\infty} - c_{0} + \eta^{-1})}, 
\\[+1.5em]
\displaystyle M_{2} = \frac{1}{t} \Bigl\{ 
\frac{c_{\infty} + c_{0} - \eta^{-1}}{2} 
\Bigl( \lambda + \frac{c_{\infty} - c_{0} 
+ \eta^{-1}}{2(\mu - 1)} \Bigr) 
 - \Bigl( \lambda + \frac{c_{\infty} 
 - c_{0} + \eta^{-1}}
 {2(\mu - 1)} \Bigr)^{2} \mu  \Bigr\}, 
 \label{eq:Backlund2}
\end{cases}
\end{eqnarray}
is a solution of $(H_{\rm III'})_{D_{6}}$ 
with the parameter ${\bf c}$ is shifted by 
$T_{j}$ $(j = 1, 2)$; that is, 
it is a solution of 
\begin{eqnarray}
\frac{d\Lambda}{dt} = 
\eta \frac{\partial {\cal H}_{j}}{\partial M}, 
\hspace{+1.em} 
\frac{dM}{dt} = 
- \eta \frac{\partial {\cal H}_{j}}{\partial \Lambda}
\hspace{+1.em} (j = 1,2),
\label{eq:Hamiltonian}  
\end{eqnarray}
where ${\cal H}_{j} = {\cal H}_{j}(\Lambda,M,t)$ 
$(j = 1,2)$ is given by 
\begin{eqnarray}
t {\cal H}_{1} & = & \Lambda^{2} M^{2} 
- \bigl( \Lambda^{2} + c_{0} \Lambda - t \bigr) M 
+ \frac{1}{2} (c_{\infty} + c_{0} + \eta^{-1}) \Lambda,
\\[+1.em]
t {\cal H}_{2} & = & \Lambda^{2} M^{2} 
- \bigl( \Lambda^{2} + (c_{0} - 2 \eta^{-1}) \Lambda - t \bigr) M 
+ \frac{1}{2} (c_{\infty} + c_{0} - \eta^{-1}) \Lambda .
\end{eqnarray}
\end{lemm}

Lemma \ref{Lemma:Backlund} can be shown by 
straightforward computations. 
Next we consider a transseries solution 
\begin{equation}
\bigl( \lambda(t,{\bf c},\eta;\alpha), 
\mu(t,{\bf c},\eta;\alpha) \bigr) = 
\Bigl( \sum_{k \ge 0} (\alpha \eta^{-{1}/{2}})^{k} 
\lambda^{(k)}(t,\textbf{c},\eta) 
e^{k \eta \phi}, 
\sum_{k \ge 0} (\alpha \eta^{-{1}/{2}})^{k} 
\mu^{(k)}(t,\textbf{c},\eta) 
e^{k \eta \phi} \Bigr)
\label{eq:transseries solution of Hamiltonian system}
\end{equation}
of $(H_{\rm III'})_{D_{6}}$. 
Here the above transseries expansion of 
$\mu(t,{\bf c},\eta;\alpha)$ is obtained from 
that of $\lambda(t,{\bf c},\eta;\alpha)$ by the equality 
\begin{equation}
\mu = \frac{1}{2 \lambda^{2}} \Bigl( \eta^{-1} 
t \hspace{+.1em} \frac{d\lambda}{dt} + \lambda^{2} 
+ (c_{0} - \eta^{-1}) \hspace{+.1em} \lambda - t \Bigr).
\label{eq:mu}
\end{equation}
Especially, the formal power series 
$\mu^{(0)}(t,{\bf c},\eta)$ and its leading term 
$\mu_{0}(t,{\bf c})$ is given by 
\begin{eqnarray}
\mu^{(0)}(t,{\bf c},\eta) & = & \frac{1}{2 {\lambda^{(0)}}^{2}} 
\Bigl( \eta^{-1} t \hspace{+.1em} 
\frac{d\lambda^{(0)}}{dt} + {\lambda^{(0)}}^{2} 
+ (c_{0} - \eta^{-1}) \hspace{+.1em} \lambda^{(0)} 
- t \Bigr),  \label{eq:mu(0)} \\[+.3em]
\mu_{0}(t,{\bf c}) & = & \frac{1}{2} 
+ \frac{c_{0}}{2 \lambda_{0}} 
- \frac{t}{2\lambda_{0}^{2}}.
\label{eq:mu0}
\end{eqnarray}
Applying the above B\"acklund transformations 
to this transseries solution, we have the following:

\begin{lemm} \label{Lemma:difference eq for R}
Let $T_{1}({\bf c}) = (c_{\infty} + \eta^{-1}, c_{0} + \eta^{-1})$ 
and $T_{2}({\bf c}) = (c_{\infty} + \eta^{-1}, c_{0} - \eta^{-1})$, 
and $R = R(t,{\bf c},\eta)$ be a formal solution 
$R = R(t,{\bf c},\eta)$ of the Riccati equation 
\eqref{eq:Riccati equation}; that is, 
$R = R_{+}$ or $R_{-}$ in Remark \ref{Remark:odd part}.
Then we have the following:
\begin{eqnarray}
\displaystyle (i) \hspace{+.3em}
R(t,T_{1}({\bf c}),\eta) - R(t,{\bf c},\eta) 
& = & \frac{d}{dt} {\rm log} \hspace{+.2em} t 
\nonumber \\[+.3em]
+ \frac{d}{dt} {\rm log} \biggl( 
\frac{1}{{\lambda^{(0)}}^{2}} 
 \hspace{+3.3em} & \hspace{-7.3em} - & \hspace{-4.em}
\frac{(c_{\infty} + c_{0} + \eta^{-1}) 
\bigl( 4 \lambda^{(0)} (\mu^{(0)} - 1)  
+ (c_{\infty} - c_{0} + \eta^{-1}) +
2{\lambda^{(0)}}^{2} X \bigr)}
{(2{\lambda^{(0)}}^{2}(\mu^{(0)}-1) 
 + (c_{\infty} - c_{0} + \eta^{-1}) \lambda^{(0)}
 + 2t )^{2}}
\biggr) . \nonumber \\[+1.em]
&  &  \label{eq:difference eq for R1} \\[+.5em]
\displaystyle (ii) \hspace{+.3em}
R(t,T_{2}({\bf c}),\eta) - R(t,{\bf c},\eta) 
& = & \frac{d}{dt} {\rm log} \hspace{+.2em} t + 
\frac{d}{dt} {\rm log} \bigl( -2(\mu^{(0)}-1)^{2} 
+ (c_{\infty} - c_{0} + \eta^{-1}) X \bigr) 
\nonumber \\[+.5em]
&   & \hspace{+.5em}
- 2 \frac{d}{dt} {\rm log} 
\bigl(  2\lambda^{(0)}(\mu^{(0)} - 1)
+ (c_{\infty} - c_{0} + \eta^{-1})  \bigr).
\label{eq:difference eq for R2}
\end{eqnarray}
Here $X = X(t,{\bf c},\eta)$ is a formal power 
series given by 
\begin{equation}
X(t,{\bf c},\eta) = 
\frac{\eta^{-1} t \hspace{+.1em} R}
{2{\lambda^{(0)}}^{2}} 
- \frac{\eta^{-1} t}{{\lambda^{(0)}}^{3}} 
\frac{d\lambda^{(0)}}{dt}
- \frac{c_{0} - \eta^{-1}}{2{\lambda^{(0)}}^{2}}
+ \frac{t}{{\lambda^{(0)}}^{3}}.
\label{eq:X}
\end{equation}
\end{lemm}

\begin{proof}
Here we only show the equality \eqref{eq:difference eq for R2}.
Applying the B\"aclund transformation 
\eqref{eq:Backlund2} to the transseries solution, 
we have the following transseries solution of the 
$(P_{\rm III'})_{D_{6}}$ with the parameter ${\bf c}$ 
is shifted by $T_{2}$
(the equation is denoted by $T_{2}(P_{\rm III'})_{D_{6}}$ 
in what follows):
\begin{equation}
\Lambda_{2}(\lambda,\mu) = 
\Lambda^{(0)}(t,{\bf c},\eta) 
+ \alpha \eta^{-1/2} \Lambda^{(1)}(t,{\bf c},\eta) e^{\eta \phi}
+ (\alpha \eta^{-1/2})^{2} \Lambda^{(2)}(t,{\bf c},\eta) 
e^{2 \eta \phi} + \cdots, 
\label{eq:Backlund transformed transseries}
\end{equation}
where $\Lambda^{(k)}(t,{\bf c},\eta)$ is a 
formal power series in $\eta^{-1}$ ($k \ge 0$).
Especially, $\Lambda^{(0)}(t,{\bf c},\eta) = 
\Lambda_{2}(\lambda^{(0)},\mu^{(0)})$ 
is a 0-parameter solution of 
$T_{2}(P_{\rm III'})_{D_{6}}$, and 
\begin{equation}
\Lambda^{(1)}(t,{\bf c},\eta) = 
\frac{-4t(\mu^{(0)}-1)^{2} \lambda^{(1)}
+ 2t (c_{\infty} - c_{0} + \eta^{-1}) \mu^{(1)}}
{\bigl( 2\lambda^{(0)}(\mu^{(0)} - 1) 
+ (c_{\infty} - c_{0} + \eta^{-1}) \bigr)^{2}}.
\label{eq:Backlund transformed first part}
\end{equation}
Note that, since we can easily check that  
\begin{equation}
\frac{2t(\mu_{0} - 1)}{2\lambda_{0}
(\mu_{0} - 1) + (c_{\infty} - c_{0})} 
 = \lambda_{0}
\end{equation}
holds, we have 
$ \Lambda^{(0)}(t,{\bf c},\eta) = 
\lambda^{(0)}(t,T_{2}({\bf c}),\eta) $
due to the uniqueness of the 0-parameter solution 
of $T_{2}(P_{\rm III'})_{D_{6}}$. Hence, as explained in 
Section \ref{section:1-parameter solutions}, 
the formal power series 
\eqref{eq:Backlund transformed first part} 
is expressed as 
\begin{equation}
\Lambda^{(1)}(t,{\bf c},\eta) e^{\eta \phi}
= C(\eta) {\rm exp} \Bigl( 
\int^{t} R(t,T_{2}({\bf c}),\eta) 
\hspace{+.1em} dt \Bigr)
\label{eq:expression sono1}
\end{equation}
with a formal power series $C(\eta)$ whose coefficients 
are independent of $t$. On the other hand, 
since $\mu^{(1)}$ can be written as 
\begin{equation}
\mu^{(1)}(t,{\bf c},\eta) = 
X(t,{\bf c},\eta) \lambda^{(1)}(t,{\bf c},\eta)
\label{eq:mu(1)}
\end{equation}
by \eqref{eq:mu} and \eqref{eq:tilde lambda (1)},  
the formal power series 
\eqref{eq:Backlund transformed first part}
also has the following expression:
\begin{equation}
\Lambda^{(1)}(t,{\bf c},\eta) e^{\eta \phi} =
C'(\eta) \hspace{+.1em}
\frac{t \hspace{+.1em} \big( -2(\mu^{(0)}-1)^{2} 
+ (c_{\infty} - c_{0} + \eta^{-1}) X \bigr)}
{\bigl( 2\lambda^{(0)}(\mu^{(0)} - 1)
+ (c_{\infty} - c_{0} + \eta^{-1}) \bigr)^{2}}
\hspace{+.1em} {\rm exp} \Bigl( 
\int^{t} R(t,{\bf c},\eta) 
\hspace{+.1em} dt \Bigr).
\label{eq:expression sono2}
\end{equation}
Here $C'(\eta)$ is a formal power series with 
constant (with respect to $t$) coefficients.
Comparing \eqref{eq:expression sono1} and 
\eqref{eq:expression sono2}, and taking logarithmic 
derivatives, we obtain \eqref{eq:difference eq for R2}. 
The equality \eqref{eq:difference eq for R1} can 
be derived in the completely same manner.
\end{proof}

Lemma \ref{Lemma:difference eq for R} arrows us to 
compute integrals of 
$R(t,T_{j}({\bf c}),\eta) - R(t,{\bf c},\eta)$ 
($j = 1,2$) explicitly. 
Using this lemma, the difference equation 
satisfied by the Voros coefficient 
$W_{\Gamma}({\bf c},\eta)$ can be derived.
We can derive difference equations 
for all other Voros coefficients in the same manner.

\begin{lemm} \label{Lemma:difference eq for W-Gamma}
The Voros coefficient $W_{\Gamma}({\bf c},\eta)$
defined by \eqref{eq:def of P-Voros coeff W-Gamma}
satisfies the following difference equations:
\begin{eqnarray}
(i) \hspace{+.2em} 
W_{\Gamma}(T_{1}({\bf c}), \eta) - 
W_{\Gamma}({\bf c},\eta) & = & 0, 
\label{eq:difference eq for W-Gamma1} \\[+.3em]
(ii) \hspace{+.2em} 
W_{\Gamma}(T_{2}({\bf c}), \eta) - 
W_{\Gamma}({\bf c},\eta) & = & 
1 - (c_{m} \eta + 1 ) \hspace{+.1em} {\rm log} 
\Bigl( 1 + \frac{1}{c_{m} \eta} \Bigr) 
+ {\rm log} \Bigl( 1 + \frac{1}{2c_{m}\eta} \Bigr) ,
\nonumber \\
&  & 
\label{eq:difference eq for W-Gamma} 
\end{eqnarray}
where $c_{m}$ is given by \eqref{eq:c-plus-minus}.
\end{lemm}

\begin{proof}
We only derive the difference equation 
\eqref{eq:difference eq for W-Gamma} here. 
We introduce the following formal power series 
defined by
\begin{eqnarray}
I_{\pm}(t,{\bf c},\eta) & = & 
\int_{\Gamma_{t}} R_{\pm}(t,T_{2}({\bf c}),\eta) \hspace{+.2em} dt  
- \int_{\Gamma_{t}}R_{\pm}(t,{\bf c},\eta) 
 \hspace{+.2em} dt  \hspace{+.2em},  \label{eq:I}  \\[+.3em]
I_{k}(t,{\bf c},\eta) & = & 
\int_{\Gamma_{t}}R_{k}(t,T_{2}({\bf c}) )\hspace{+.2em} dt  
- \int_{\Gamma_{t}}R_{k}(t,{\bf c}) \hspace{+.2em} dt   
\hspace{+1.em}(k \ge -1)  \hspace{+.2em}, 
\label{eq:I_k}
\end{eqnarray}
where $\Gamma_{t}$ is a path shown in Figure 
\ref{fig:P3D6-Gamma-t-sheet-1} and 
\ref{fig:P3D6-Gamma-t-sheet-2}.
Then, it follows from the definition \eqref{eq:R-odd} 
of $R_{\rm odd}$ that  
\begin{equation}
W_{\Gamma}(T_{2}({\bf c}),\eta) - 
W_{\Gamma}({\bf c},\eta) 
 = \frac{1}{2} 
\lim_{t \rightarrow \infty_{3,+}} 
\biggl( \frac{I_{+}(t,{\bf c},\eta) 
- I_{-}(t,{\bf c},\eta) }{2} 
 - \eta \hspace{+.1em} I_{-1}(t,{\bf c},\eta) \biggr),
\label{eq:difference limit}
\end{equation}
where the limit is taken along the 
Stokes curve $\Gamma$  in Figure \ref{fig:PIII',0}. 
Using Lemma \ref{Lemma:difference eq for R}, 
we have 
\begin{equation}
 \frac{I_{+}(t,{\bf c},\eta) 
- I_{-}(t,{\bf c},\eta) }{2} 
 = \hspace{+.2em} {\rm log} \biggl( 
\frac{-2(\mu^{(0)} - 1)^{2} 
+ (c_{\infty} - c_{0} + \eta^{-1}) X_{+}}
{-2(\mu^{(0)} - 1)^{2} 
+ (c_{\infty} - c_{0} + \eta^{-1}) X_{-}} \biggr), 
\label{eq:I}
\end{equation}
where $X_{\pm}$ is a formal power series 
defined by taking $R = R_{\pm}$ 
in \eqref{eq:X}. Then, it follows from 
\eqref{eq:behavior of lambda(0) infinity-minus},
\eqref{eq:behavior of mu(0) infinity-minus} 
and \eqref{eq:behavior of R infinity-minus}
in Appendix \ref{Appendix:asymptotics} that 
\begin{equation}
 \frac{I_{+}(t,{\bf c},\eta) 
- I_{-}(t,{\bf c},\eta) }{2} 
 = 2 \hspace{+.2em} {\rm log} \biggl( 1 + 
\frac{1}{(c_{\infty}-c_{0})\eta}
\biggr)
+ {\rm log} \biggl( 
\frac{-(c_{\infty} - c_{0})^{2}}
{256 t} \biggr) 
+ O(t^{-1/2})
\label{eq:behavior of I}
\end{equation}
as $t \rightarrow \infty_{3,+}$, 
with suitable branches of the logarithm.
Moreover, we can compute the integral of 
$R_{-1}$ explicitly as 
\begin{eqnarray}
\int_{\Gamma_{t}} R_{-1}(t,{\bf c}) \hspace{+.1em} dt 
& = & 
4 t R_{-1} - c_{\infty} \hspace{+.2em} {\rm log} 
\biggl( 
\frac{2\lambda_{0} - c_{\infty} + t R_{-1}}
{2\lambda_{0} - c_{\infty} - t R_{-1}}  
\biggr)
- c_{0} \hspace{+.2em} {\rm log}
\biggl( \frac{ 2t^{2} - c_{0} t \lambda_{0} 
+ t^{2} \lambda_{0} R_{-1} }
{ 2t^{2} - c_{0} t \lambda_{0} 
- t^{2} \lambda_{0} R_{-1} } 
\biggr). \nonumber \\ & & 
\label{eq:the integral of R-1}
\end{eqnarray}
The equality \eqref{eq:the integral of R-1}
follows from the relationship 
between the Painlev\'e equations 
and associated isomonodromic deformation 
(cf.\ \cite{Jimbo-Miwa}) of linear differential equations. 
Let $Q_{0}(x,t,{\bf c})$ be a rational function 
\[
Q_{0}(x,t,{\bf c}) = \frac{(x-\lambda_{0})^{2}}{4x^{4}} 
P(x,t,{\bf c}), \hspace{+1.em}
P(x,t,{\bf c}) = x^{2} + 2(\lambda_{0}-c_{\infty})x 
+ \frac{t^{2}}{\lambda_{0}^{2}}, 
\]
which is the leading term of the potential function 
of a Schr\"odinger equation relevant to 
$(P_{\rm III'})_{D_{6}}$ (cf.\ \cite{KT iwanami}), 
and $x = a(t,{\bf c})$ be a zero of $P(x,t,{\bf c})$.
Then, it is shown in \cite{KT iwanami} that 
\begin{equation}
\int_{a(t,{\bf c})}^{\lambda_{0}(t,{\bf c})}
\sqrt{Q_{0}(x,t,{\bf c})} = \frac{1}{2}
\int_{\tau}^{t} R_{-1}(t,{\bf c}) dt. 
\label{eq:integral relation}
\end{equation}
Therefore, the equality \eqref{eq:the integral of R-1} 
follows from the equalities \eqref{eq:integral relation}
and the explicit computation of the integral of 
$\sqrt{Q_{0}(x,t,{\bf c})}$. Thus we have 
\begin{eqnarray}
I_{-1}(t,{\bf c},\eta) & = & - 2 \eta^{-1}
+ (c_{\infty} - c_{0} + 2 \eta^{-1}) \hspace{+.1em} 
{\rm log} \biggl( 1 + \frac{2}{(c_{\infty} - c_{0})\eta}
\biggr)  \nonumber \\[+.3em] 
&  & 
+ \eta^{-1} {\rm log} \biggl( 
\frac{-(c_{\infty} - c_{0})^{2}}
{256 \hspace{+.1em} t} \biggr)
+ O(t^{-1/2}), \label{eq:behavior of I_-1} 
\end{eqnarray}
as $t \rightarrow \infty_{3,+}$, 
with suitable branchs of the logarithm.
As a result, the difference equation 
\eqref{eq:difference eq for W-Gamma}
follows from \eqref{eq:difference limit}, 
\eqref{eq:behavior of I} and 
\eqref{eq:behavior of I_-1} directly. 
Here we note that, since any Voros coefficients 
has the form 
\[
\sum_{n \ge 1} \eta^{1-2n} W_{2n-1}({\bf c}), 
\hspace{+1.em} \text{$W_{2n-1}(\bf c) \in {\mathbb C}$}
\]
(i.e., a formal power series in $\eta^{-1}$ 
without constant terms), the branches of logarithms in 
\eqref{eq:behavior of I} and \eqref{eq:behavior of I_-1} 
must coincide, otherwise we have a contradiction; 
there appears a term of $\eta^{0}$ in the 
right-hand side of the difference equation
\eqref{eq:difference eq for W-Gamma}.
The difference equation \eqref{eq:difference eq for W-Gamma1} 
can be derived by the completely same manner.
\end{proof}

The shift operators $T_{1}$ and $T_{2}$ 
induce $T_{1}(c_{p}, c_{m}) = (c_{p} + \eta^{-1}, c_{m})$ 
and $T_{2}(c_{p}, c_{m}) = (c_{p}, c_{m} + \eta^{-1})$.
Therefore, the first difference equation 
\eqref{eq:difference eq for W-Gamma1} 
implies that $W_{\Gamma}$ does not depend on $c_{p}$; 
that is, $W_{\Gamma}$ is a formal power series of $\eta^{-1}$
whose coefficients depend only on $c_{m}$.
Moreover, according to the homogeneity property 
\eqref{eq:degree of W}, $W_{\Gamma}$ can be expressed as 
\begin{equation}
W_{\Gamma}({\bf c},\eta) = \sum_{n \ge 1} 
w_{2n - 1} \hspace{+.1em}
(c_{m} \hspace{+.1em} \eta)^{1-2n}, 
\label{eq:W-Gamma pre-expression}
\end{equation}
with some $w_{2n-1} \in {\mathbb C}$ which is independent 
of both $\eta$ and $c_{m}$ ($n \ge 1$). 
Then, \eqref{eq:difference eq for W-Gamma} 
and \eqref{eq:W-Gamma pre-expression} implies 
that $W_{\Gamma}$ has the explicit representation 
\eqref{eq:Main Theorem} by (ii) of Lemma \ref{key lemma}.
Thus we have proved Theorem \ref{Main Theorem}. 
$\Box$  \vspace{+.5em}

As well as the above proof, we can obtain an explicit 
representations of Voros coefficients from difference 
equations satisfied by them, and these difference equations 
can be derived in the same manner. 
The list of asymptotic behaviors of $\lambda^{(0)}$, 
$\mu^{(0)}$ and $R_{\pm}$ when $t$ tends to infinity, 
which are necessary for derivations of 
the difference equations, are summarized 
in Appendix \ref{Appendix:asymptotics}. 
Here we show the difference equations 
satisfied by the Voros coefficients 
$W_{\infty}$ for any choice of $\infty$. 

\begin{lemm} \label{Lemma:difference eq for W for infinity}
(i) The Voros coefficient 
$W_{\pm}({\bf c}, \eta) = 
W_{\infty_{j, \pm}}({\bf c},\eta)$
(for $j = 1,2$) satisfies the following 
difference equations: 
\begin{eqnarray}
W_{\pm}(T_{1}({\bf c}), \eta) - W_{\pm}({\bf c},\eta) 
& = & 
\pm \Bigl\{ 1 
- \bigl( c_{p} \eta + 1 \bigr) {\rm log} 
\Bigl( 1 + \frac{1}{c_{p} \eta} \Bigr) 
+ {\rm log} \Bigl( 1 + \frac{1}{2 c_{p} \eta} \Bigr) 
\Bigr\}, \nonumber  \\[+.3em] \\[+.3em]
W_{\pm}(T_{2}({\bf c}), \eta) 
- W_{\pm}({\bf c},\eta) & = & 0.
\end{eqnarray}

\vspace{+.3em}
\noindent 
(ii) The Voros coefficient 
$W_{\pm}({\bf c}, \eta) = 
W_{\infty_{j, \pm}}({\bf c},\eta)$
(for $j = 3,4$) satisfies the following 
difference equations: 
\begin{eqnarray}
W_{\pm}(T_{1}({\bf c}), \eta) 
- W_{\pm}({\bf c},\eta) & = & 0, \\[+.3em]
W_{\pm}(T_{2}({\bf c}), \eta) 
- W_{\pm}({\bf c},\eta) & = & 
\pm \Bigl\{ 1 
- \bigl( c_{m} \eta + 1 \bigr) {\rm log} 
\Bigl( 1 + \frac{1}{c_{m} \eta} \Bigr) 
+ {\rm log} \Bigl( 1 + \frac{1}{2 c_{m} \eta} \Bigr) 
\Bigl\}. \nonumber  \\[+.2em] 
\end{eqnarray}
Here $c_{p}$ and $c_{m}$ are given by 
\eqref{eq:c-plus-minus}.
\end{lemm}

The explicit representations of Voros coefficients 
in Theorem \ref{Main Theorem3} are obtained from these 
difference equations and Lemma \ref{key lemma} directly. 
Thus we have proved Theorem \ref{Main Theorem3}. $\Box$

\subsection{Proof of Theorem \ref{Main Theorem2}}
\label{subsection:Proof of Main Theorem2}

Next we show Theorem \ref{Main Theorem2} about the explicit 
representations of the Voros coefficients 
$W_{0_{c_{\infty},\pm}}$ and 
$W_{0_{c_{0},\pm}}$ for double-poles 
defined by \eqref{eq:P-Voros coeff for 0}.
Using Lemma \ref{Lemma:difference eq for R}
and the list of asymptotic behaviors when $t$ tends to 
double-poles which are summarized in 
Appendix \ref{Appendix:asymptotics},
we can derive the following difference equations in the 
completely same method in 
Section \ref{subsection:Proof of Main Theorem1}.

\begin{lemm} \label{Lemma:difference eq for W for 0}
(i) The Voros coefficient 
$W_{c_{\infty}, \pm}({\bf c}, \eta) = 
W_{0_{c_{\infty},\pm}}({\bf c},\eta)$
satisfies the following difference equations:
\begin{eqnarray}
W_{c_{\infty},\pm}(T_{1}({\bf c}), \eta) 
- W_{c_{\infty},\pm}({\bf c},\eta) & = & \pm \Bigl\{  -2 
- \bigl(c_{p} \eta + 1\bigr) {\rm log} 
\Bigl( 1 + \frac{1}{c_{p} \eta} \Bigr) 
+ {\rm log} \Bigl( 1 + \frac{1}{2 c_{p} \eta} \Bigr) 
\nonumber  \\[+.3em] 
&  & 
+ 3 \Bigl( c_{\infty}\eta + \frac{1}{2} \Bigr) 
{\rm log} \Bigl( 1 + \frac{1}{c_{\infty} \eta} \Bigr) 
\Bigr\},  \\[+.5em]
W_{c_{\infty},\pm}(T_{2}({\bf c}), \eta) 
- W_{c_{\infty},\pm}({\bf c},\eta) & = & \pm \Bigl\{ - 2 
- \bigl(c_{m} \eta + 1\bigr) {\rm log} 
\Bigl( 1 + \frac{1}{c_{m} \eta} \Bigr) 
+ {\rm log} \Bigl( 1 + \frac{1}{2 c_{m} \eta} \Bigr) 
\nonumber  \\[+.3em] 
&  & 
+ 3 \Bigl( c_{\infty}\eta + \frac{1}{2} \Bigr) 
{\rm log} \Bigl( 1 + \frac{1}{c_{\infty} \eta} \Bigr) 
\Bigr\}.
\end{eqnarray}

\noindent
(ii) The Voros coefficient 
$W_{c_{0},\pm}({\bf c}, \eta) = 
W_{0_{c_{0},\pm}}({\bf c},\eta)$
satisfies the following difference equations:
\begin{eqnarray}
W_{c_{0},\pm}(T_{1}({\bf c}), \eta) 
- W_{c_{0},\pm}({\bf c},\eta) & = & \pm \Bigl\{ -2 
- \bigl(c_{p} \eta + 1\bigr) {\rm log} 
\Bigl( 1 + \frac{1}{c_{p} \eta} \Bigr) 
+ {\rm log} \Bigl( 1 + \frac{1}{2 c_{p} \eta} \Bigr) 
\nonumber  \\[+.3em] 
&  & 
+ 3 \Bigl( c_{0}\eta + \frac{1}{2} \Bigr) 
{\rm log} \Bigl( 1 + \frac{1}{c_{0} \eta} \Bigr)
\Bigr\}, \\[+.5em]
W_{c_{0},\pm}(T_{2}({\bf c}), \eta) 
- W_{c_{0},\pm}({\bf c},\eta) & = & \pm \Bigl\{ 2 
+ \bigl(c_{m} \eta + 1\bigr) {\rm log} 
\Bigl( 1 + \frac{1}{c_{m} \eta} \Bigr) 
- {\rm log} \Bigl( 1 + \frac{1}{2 c_{m} \eta} \Bigr) 
\nonumber  \\[+.3em] 
&  & 
+ 3 \Bigl( c_{0}\eta - \frac{1}{2} \Bigr) 
{\rm log} \Bigl( 1 - \frac{1}{c_{0} \eta} \Bigr) 
\Bigr\}. 
\end{eqnarray}
Here $c_{p}$ and $c_{m}$ are given by 
\eqref{eq:c-plus-minus}.
\end{lemm}

It follows from Lemma \ref{key lemma} and 
Lemma \ref{Lemma:difference eq for W for 0} 
that, the formal power series 
\begin{eqnarray*}
\tilde{W}_{c_{\infty}, \pm}({\bf c},\eta) & = & 
W_{c_{\infty},\pm}({\bf c},\eta) 
\mp \bigl( {\cal F}(c_{p},\eta) 
+ {\cal F}(c_{m},\eta) \bigr), \\[+.3em]
\tilde{W}_{c_{0},\pm}({\bf c},\eta) & = &  
W_{c_{0},\pm}({\bf c},\eta) 
\mp \bigl( {\cal F}(c_{p},\eta) 
- {\cal F}(c_{m},\eta) \bigr), 
\end{eqnarray*}
satisfies the following equations:
\begin{eqnarray}
\hspace{-1.em}
\tilde{W}_{c_{\infty},\pm}(T_{1}({\bf c}), \eta) 
- \tilde{W}_{c_{\infty},\pm}({\bf c},\eta) & = & 
\pm \Bigl\{ - 3 
+ 3 \Bigl( c_{\infty}\eta + \frac{1}{2} \Bigr) 
{\rm log} \Bigl( 1 + \frac{1}{c_{\infty} \eta} \Bigr) 
\Bigr\}, \label{eq:tilde-diff-1} \\[+.5em]
\hspace{-1.em}
\tilde{W}_{c_{\infty},\pm}(T_{2}({\bf c}), \eta) 
- \tilde{W}_{c_{\infty},\pm}({\bf c},\eta) & = &
\pm \Bigl\{ - 3 
+ 3 \Bigl( c_{\infty}\eta + \frac{1}{2} \Bigr) 
{\rm log} \biggl( 1 + \frac{1}{c_{\infty} \eta} \biggr)
\Bigr\}, \label{eq:tilde-diff-2} \\[+.5em]
\hspace{-1.em}
\tilde{W}_{c_{0},\pm}(T_{1}({\bf c}), \eta) 
- \tilde{W}_{c_{0},\pm}({\bf c},\eta) & = & 
\pm \Bigl\{ - 3 
+ 3 \Bigl( c_{0}\eta + \frac{1}{2} \Bigr) 
{\rm log} \biggl( 1 + \frac{1}{c_{0} \eta} \biggr) 
\Bigr\}, \label{eq:tilde-diff-3} \\[+.5em]
\hspace{-1.em}
\tilde{W}_{c_{0}, \pm}(T_{2}({\bf c}), \eta) 
- \tilde{W}_{c_{0}, \pm}({\bf c},\eta) & = & 
\pm \Bigl\{ 3 + 3 \Bigl( c_{0}\eta - \frac{1}{2} \Bigr) 
{\rm log} \biggl( 1 - \frac{1}{c_{0} \eta} \biggr) \Bigr\}. 
\label{eq:tilde-diff-4}
\end{eqnarray}
We can confirm that $\tilde{W}_{c_{\infty}, \pm}$ 
and $\tilde{W}_{c_{0}, \pm}$ satisfy
\[
\tilde{W}_{c_{\infty},\pm}(c_{\infty}, c_{0}+\eta^{-1})
- \tilde{W}_{c_{\infty},\pm}(c_{\infty},c_{0}-\eta^{-1}) = 0
\]
\[
\tilde{W}_{c_{0},\pm}(c_{\infty}+\eta^{-1},c_{0})
- \tilde{W}_{c_{0},\pm}(c_{\infty}-\eta^{-1},c_{0}) = 0
\]
by combinig the equations 
\eqref{eq:tilde-diff-1} $\sim$ \eqref{eq:tilde-diff-4}. 
This implies that $\tilde{W}_{c_{\infty}, \pm}$ 
(resp., $\tilde{W}_{c_{0}, \pm}$)
does not depend on $c_{0}$ (resp., $c_{\infty}$). 
As a result, we have  
\begin{eqnarray}
\tilde{W}_{c_{\infty}\pm}({\bf c},\eta) & = & 
\mp 3 \hspace{+.1em} {\cal G}(c_{\infty},\eta), \\[+.3em]
\tilde{W}_{c_{0},\pm}({\bf c},\eta) & = & 
\mp 3 \hspace{+.1em} {\cal G}(c_{0},\eta),
\end{eqnarray}
due to (ii) of Lemma \ref{key lemma} and \eqref{eq:degree of W}. 
Thus we have proved Theorem \ref{Main Theorem2}.
$\Box$

\section{Connection formulas for parametric Stokes phenomena}
\label{section:connection formulas}

In this section we analyze parametric Stokes phenomena 
relevant to the degeneration observed in 
Section \ref{section:P-Stokes geometry},
and derive connection formulas in some cases. 

\subsection{Borel sum of Voros coefficients and 
their jump property}

First we note that the Borel sum 
(as a formal power series in $\eta^{-1}$)
of Voros coefficients can be computed explicitly. 
(See \cite[$\S$1]{KT iwanami} for the definition 
of Borel sums of formal power series.)

\begin{prop}[cf.\ {\cite[$\S2$]{Takei Sato conjecture}}, 
\cite{Koike-Takei}]
\label{Prop:Borel sum of F and G}
Let ${\cal F}(c,\eta)$ and ${\cal G}(c,\eta)$ 
be the formal power series in $\eta^{-1}$
given by \eqref{eq:Voros-coeff-F} and \eqref{eq:Voros-coeff-G}.
They are not Borel summable when $c \in i {\mathbb R}$, 
and Borel summable otherwise. 
Moreover, the Borel sum 
${\cal S}_{\pm}[{\cal F}(c,\eta)]$ and 
${\cal S}_{\pm}[{\cal G}(c,\eta)]$ of 
the formal power series ${\cal F}(c,\eta)$
and ${\cal G}(c,\eta)$ when 
$\arg c = {\pi}/{2} \pm \delta$ 
for a sufficiently small positive number $\delta$
are given explicitly by the followings:
\begin{eqnarray}
{\mathcal S}_{-} \bigl[{\cal F}(c,\eta) \bigr]
& = & 
{\rm log} \hspace{+.1em}
\frac{\Gamma (c
\hspace{+.1em} \eta + 1/2)}{\sqrt{2 \pi}} 
- c \hspace{+.1em} \eta \hspace{+.2em} 
\Bigl( {\rm log} \bigl( c
\hspace{+.1em} \eta \bigr) - 1 \Bigr). 
\label{eq:Borel sum of F minus} \\[+1.em]
{\mathcal S}_{+} \bigl[{\cal F}(c,\eta) \bigr]
& = & 
- {\rm log} \frac{\Gamma (- c \hspace{+.1em} \eta + 1/2)}
{\sqrt{2 \pi}} - c \hspace{+.1em} \eta \hspace{+.2em} 
\Bigl( {\rm log} \bigl( c \hspace{+.1em} \eta \bigr) - 1 \Bigr) 
+ \pi i \hspace{+.1em} 
c \hspace{+.1em} \eta.
\label{eq:Borel sum of F plus}  \\[+1.em]
{\mathcal S}_{-} \bigl[{\cal G}(c,\eta) \bigr]
& = & 
{\rm log} \frac{\Gamma(c \hspace{+.1em} \eta)}{\sqrt{2 \pi}}
- c \hspace{+.1em} \eta \hspace{+.1em} \Bigl(
{\rm log}(c \hspace{+.1em} \eta) - 1 \Bigr) 
+ \frac{1}{2} \hspace{+.1em} {\rm log} (c \hspace{+.1em} \eta).
\label{eq:Borel sum of G minus}
\\[+1.em]
{\mathcal S}_{+} \bigl[{\cal G}(c,\eta) \bigr]
& = &  
- {\rm log} \frac{\Gamma(- c \hspace{+.1em} \eta)}{\sqrt{2 \pi}}
- c \hspace{+.1em} \eta \hspace{+.1em} \Bigl(
{\rm log}(c \hspace{+.1em} \eta) - 1 \Bigr) 
- \frac{1}{2} \hspace{+.1em} {\rm log} (c \hspace{+.1em} \eta) 
+ \pi i (c \eta + 1/2). \nonumber \\
\label{eq:Borel sum of G plus}
\end{eqnarray}
\end{prop}
As a corollary of Proposition \ref{Prop:Borel sum of F and G}, 
we have the following equalities.
\begin{cor} \label{corollary:jump property}
After the analytic continuation across 
the imaginary axis \{$\arg c = \pi/2$\}, 
the following holds:
\begin{eqnarray}
{\mathcal S}_{+} \bigl[ e^{{\cal F}(c,\eta)} \bigr] 
& = & (1 + e^{2\pi i c \eta}) \hspace{+.2em}
{\mathcal S}_{-} \bigl[ e^{{\cal F}(c,\eta)} \bigr].
\label{eq:PSP for F}  \\[+.2em]
{\mathcal S}_{+} \bigl[ e^{{\cal G}(c,\eta)} \bigr] 
& = & (1 - e^{2\pi i c \eta}) \hspace{+.2em}
{\mathcal S}_{-} \bigl[ e^{{\cal G}(c,\eta)} \bigr].
\label{eq:PSP for G} 
\end{eqnarray}
\end{cor}
These jump properties of the Voros coefficients 
are essential in the derivation of connection formulas 
describing the parametric Stokes phenomena. 

  \begin{figure}[h]
  \begin{center}
  \includegraphics[width=75mm]
  {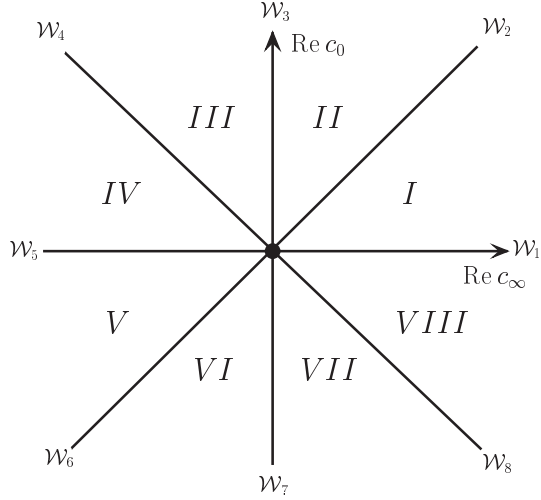}
  \end{center}
  \caption{\small{Projection of the walls and chambers in 
  the parameter space of ${\bf c} = (c_{\infty},c_{0})$.}}
  \label{fig:walls and chambers}
  \end{figure}

It follows from Proposition \ref{Prop:Borel sum of F and G}
and the explicit representations of Voros coefficients 
obtained by Theorem \ref{Main Theorem3} and 
Theorem \ref{Main Theorem2} that Voros coefficients 
are not Borel summable when the parameter 
${\bf c} = (c_{\infty},c_{0})$ lies on the ``walls'' 
${\cal W}_{1}, \cdots , {\cal W}_{8}$ in the parameter 
space defined by 
\begin{eqnarray*}
{\cal W}_{1} = \{ {\bf c} \in {\bf S}; 
{\rm Re} \hspace{+.2em} c_{0} \hspace{+.1em} = 0, 
{\rm Re} \hspace{+.2em} c_{\infty} > 0 \}, &  &
{\cal W}_{2}  =  \{ {\bf c} \in {\bf S}; 
{\rm Re} \hspace{+.2em} c_{m} = 0, 
{\rm Re} \hspace{+.2em} c_{p} > 0 \}, \\[+.2em]
{\cal W}_{3} = \{ {\bf c} \in {\bf S};  
{\rm Re} \hspace{+.2em} c_{\infty} = 0, 
{\rm Re} \hspace{+.2em} c_{0} > 0 \}, &  &
{\cal W}_{4}  =  \{ {\bf c} \in {\bf S}; 
{\rm Re} \hspace{+.2em} c_{p} = 0, 
{\rm Re} \hspace{+.2em} c_{m} < 0 \}, \\[+.2em]
{\cal W}_{5} = \{ {\bf c} \in {\bf S}; 
{\rm Re} \hspace{+.2em} c_{0} = 0, 
{\rm Re} \hspace{+.2em} c_{\infty} < 0 \}, &  &
{\cal W}_{6}  =  \{ {\bf c} \in {\bf S}; 
{\rm Re} \hspace{+.2em} c_{m} = 0, 
{\rm Re} \hspace{+.2em} c_{p} < 0  \}, \\[+.2em]
{\cal W}_{7} = \{ {\bf c} \in {\bf S}; 
{\rm Re} \hspace{+.2em} c_{0} = 0, 
{\rm Re} \hspace{+.2em} c_{\infty} < 0 \}, &  &
{\cal W}_{8}  =  \{ {\bf c} \in {\bf S}; 
{\rm Re} \hspace{+.2em} c_{p} = 0, 
{\rm Re} \hspace{+.2em} c_{m} > 0 \},
\end{eqnarray*}
where ${\bf S}$ is given by \eqref{eq:THE condition}, 
$c_{p}$ and $c_{m}$ are given by \eqref{eq:c-plus-minus}.
These walls divide the parameter space into 
eight chambers $I, \cdots, VIII$. 
Figure \ref{fig:walls and chambers} describes 
the projections of these walls and chambers to 
$({\rm Re} \hspace{+.2em} c_{\infty}, 
{\rm Re} \hspace{+.2em} c_{0})$-plane. 
The Voros coefficients are Borel summable in each chambers, 
and jump when the parameter cross these walls.
This jump property causes parametric Stokes phenomena. 
We show some examples of connection formulas 
on these walls in subsequent discussions.

\subsection{Connection problem for parametric 
Stokes phenomena relevant to the triangel-type degeneration}
\label{subsection:connection of triangle-type}

Here we discuss the connection problem 
for the parametric Stokes phenomenon 
relevant to the triangle-type degeneration of 
Stokes geometry observed when $\textbf{c} = (2, 2-i)$; 
that is, the connection problem 
on the wall ${\cal W}_{2}$ in Figure \ref{fig:walls and chambers}.
To be more specific, we discuss the connection problem 
for the transseries solution 
$\lambda_{\tau_{1}}(t,{\bf c},\eta;\alpha)$
normalized at the turning point $\tau_{1}$,  
when the independent variable $t$ is in a 
sufficiently small neighborhood of the point 
$t_{0}$ in Figure \ref{fig:normalization at infinity0} 
and $t_{1}$ in Figure \ref{fig:normalization at infinity1}.
These figures describe only Sheet 2 of the Riemann surface 
of $\lambda_{0}$; see Figure \ref{fig:P3D6-Gamma-t-sheet-1} 
$\sim$ \ref{fig:P3D6-Gamma-t-sheet-4}.  Note that $t_{0}$ 
(or the corresponding point on the $u$-plane) 
lies inside of the ``triangle'' in Figure 
\ref{fig:P3D6-u-plane-triangle-0} formed by three 
bounded Stokes curves, and $t_{1}$ lies outside 
of the triangle. Here we mean the triangle by 
``the triangle {\it on the $u$-plane} described in 
in Figure \ref{fig:P3D6-u-plane-triangle-0}'', which is not 
the triangle on the $t$-plane in Figure \ref{fig:PIII',0}.
(The inside of the triangle in Figure \ref{fig:PIII',0}
are not mapped to the inside of the triangle in 
Figure \ref{fig:P3D6-u-plane-triangle-0} by \eqref{eq:u-coordinate}.)

Before the discussion of connection problems, 
we impose the following assumptions 
(A-1) and (A-2) for a neighborhood $U_{t_{\ast}}$ 
of the point $t = t_{\ast}$ ($\ast = 0$ or $1$) 
and $\varepsilon > 0$:
\begin{itemize}
\item[(A-1)] 
There exists a neighborhood $U_{t_{\ast}}$
of the point $t_{\ast}$ such that 
\[
t \in U_{t_{\ast}} \hspace{+.3em} 
\Rightarrow \hspace{+.3em}
{\rm Re} \hspace{+.2em} \phi(t) < 0.
\]
\item[(A-2)]
The small number $\varepsilon > 0$ satisfies that, 
for any $0 \le \varepsilon' \le \varepsilon$, 
any Stokes curves never touch with 
any points in $U_{t_{\ast}}$ when 
${\bf c} = (2 \pm \varepsilon', 2 - i)$.
\end{itemize}
These assumptions are expected to be 
essential for the Borel summability 
of transseries solutions.
According to a recent result obtained by Kamimoto
(for the Borel summability of transseries solutions), 
we can expect the following.

\begin{conj} \label{conjecture:Kamimoto}
Assume that the integral of $R_{\rm odd}$ in 
\eqref{eq:first part of 1-parameter solution} 
is taken along a path which never 
touches with any turning points, 
the simple-pole and Stokes curves, 
and the real part of $\phi$ is negative. 
Then, the corresponding transseries solution 
$\lambda(t,{\bf c},\eta;\alpha)$ is Borel summable 
(in general sense of \cite{Costin}); that is, 
the $k$-th formal power series $\lambda^{(k)}(t,{\bf c},\eta)$ 
in $\lambda(t,{\bf c},\eta;\alpha)$ is Borel summable 
for each $k$, and the (generalized) Borel sum of 
$\lambda(t,{\bf c},\eta;\alpha)$ defined by 
the infinite sum
\[
{\cal S}[\lambda(t,{\bf c},\eta;\alpha)] = 
\sum_{k \ge 0} (\alpha \eta^{-1/2})^{k} 
{\cal S}[\lambda^{(k)}(t,{\bf c},\eta)] e^{k \eta \phi}
\]
converges for sufficiently large $\eta > 0$ and 
represents an analytic solution of $(P_{\rm III'})_{D_{6}}$.
Here ${\cal S}[\lambda^{(k)}(t,{\bf c},\eta)]$ 
is the Borel sum of the formal power series 
$\lambda^{(k)}(t,{\bf c},\eta)$.
\end{conj}

Kamimoto proved the same statement for the first and 
second Painlev\'e equation (\cite{Kamimoto}). 
In this paper we assume that the 
Conjecture \ref{conjecture:Kamimoto}
is true, and discuss connection problems 
for parametric Stokes phenomena. 

  \begin{figure}[h]
  \begin{minipage}{0.5\hsize}
  \begin{center}
  \includegraphics[width=52mm]
  {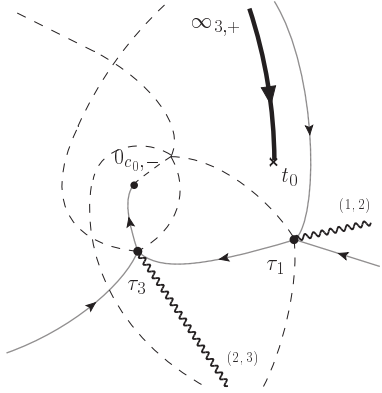}
  \end{center}
  \caption{\small{The path from $\infty_{3,+}$ 
  for $t \in U_{t_{0}}$.}}
  \label{fig:normalization at infinity0}
  \end{minipage} \hspace{+.5em}
  \begin{minipage}{0.5\hsize}
  \begin{center}
  \includegraphics[width=52mm]
  {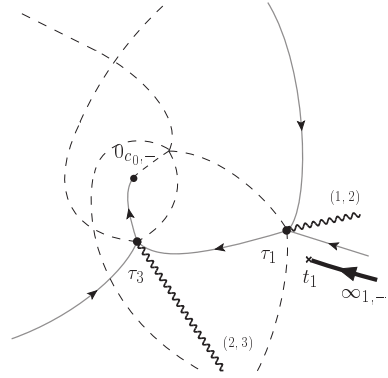}
  \end{center}
  \caption{\small{A path from $\infty_{1,+}$
  for $t \in U_{t_{1}}$.}}
  \label{fig:normalization at infinity1}
  \end{minipage}
  \end{figure}

First we discuss the connection problem 
for the parametric Stokes phenomenon 
in the case $t \in U_{t_{0}}$. Let 
\begin{eqnarray}
\lambda_{\infty}(t,{\bf c},\eta;\alpha) & = & 
\sum_{k \ge 0} (\alpha \eta^{-1/2})^{k} 
\lambda_{\infty}^{(k)}(t,{\bf c},\eta) e^{k \eta \phi}
\label{eq:normalization at infinity near t0}
\\[+.5em]
\biggl( {\rm resp}., \hspace{+.5em}
\lambda_{\tau_{1}}(t,{\bf c},\eta;\alpha) & = & 
\sum_{k \ge 0} (\alpha \eta^{-1/2})^{k} 
\lambda_{\tau_{1}}^{(k)}(t,{\bf c},\eta) e^{k \eta \phi}
\biggr) 
\label{eq:normalization at tau1 near t0}
\end{eqnarray}
be the transseries solution whose $1$-instanton part 
is normalized at infinity along the path shown in 
Figure \ref{fig:normalization at infinity0}
(resp., normalized at the turning point $\tau_{1}$ 
along the path $\Gamma_{t}$ in 
Figure \ref{fig:P3D6-Gamma-t-sheet-1} 
and \ref{fig:P3D6-Gamma-t-sheet-2}):
\begin{eqnarray*}
\alpha \eta^{-1/2} 
{\lambda}^{(1)}_{\infty}(t,{\bf c},\eta) e^{\eta \phi} 
& = & \alpha \frac{\lambda^{(0)}}{\sqrt{t R_{\rm odd}}} 
{\rm exp} \Bigl( \eta \int_{\tau_{1}}^{t} R_{-1} dt 
+ \int_{\infty_{3,+}}^{t} 
\bigl( R_{\rm odd} - \eta R_{-1} \bigr) dt \Bigr) 
\\[+.5em]
\biggl( {\rm resp}., \hspace{+.5em}
\alpha \eta^{-1/2} 
{\lambda}^{(1)}_{\tau_{1}}(t,{\bf c},\eta) e^{\eta \phi}
& = & \alpha \frac{\lambda^{(0)}}{\sqrt{t R_{\rm odd}}} 
{\rm exp} \Bigl( \int_{\tau_{1}}^{t} 
R_{\rm odd} \hspace{+.2em} dt \Bigr)  
\hspace{+.5em} \biggr).
\end{eqnarray*}
If Conjecture \ref{conjecture:Kamimoto} holds, 
for any $0 \le \varepsilon' \le \varepsilon$
the transseries solution 
$\lambda_{\infty}(t,{\bf c},\eta;\alpha)$ 
is Borel summable when $t \in U_{t_{0}}$ and  
${\bf c} = (2 \pm \varepsilon', 2 -i)$ 
since the real part of 
\[
\phi(t, {\bf c}) = \int_{\tau_{1}}^{t} R_{-1}(t,{\bf c}) dt 
= \int_{\tau_{1}}^{t} \sqrt{\Delta(t,{\bf c})} dt
\]
is negative by the assumption (A-1) and 
the normalization path in 
Figure \ref{fig:normalization at infinity0} 
can be deformed homotopically such that 
it does not touch with any turning points 
and Stokes curves by the assumptions (A-2).
On the other hand, using the relations
\begin{equation}
\lambda^{(k)}_{\tau_{1}}(t,{\bf c},\eta) = 
e^{k W_{\Gamma}({\bf c},\eta)} 
\lambda^{(k)}_{\infty}(t,{\bf c},\eta) 
\hspace{+1.em} (k \ge 0)
\label{eq:comparison of instanton}
\end{equation}
and 
\begin{equation}
W_{\Gamma}({\bf c},\eta) = {\cal F}(c_{m},\eta)
\label{eq:W-Gamma = F}
\end{equation}
(cf.\ Theorem \ref{Main Theorem})  
and Proposition \ref{Prop:Borel sum of F and G}, 
we can conclude that  
$\lambda_{\tau_{1}}(t,{\bf c},\eta;\alpha)$ is 
Borel summable 
when ${\bf c} = (2 \pm \varepsilon', 2 - i)$ for any 
$0 < \varepsilon' \le \varepsilon$, 
and the Borel sum has the analytic continuation 
with respect to the parameter ${\bf c}$ 
across the wall ${\cal W}_{2}$.

Now we derive the connection formula 
for the transseries solution $\lambda_{\tau_{1}}$. 
Let ${\cal S}_{II}
[\lambda_{\tau_{1}}(t,{\bf c},\eta; \alpha)]$ 
(resp., ${\cal S}_{I}
[\lambda_{\tau_{1}}(t,{\bf c},\eta; \tilde{\alpha})]$)
be the Borel sum of 
$\lambda_{\tau_{1}}(t,{\bf c},\eta; \alpha)$
when ${\bf c} = (2 - \varepsilon, 2 -i)$
(resp., of $\lambda_{\tau_{1}}(t,{\bf c},\eta; \tilde{\alpha})$ 
when ${\bf c} = (2 + \varepsilon, 2 -i)$) 
for a sufficiently small $\varepsilon > 0$. 
Moreover, we assume that, after the analytic continuation 
with respect to the parameter ${\bf c}$ across 
the wall ${\cal W}_{2}$, ${\cal S}_{II}[\lambda_{\tau_{1}}
(t,{\bf c},\eta; \alpha)]$ and ${\cal S}_{I}[\lambda_{\tau_{1}}
(t,{\bf c},\eta; \tilde{\alpha})]$ represent 
the same analytic solution of $(P_{\rm III'})_{D_{6}}$
defined on $U_{t_{0}}$. That is, we assume 
\begin{equation}
{\cal S}_{II}[\lambda_{\tau_{1}}(t,{\bf c},\eta; \alpha)] = 
{\cal S}_{I}[\lambda_{\tau_{1}}(t,{\bf c},\eta; \tilde{\alpha})].  
\label{eq:connection problem on B2-1}
\end{equation}
Taking into account the relation 
\eqref{eq:comparison of instanton}
and comparing the coefficients of 
${\rm exp}(k \eta \phi)$ $(k \ge 0)$ 
in \eqref{eq:connection problem on B2-1}, 
we have 
\begin{equation}
\alpha^{k} \hspace{+.1em}
{\mathcal S}_{II}\bigl[ e^{k W_{\Gamma}(\textbf{c},\eta)} 
\hspace{+.1em} {\lambda}_{\infty}^{(k)}(t,\textbf{c},\eta) \bigr]
= 
\tilde{\alpha}^{k} \hspace{+.1em} 
{\mathcal S}_{I}\bigl[ e^{k W_{\Gamma}(\textbf{c},\eta)} 
\hspace{+.1em} {\lambda}_{\infty}^{(k)}(t,\textbf{c},\eta) \bigr].
\label{eq:relation of first parts}
\end{equation}
The Borel summability of 
$\lambda_{\infty}^{(k)}(t,{\bf c},\eta)$ 
when ${\bf c} = (2, 2-i)$ implies that 
\begin{equation}
{\mathcal S}_{II}\bigl[{\lambda}_{\infty}^{(k)}
(t,\textbf{c},\eta)\bigr] 
= {\mathcal S}_{I}\bigl[{\lambda}_{\infty}^{(k)}
(t,\textbf{c},\eta)\bigr] 
\label{eq:no parametric Stokes phenomena for this normalization}
\end{equation}
holds for all $k \ge 0$ after the analytic continuation 
across ${\cal W}_{2}$. Therefore it follows from 
\eqref{eq:relation of first parts} and 
\eqref{eq:no parametric Stokes phenomena for this normalization} 
that 
\begin{equation} 
\alpha^{k}  \hspace{+.1em} 
{\mathcal S}_{II}\bigl[e^{k W_{\Gamma}(\textbf{c},\eta)} \bigr]
= \tilde{\alpha}^{k} \hspace{+.1em}
{\mathcal S}_{I}\bigl[e^{k W_{\Gamma}(\textbf{c},\eta)} \bigr]
\label{eq:relation of parameters}
\end{equation}
holds for all $k \ge 0$. 
Moreover, \eqref{eq:W-Gamma = F} and 
Corollary \ref{corollary:jump property} implies that 
\begin{eqnarray}
{\mathcal S}_{II}\bigl[e^{W_{\Gamma}(\textbf{c},\eta)} \bigr]
& = & (1 + e^{2 \pi i c_{m} \eta}) \hspace{+.2em}
{\mathcal S}_{I}\bigl[e^{W_{\Gamma}(\textbf{c},\eta)} \bigr] 
\nonumber \\[+.3em]
& = & 
(1 + e^{\pi i (c_{\infty} - c_{0}) \eta}) \hspace{+.2em}
{\mathcal S}_{I}\bigl[e^{W_{\Gamma}(\textbf{c},\eta)} \bigr].
\label{eq:connection formula for W-Gamma}
\end{eqnarray}
Thus, \eqref{eq:relation of parameters} 
is true if and only if the parameters satisfy 
the relation 
\begin{equation}
\tilde{\alpha} = (1 + e^{\pi i (c_{\infty} - c_{0}) \eta}) 
\hspace{+.1em} \alpha. 
\label{eq:connection formula at t = t0 for parameters}
\end{equation}
Thus we obtain the following connection formula.\\[-.4em]

\noindent
\textbf{Connection formula on ${\cal W}_{2}$ near $t = t_{0}$.} 
\hspace{+.5em}
\textit{Assume that a neighborhood $U_{t_{0}}$ of the point 
$t_{0}$ and a small number $\varepsilon > 0$ 
satisfies the assumptions {\rm (A-1)} and {\rm (A-2)}.
If the Borel sum of 
$\lambda_{\tau_{1}}(t,{\bf c},\eta;\alpha)$ 
when ${\bf c} = (2 - \varepsilon, 2 -i)$ and 
that of $\lambda_{\tau_{1}}(t,{\bf c},\eta;\tilde{\alpha})$ 
when ${\bf c} = (2 + \varepsilon, 2 -i)$ represent 
the same analytic solution of 
$(P_{\rm III'})_{D_{6}}$ defined on 
on $t \in U_{t_{0}}$ after the analytic continuation 
with respect to the parameter ${\bf c}$ across ${\cal W}_{2}$, 
then the parameters $\alpha$ and $\tilde{\alpha}$ satisfy 
\eqref{eq:connection formula at t = t0 for parameters}. 
That is, we have the following connection formula:
\begin{equation}
{\mathcal S}_{II}\bigl[\lambda_{\tau_{1}}
(t,\textbf{c},\eta;\alpha) \bigr]
= {\mathcal S}_{I}\bigl[\lambda_{\tau_{1}}
(t,\textbf{c},\eta;\tilde{\alpha}) 
\bigr]\Bigl|_{\tilde{\alpha} = (1 + e^{\pi i (c_{\infty} - c_{0}) \eta}) 
\hspace{+.1em} \alpha}. 
\label{eq:THE connection formula at t0}
\end{equation}
}

The formula \eqref{eq:THE connection formula at t0} 
describes the parametric Stokes phenomenon; 
that is, it gives an explicit relationship between 
the Borel sums of the transseries solution 
in different regions in the parameter space of ${\bf c}$. 
Note that the difference 
${\rm exp}({\pi i (c_{\infty} - c_{0})\eta)}$
of the Borel sums is exponentially small for large $\eta$ 
near ${\bf c} = (2,2-i)$.

\begin{rem} \normalfont
As we see, the formula \eqref{eq:THE connection formula at t0} 
is derived by comparing $\lambda_{\tau_{1}}$ with 
$\lambda_{\infty}$, and the point is that 
the latter one is (conjectured to be) Borel summable 
even if the Stokes geometry degenerates.
We note that, in deriving the formula for 
$\lambda_{\tau_{1}}$ we may compare it with 
another transseries solution $\lambda_{0_{c_{0,-}}}$
which is normalized at the double-pole $t = 0_{c_{0,-}}$ 
in Figure \ref{fig:normalization at infinity0} 
instead of $\lambda_{\infty}$, since $\lambda_{0_{c_{0,-}}}$ 
is also Borel summable when the Stokes geometry 
degenerates under the assumption that Conjecture 
\ref{conjecture:Kamimoto} holds, as well as $\lambda_{\infty}$. 
These transseries solutions are related as 
\begin{equation}
\lambda_{\tau_{1}}(t,{\bf c},\eta;\alpha) = 
\lambda_{0_{c_{0,-}}}(t,{\bf c},\eta;\alpha e^{W_{0_{c_{0,-}}}})
\end{equation}
with 
\begin{equation}
W_{0_{c_{0,-}}} = W_{0_{c_{0,-}}}({\bf c},\eta) = 
- {\cal F}(c_{p}, \eta) 
+ {\cal F}(c_{m}, \eta) + 3 {\cal G}(c_{0}, \eta),
\label{eq:Voros coeff of tau1 and 0c0-}
\end{equation}
due to Theorem \ref{Main Theorem2}.
Since the formal power series 
${\cal F}(c_{p}, \eta)$ and ${\cal G}(c_{0}, \eta)$ 
in \eqref{eq:Voros coeff of tau1 and 0c0-} 
are Borel summable when ${\bf c} = (2, 2-i)$ 
by Proposition \ref{Prop:Borel sum of F and G}, 
they never jump at the wall ${\cal W}_{2}$. 
Thus we have 
\[
{\mathcal S}_{II}\bigl[e^{W_{0_{c_{0,-}}}(\textbf{c},\eta)} \bigr]
= (1 + e^{\pi i (c_{\infty} - c_{0}) \eta}) \hspace{+.2em}
{\mathcal S}_{I}\bigl[e^{W_{0_{c_{0,-}}}(\textbf{c},\eta)} \bigr]. 
\]
Therefore, we obtain the same conclusion as
\eqref{eq:THE connection formula at t0}. 
\end{rem}

Next we discuss the connection problem when 
$t \in U_{t_{1}}$. In this case we should 
compare $\lambda_{\tau_{1}}$ with 
$\lambda_{\infty}$ which is normalized along the 
path from $\infty_{1, -}$ in 
Figure \ref{fig:normalization at infinity1} 
so that $\lambda_{\infty}$ is Borel summable 
when ${\bf c} = (2, 2 - i)$. 
Then, we have the following relation 
between these two transseries 
instead of \eqref{eq:comparison of instanton}:
\begin{equation}
\lambda^{(k)}_{\tau_{1}}(t,{\bf c},\eta) = 
e^{k W} \lambda^{(k)}_{\infty}(t,{\bf c},\eta) 
\hspace{+1.em} (k \ge 0), 
\label{eq:comparison of instanton at t1}
\end{equation}
with $W = W({\bf c},\eta) = - {\cal F}(c_{p},\eta)$ 
because of Theorem \ref{Main Theorem2}.
Since the formal power series ${\cal F}(c_{p},\eta)$
is Borel summable when ${\bf c} = (2, 2-i)$ 
by Proposition \ref{Prop:Borel sum of F and G}, 
we have  
\begin{equation}
{\mathcal S}_{II}\bigl[e^{W(\textbf{c},\eta)} \bigr]
= {\mathcal S}_{I}\bigl[e^{W(\textbf{c},\eta)} \bigr] 
\label{eq:connection formula for W at t1}
\end{equation}
instead of \eqref{eq:connection formula for W-Gamma}. 
As a result, we can conclude the following. \\[-.5em]

\noindent
\textbf{Connection formula on ${\cal W}_{2}$ near $t = t_{1}$.} 
\hspace{+.5em}
\textit{Assume that a neighborhood $U_{t_{1}}$ of the point 
$t_{1}$ and a small number $\varepsilon > 0$ 
satisfies the assumptions {\rm (A-1)} and {\rm (A-2)}.
If the Borel sum of 
$\lambda_{\tau_{1}}(t,{\bf c},\eta;\alpha)$ 
when ${\bf c} = (2 - \varepsilon, 2 -i)$ and 
that of $\lambda_{\tau_{1}}(t,{\bf c},\eta;\tilde{\alpha})$ 
when ${\bf c} = (2 + \varepsilon, 2 -i)$ represent 
the same analytic solution of 
$(P_{\rm III'})_{D_{6}}$ defined on 
on $t \in U_{t_{1}}$ after the analytic continuation 
with respect to the parameter ${\bf c}$ across ${\cal W}_{2}$, 
then the parameters 
$\alpha$ and $\tilde{\alpha}$ satisfy 
$\tilde{\alpha} = \alpha$. 
That is, no parametric Stokes phenomenon occurs 
on ${\cal W}_{2}$ when $t \in U_{t_{1}}$:
\begin{equation}
{\mathcal S}_{II}\bigl[\lambda_{\tau_{1}}
(t,\textbf{c},\eta;\alpha) \bigr] = 
{\mathcal S}_{I}\bigl[\lambda_{\tau_{1}}
(t,\textbf{c},\eta;{\alpha}) \bigr]. 
\label{eq:THE connection formula at t1}
\end{equation}
} 
\begin{rem} \normalfont
From the comparison of two formulas 
\eqref{eq:THE connection formula at t0} and 
\eqref{eq:THE connection formula at t1}, 
we can see that the connection formulas describing 
parametric Stokes phenomena are different depending 
on the location of the independent variable $t$. 
More precisely, the connection formula when $t$ lies inside
of the triangle formed by three bounded Stokes curves 
in Figure \ref{fig:P3D6-u-plane-triangle-0} 
is different from that when $t$ lies outside of the triangle. 
This is also observed for the second Painlev\'e 
equation in \cite{Iwaki-Bessatsu}. We expect that the same 
phenomena also happen to other Painlev\'e equations 
and higher order analogues of them.
\end{rem}

\begin{figure}[h]
\begin{minipage}{0.31\hsize}
\begin{center}
\includegraphics[width=47mm]{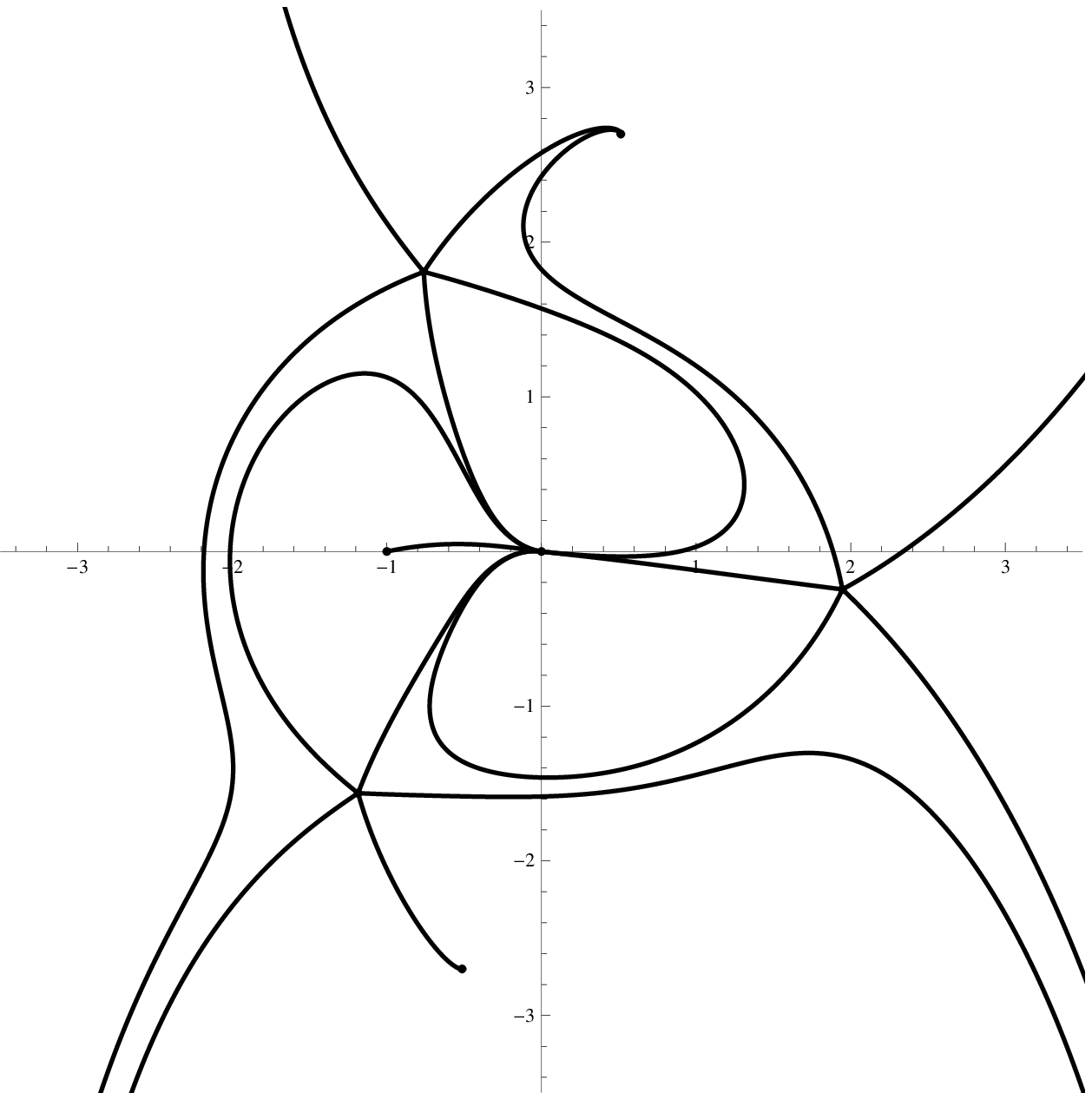}
\end{center}
\caption{\small{${\bf c}=(-2-0.1+i,2+i/2)$.}}
\end{minipage} \hspace{+.5em}
\begin{minipage}{0.31\hsize}
\begin{center}
\includegraphics[width=47mm]{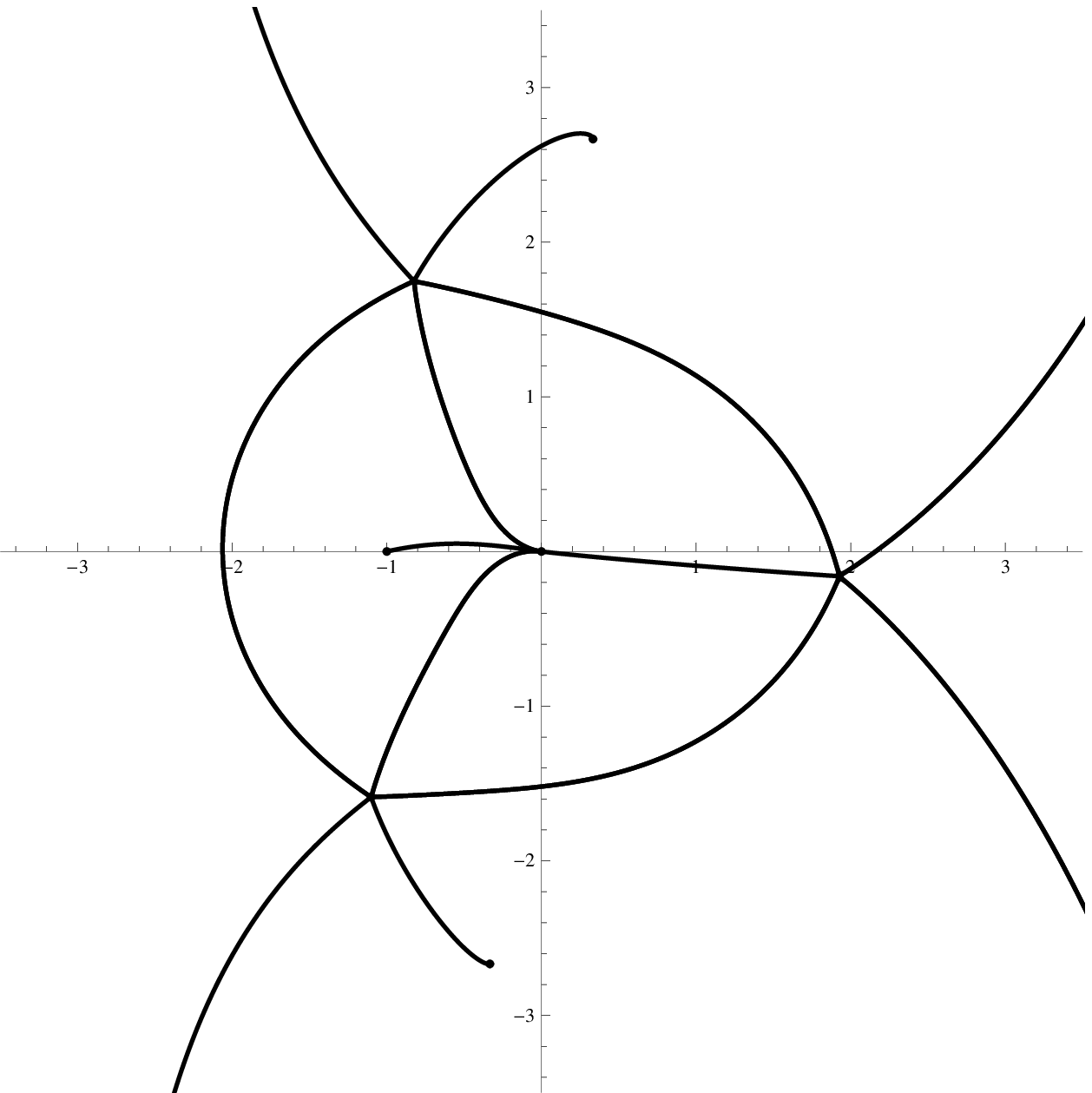}
\end{center}
\caption{\small{${\bf c} = (-2+i,2+i/2)$. \hspace{+.2em}}}
\label{fig:on W4} 
\end{minipage} \hspace{+.3em}
\begin{minipage}{0.31\hsize}
\begin{center}
\includegraphics[width=47mm]{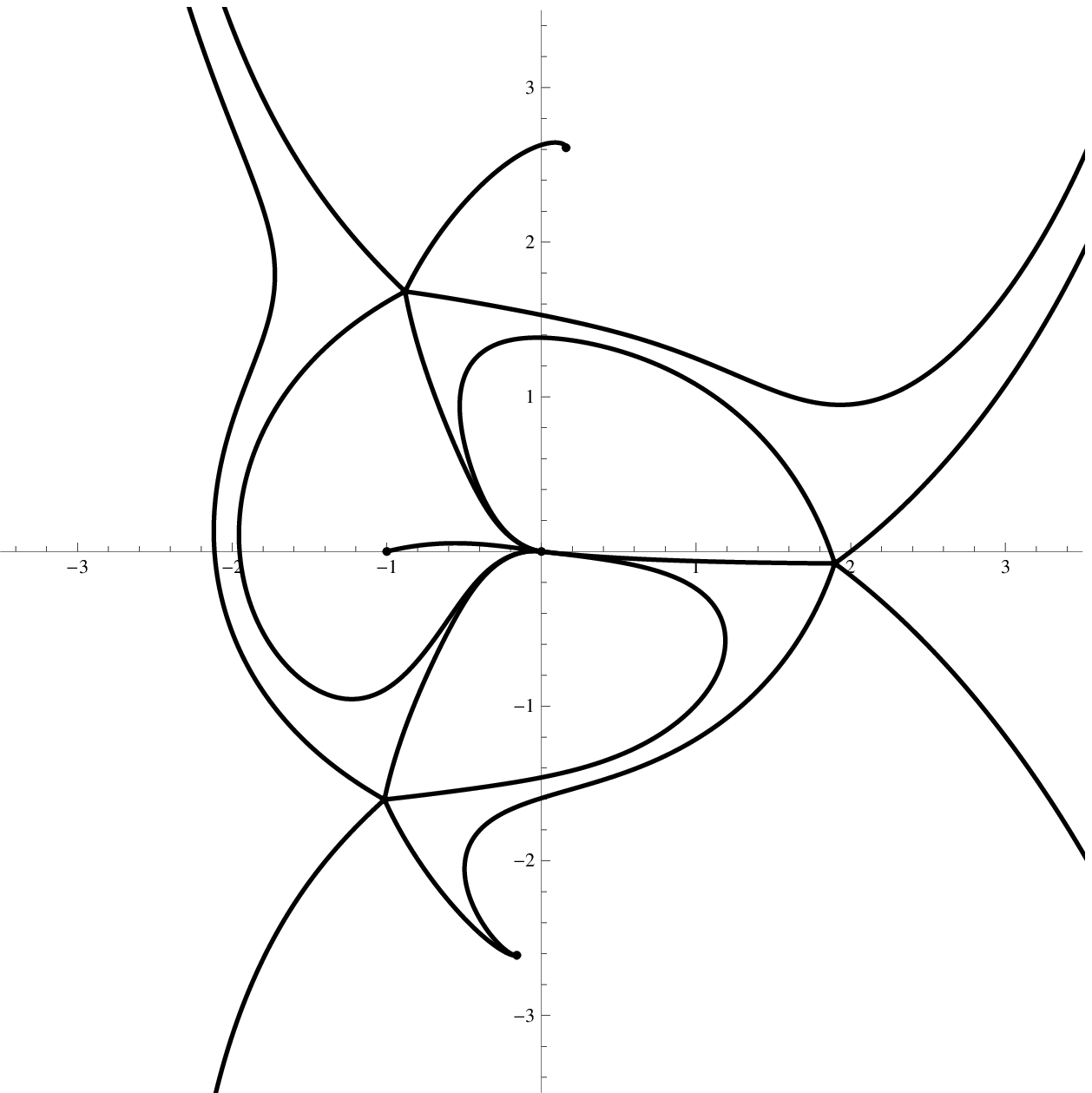}
\end{center}
\caption{\small{${\bf c} = (-2+0.1+i,2+i/2)$}.}
\end{minipage}
\end{figure}

Since we have computed all Voros coefficients of 
$(P_{\rm III'})_{D_{6}}$, we can derive 
connection formulas for parametric Stokes phenomenon 
at any point $t_{\ast} \in \Omega_{D_{6}}$ 
if the assumptions (A-1) and (A-2) are satisfied 
and Conjecture \ref{conjecture:Kamimoto} holds. 
Here we give the statement of connection formulas 
for parametric Stokes phenomena relevant to 
the other triangle-type degeneration in 
Figure \ref{fig:on W4}, which describes the Stokes geometry 
on the $u$-plane when ${\bf c} = (-2+i,2+i/2)$; 
that is ${\bf c}$ lies on the wall ${\cal W}_{4}$. 
In this case we impose the following assumption 
(A-2)' for $\varepsilon > 0$ instead of (A-2):
\begin{itemize}
\item[(A-2)']
The small number $\varepsilon > 0$ satisfies that, 
for any $0 \le \varepsilon' \le \varepsilon$, 
any Stokes curves never touch with 
any points in $U_{t_{\ast}}$ when 
${\bf c} = (-2 \pm \varepsilon' + i, 2 - i)$.
\end{itemize}
\vspace{+.5em}

\noindent
\textbf{Connection formula on ${\cal W}_{4}$.} \hspace{+.5em}
Let $\lambda(t,{\bf c},\eta;\alpha)$ be a transseries solution 
normalized at a turning point or the simple-pole, 
and ${\cal S}_{IV}[\lambda(t,{\bf c},\eta;\alpha)]$ 
(resp., ${\cal S}_{III}[\lambda(t,{\bf c},\eta;\tilde{\alpha})]$)
be the Borel sum of $\lambda(t,{\bf c},\eta;\alpha)$ 
when ${\bf c} = (- 2 - \varepsilon + i, 2 + i/2)$ 
(resp., of $\lambda(t,{\bf c},\eta;\tilde{\alpha})$
when ${\bf c} = (- 2 + \varepsilon + i, 2 + i/2)$) 
for a sufficiently small $\varepsilon > 0$.
\\[+.2em]
\textit{(i) Fix $u_{*}$ {\bf outside} of the 
triangle formed by bounded Stokes curves in 
Figure \ref{fig:on W4}. Assume that a neighborhood $U_{t_{*}}$ 
of the point $t_{*}$ which corresponds to $u_{*}$ 
and a small number $\varepsilon > 0$ 
satisfies the assumptions {\rm (A-1)} and {\rm (A-2)'}.
If ${\cal S}_{IV}[\lambda(t,{\bf c},\eta;\alpha)]$ and 
${\cal S}_{III}[\lambda(t,{\bf c},\eta;\tilde{\alpha})]$
represent the same analytic solution of $(P_{\rm III'})_{D_{6}}$ 
defined on $t \in U_{t_{*}}$ after the analytic continuation 
with respect to the parameter ${\bf c}$ across ${\cal W}_{4}$, 
then the parameters $\alpha$ and $\tilde{\alpha}$ satisfy 
\begin{equation}
\tilde{\alpha} = (1 + e^{\pi i (c_{\infty} + c_{0}) \eta})^{\pm1} 
\hspace{+.2em} \alpha, 
\end{equation}
where the sign $\pm 1$ depends on the choice 
of square root $R_{-1}(t,{\bf c}) = \sqrt{\Delta(t,{\bf c})}$.
That is, we have the following connection formula:
\begin{equation}
{\mathcal S}_{IV}\bigl[\lambda(t,\textbf{c},\eta;\alpha) \bigr] 
= {\mathcal S}_{III}\bigl[\lambda(t,\textbf{c},\eta;\tilde{\alpha}) 
\bigr]\Bigl|_{\tilde{\alpha} = 
(1 + e^{\pi i (c_{\infty} + c_{0}) \eta})^{\pm1} \hspace{+.2em} \alpha}. 
\label{eq:THE connection formula on W4 sono1}
\end{equation}
\vspace{+.2em} \noindent
(ii) Fix $u_{*}$ {\bf inside} of the 
triangle formed by bounded Stokes curves in 
Figure \ref{fig:on W4}. Assume that a neighborhood $U_{t_{*}}$ 
of the point $t_{*}$ which corresponds to $u_{*}$ 
and a small number $\varepsilon > 0$ 
satisfies the assumptions {\rm (A-1)} and {\rm (A-2)'}.
If ${\cal S}_{IV}[\lambda(t,{\bf c},\eta;\alpha)]$ and 
${\cal S}_{III}[\lambda(t,{\bf c},\eta;\tilde{\alpha})]$
represent the same analytic solution of $(P_{\rm III'})_{D_{6}}$ 
defined on $t \in U_{t_{*}}$ after the analytic continuation 
with respect to the parameter ${\bf c}$ across ${\cal W}_{4}$, 
then the parameters $\alpha$ and $\tilde{\alpha}$ satisfy 
$\tilde{\alpha} = \alpha$. That is, no parametric Stokes phenomenon 
occurs on the wall ${\cal W}_{4}$ when $t \in U_{t_{*}}$:
\begin{equation}
{\mathcal S}_{IV}\bigl[\lambda(t,\textbf{c},\eta;\alpha) \bigr] = 
{\mathcal S}_{III}\bigl[\lambda(t,\textbf{c},\eta;{\alpha}) \bigr]. 
\label{eq:THE connection formula on W4 sono2}
\end{equation}
} 

\subsection{Connection problem for parametric 
Stokes phenomena relevant to the loop-type degeneration}
\label{subsection:connection of loop-type}

Next we discuss the connection problem relevant to 
the loop-type degeneration of Stokes geometry observed 
when ${\bf c} = (i,3+i/2)$; that is, the connection problem 
on the wall ${\cal W}_{3}$. Figure \ref{fig:P3D6-u-plane-loop2-minus} 
$\sim$ \ref{fig:P3D6-u-plane-loop2-plus} 
describes the Stokes geometry on the $u$-plane 
near ${\bf c} = (i,3+i/2)$. 

\begin{figure}[h]
\begin{minipage}{0.31\hsize}
\begin{center}
\includegraphics[width=48mm]{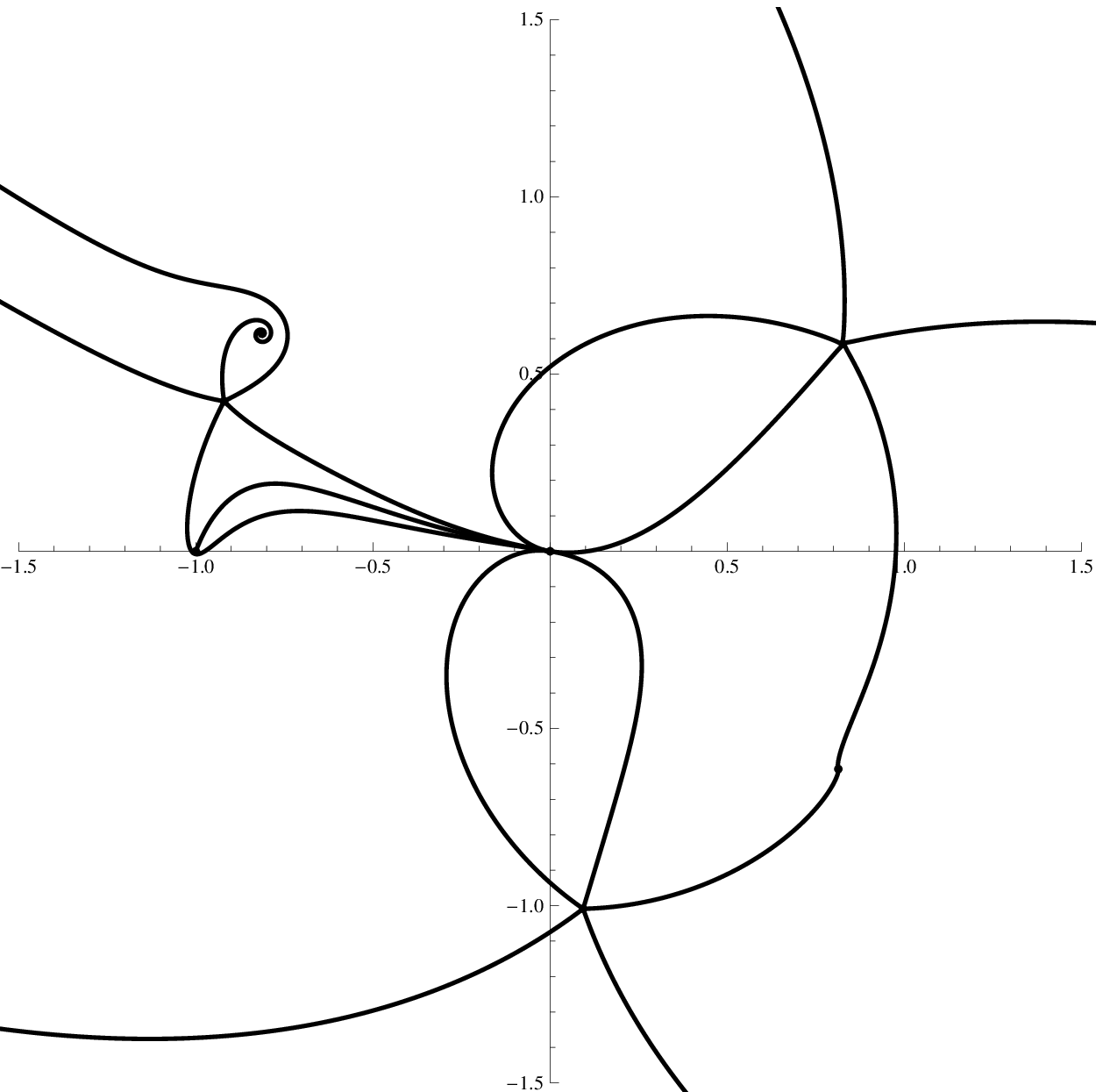}
\end{center}
\caption{\small{${\bf c}=(-0.2+i,3+i/2)$.}}
\label{fig:P3D6-u-plane-loop2-minus}
\end{minipage} \hspace{+.5em}
\begin{minipage}{0.31\hsize}
\begin{center}
\includegraphics[width=48mm]{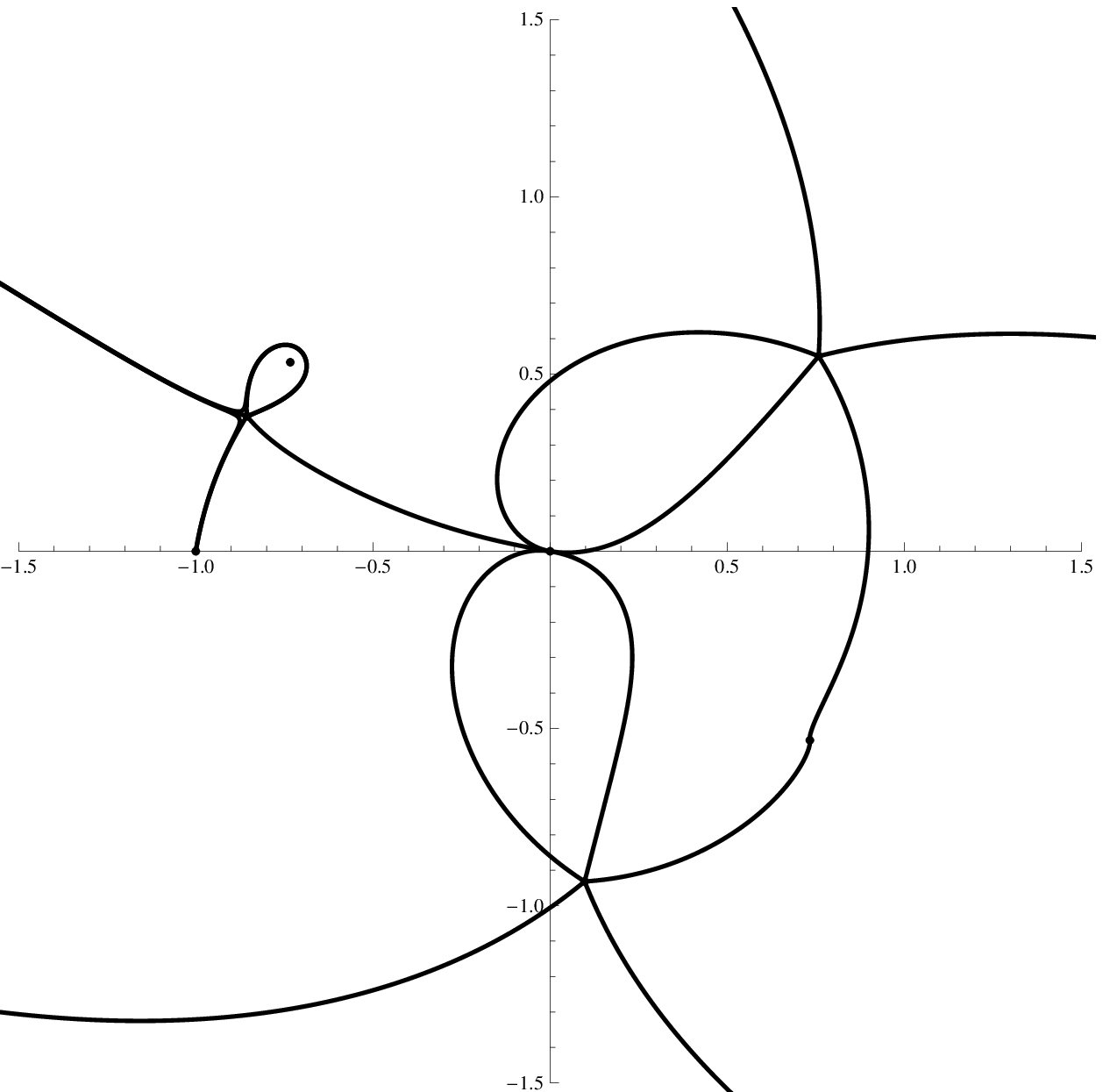}
\end{center}
\caption{\small{${\bf c} = (i,3+i/2)$. \hspace{+.2em}}}
\label{fig:on the wall W3} 
\end{minipage} \hspace{+.3em}
\begin{minipage}{0.31\hsize}
\begin{center}
\includegraphics[width=48mm]{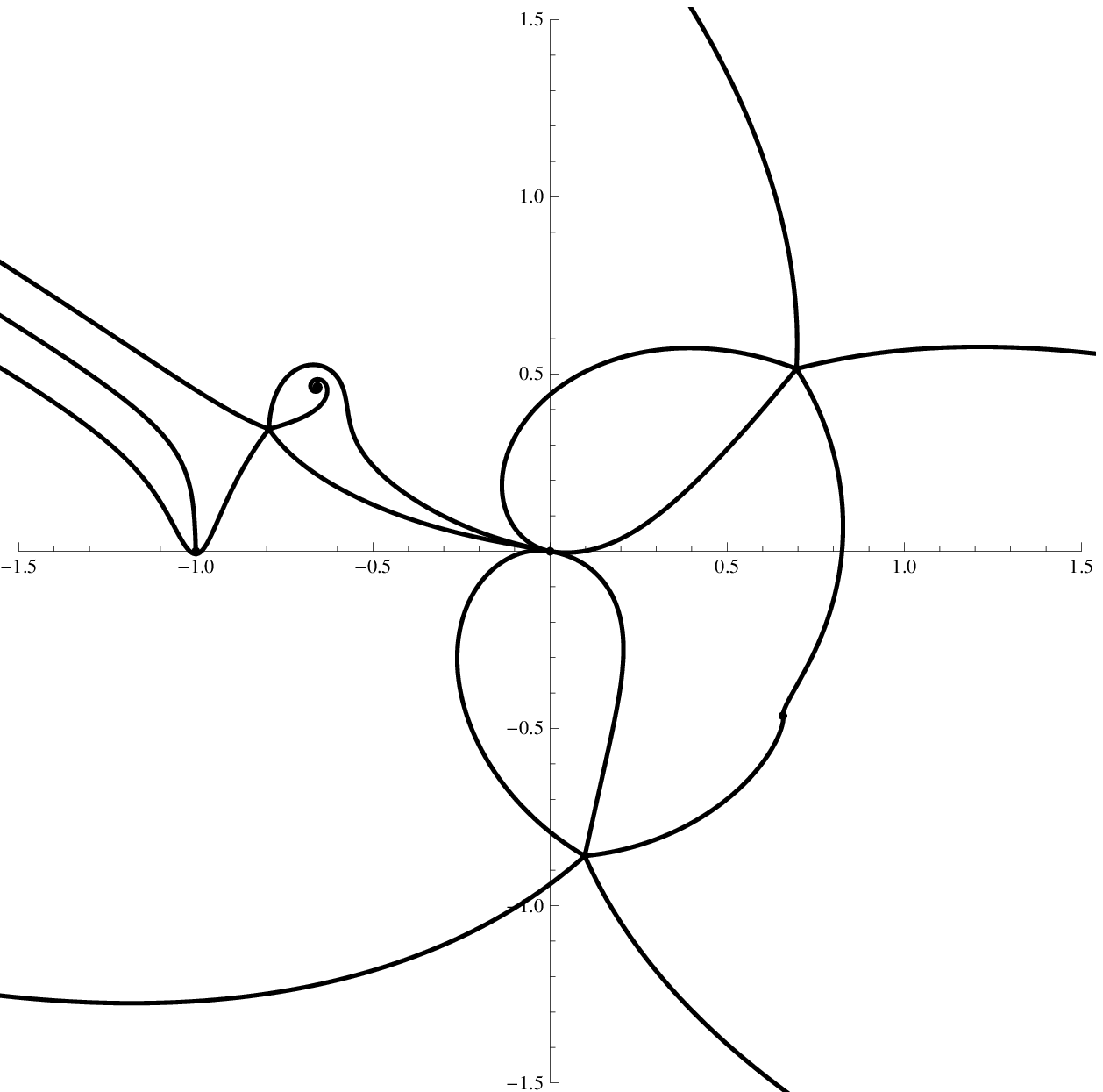}
\end{center}
\caption{\small{${\bf c} = (+0.2+i,3+i/2)$}.}
\label{fig:P3D6-u-plane-loop2-plus}
\end{minipage}
\end{figure}

First we consider the case that the independent variable 
$t$ (or $u$) lies outside of the ``loop''. 
Here we impose the following assumption 
(A-2)'' instead of (A-2):
\begin{itemize}
\item[(A-2)'']
The small number $\varepsilon > 0$ satisfies that, 
for any $0 \le \varepsilon' \le \varepsilon$, 
any Stokes curves never touch with 
any points in $U_{t_{\ast}}$ when 
${\bf c} = (\pm \varepsilon' + i, 3 + i/2)$.
\end{itemize}
In this case the discussion given in Section 
\ref{subsection:connection of triangle-type}
can be applicable, and the relevant Voros coefficient 
to be considered to derive the connection formula 
are one of $W_{\infty_{j,\pm}}({\bf c}, \eta)$
($1 \le j \le 4$) or $W_{0_{c_{0,\pm}}}({\bf c},\eta)$. 
These Voros coefficients are all Borel summable 
when ${\bf c}$ lies on the wall ${\cal W}_{3}$, 
and hence they never jump on the wall.
Therefore, we have the following 
conclusion: \\[-.5em]

\noindent
\textbf{Connection formula on ${\cal W}_{3}$.} \hspace{+.5em}
{\it Let $\lambda(t,{\bf c},\eta;\alpha)$ be a transseries solution 
normalized at a turning point or the simple-pole, 
and ${\cal S}_{III}[\lambda(t,{\bf c},\eta;\alpha)]$ 
(resp., ${\cal S}_{II}[\lambda(t,{\bf c},\eta;\tilde{\alpha})]$)
be the Borel sum of $\lambda(t,{\bf c},\eta;\alpha)$ 
when ${\bf c} = (- \varepsilon + i, 3 + i/2)$ 
(resp., of $\lambda(t,{\bf c},\eta; \tilde{\alpha})$
when ${\bf c} = (+ \varepsilon + i, 3 + i/2)$) 
for a sufficiently small $\varepsilon > 0$. 
Fix $u_{*}$ {\bf outside} of the loop 
formed by a bounded Stokes curve in 
Figure \ref{fig:on the wall W3}. 
Assume that a neighborhood $U_{t_{*}}$ 
of the point $t_{*}$ which corresponds to $u_{*}$ 
and a small number $\varepsilon > 0$ 
satisfies the assumptions {\rm (A-1)} and {\rm (A-2)''}.
If ${\cal S}_{III}[\lambda(t,{\bf c},\eta;\alpha)]$ and 
${\cal S}_{II}[\lambda(t,{\bf c},\eta;\tilde{\alpha})]$
represent the same analytic solution of $(P_{\rm III'})_{D_{6}}$ 
defined on $t \in U_{t_{*}}$ after the analytic continuation 
with respect to the parameter ${\bf c}$ across ${\cal W}_{3}$, 
then the parameters $\alpha$ and $\tilde{\alpha}$ satisfy 
$\tilde{\alpha} = \alpha$. That is, no parametric Stokes phenomenon 
occurs on the wall ${\cal W}_{3}$ when $t \in U_{t_{*}}$:
\begin{equation}
{\mathcal S}_{III}\bigl[\lambda(t,\textbf{c},\eta;\alpha) \bigr] = 
{\mathcal S}_{II}\bigl[\lambda(t,\textbf{c},\eta;{\alpha}) \bigr]. 
\label{eq:THE connection formula on W3}
\end{equation}
}

On the other hand, the cases when $t$ (or $u$) lies 
inside the loop are quite different from the above case. 
When the loop-type degeneration appears 
and $t_{*}$ lies inside the loop, 
the assumption (A-2)'' is not satisfied 
for any $\varepsilon > 0$ because of  
``infinitely many spirals''. 
Actually, if $t_{\ast}$ is fixed inside of the loop, 
then Stokes curves touch with $t_{\ast}$
infinitely many times as ${\bf c}$ 
tends to $\varepsilon'$ 
varies $0 \le \varepsilon' \le \varepsilon$, 
and hence usual Stokes phenomena occur to 
transseries solutions infinitely many times. 
These loop-type degenerations have not been analyzed 
even in the case of linear differential equations, 
due to the same difficulty. 
In order to describe connection formulas, 
we need some modification of the Borel resummation method, 
but we do not have an appropriate way at this time. 
This is a future issue to be discussed.

\appendix

\section{Examples of Stokes geometry of 
$(P_{\rm III'})_{D_{6}}$}
\label{Appendix:examples of P-Stokes geometry}

Here we show the examples of Stokes geometry of 
$(P_{\rm III'})_{D_{6}}$ on the $u$-plane 
($u$ is defined by \eqref{eq:u-coordinate}) 
when the parameter ${\bf c}$ 
is in the chambers $I$ $\sim$ $VIII$ and on the walls 
${\cal W}_{1}$ $\sim$ ${\cal W}_{8}$ in 
Figure \ref{fig:walls and chambers}.
Note that, since the quadratic differential 
\eqref{eq:quadratic differential in u-plane} is 
invariant under the exchange of the parameters 
$c_{\infty} \leftrightarrow c_{0}$, 
so is the Stokes geometry. 
Therefore it is enough to show the Stokes geometries 
when ${\bf c}$ is in the chambers $II, III, IV, V$ 
and on the boundaries of them.
We conjecture that, as long as the parameter 
${\bf c}$ does not across the walls, 
the topological type of configuration 
of the Stokes curves never change, 
and degeneration of Stokes geometry occur 
only on these walls. This is true as far as we 
checked by numerical experiments, 
but we can not confirm it analytically. 

  \begin{figure}[h]
  \begin{minipage}{0.31\hsize}
  \begin{center}
  \includegraphics[width=50mm]
  {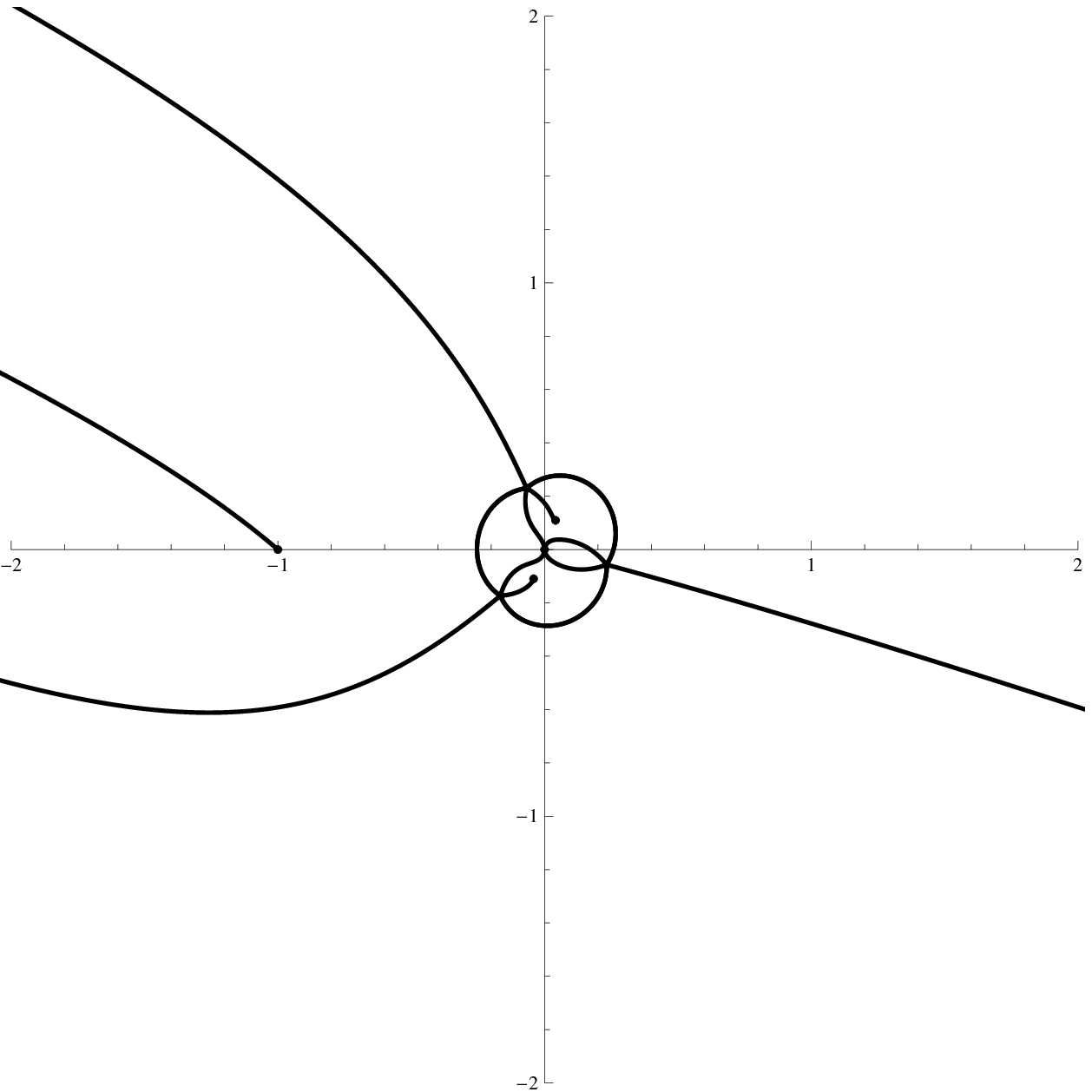}
  \end{center}
  \caption{\small{On ${\cal W}_{2}$ : ${\bf c} = (2+i,2+i/2)$.}}
  \label{fig:A-1}
  \end{minipage} \hspace{+.5em}
  \begin{minipage}{0.31\hsize}
  \begin{center}
  \includegraphics[width=50mm]
  {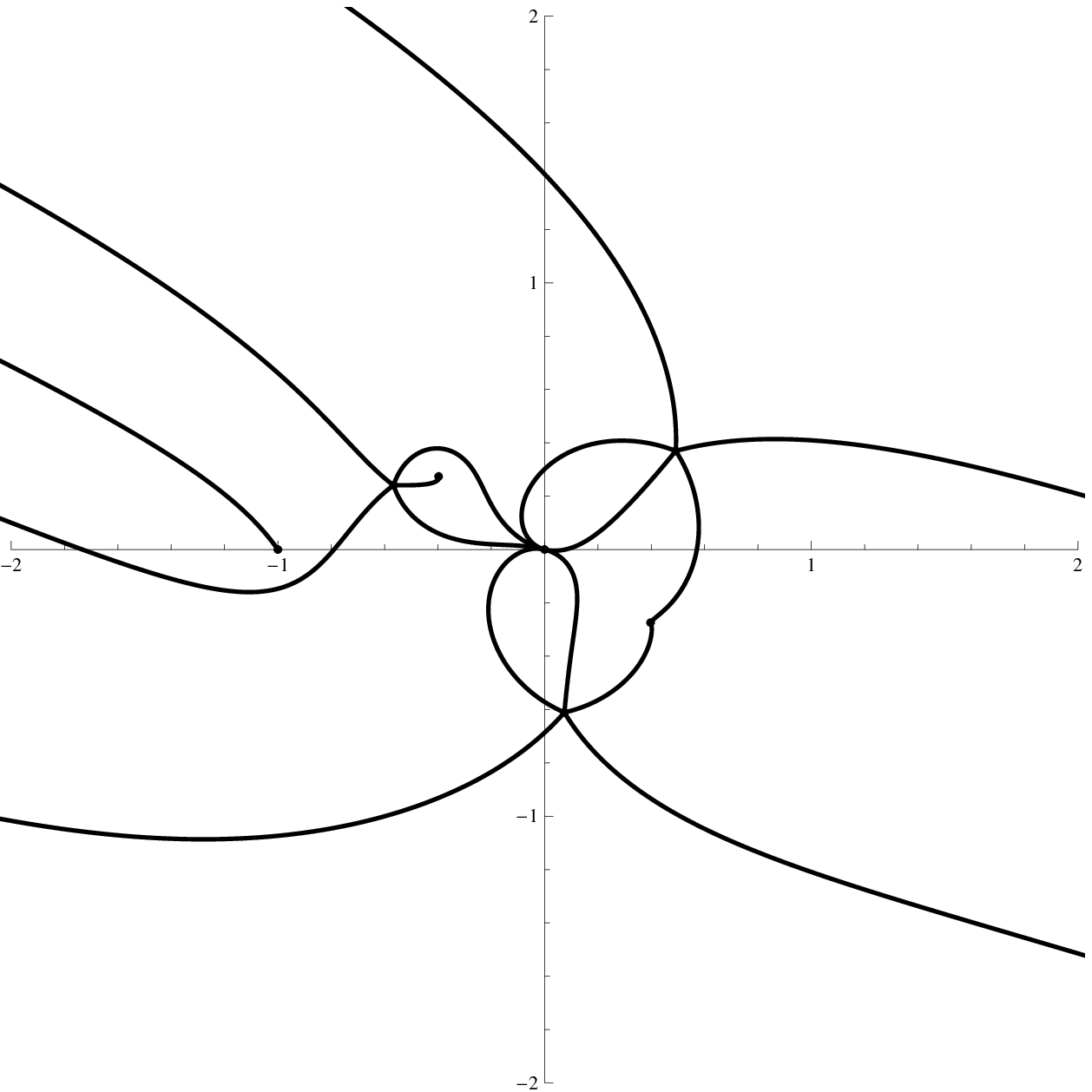}
  \end{center}
  \caption{\small{In $II$ : ${\bf c} = (1+i,3+i/2)$.}}
  \label{fig:A-2}
  \end{minipage} \hspace{+.3em}
  \begin{minipage}{0.31\hsize}
  \begin{center}
  \includegraphics[width=50mm]
  {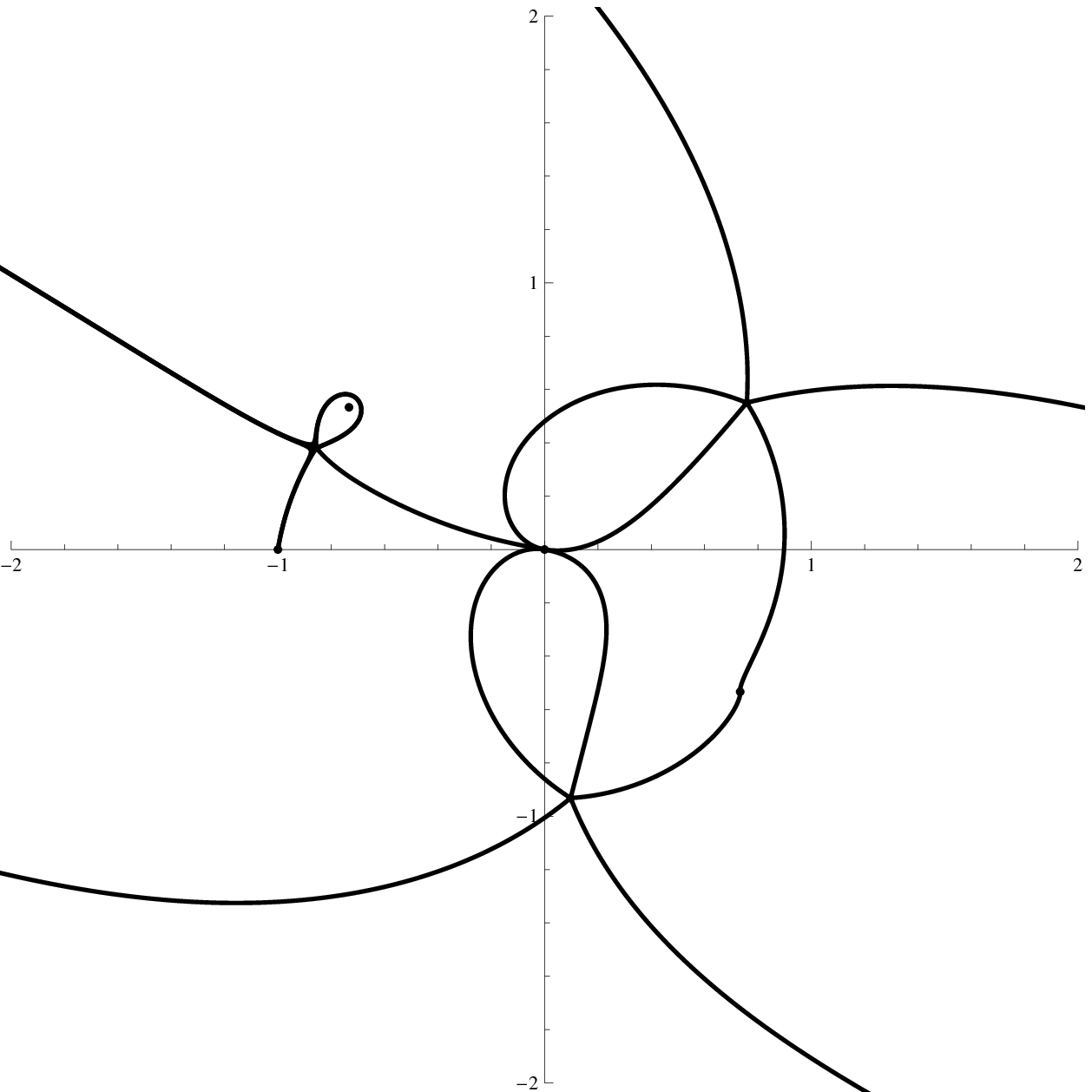}
  \end{center}
  \caption{\small{On ${\cal W}_{3}$ : ${\bf c} = (0+i,3+i/2)$}.}
  \label{fig:A-3}
  \end{minipage}
  \end{figure}
  \begin{figure}[h]
  \begin{minipage}{0.31\hsize}
  \begin{center}
  \includegraphics[width=50mm]
  {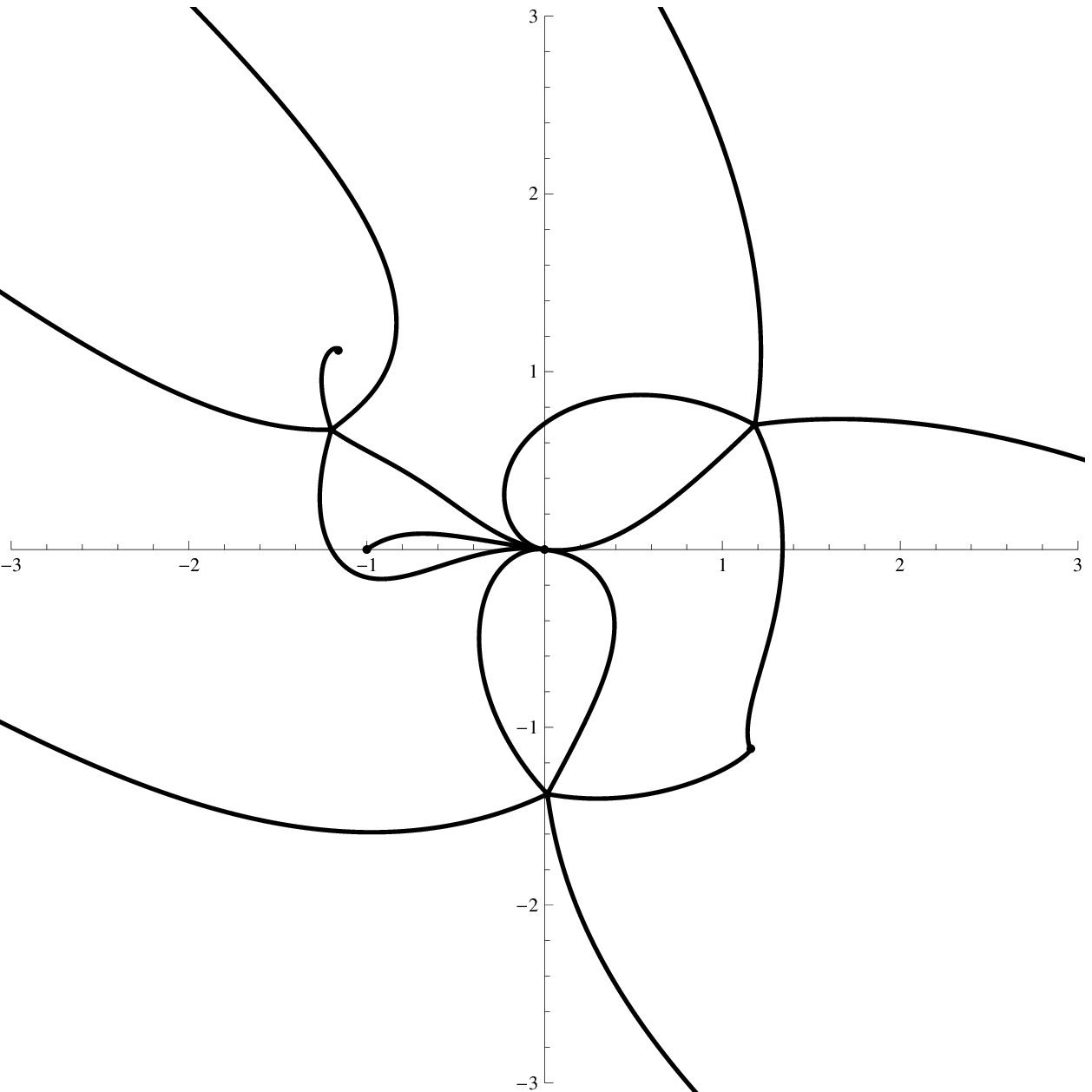}
  \end{center}
  \caption{\small{In $III$ : ${\bf c} = (-1+i,3+i/2)$.}}
  \label{fig:A-4}
  \end{minipage} \hspace{+.5em}
  \begin{minipage}{0.31\hsize}
  \begin{center}
  \includegraphics[width=50mm]
  {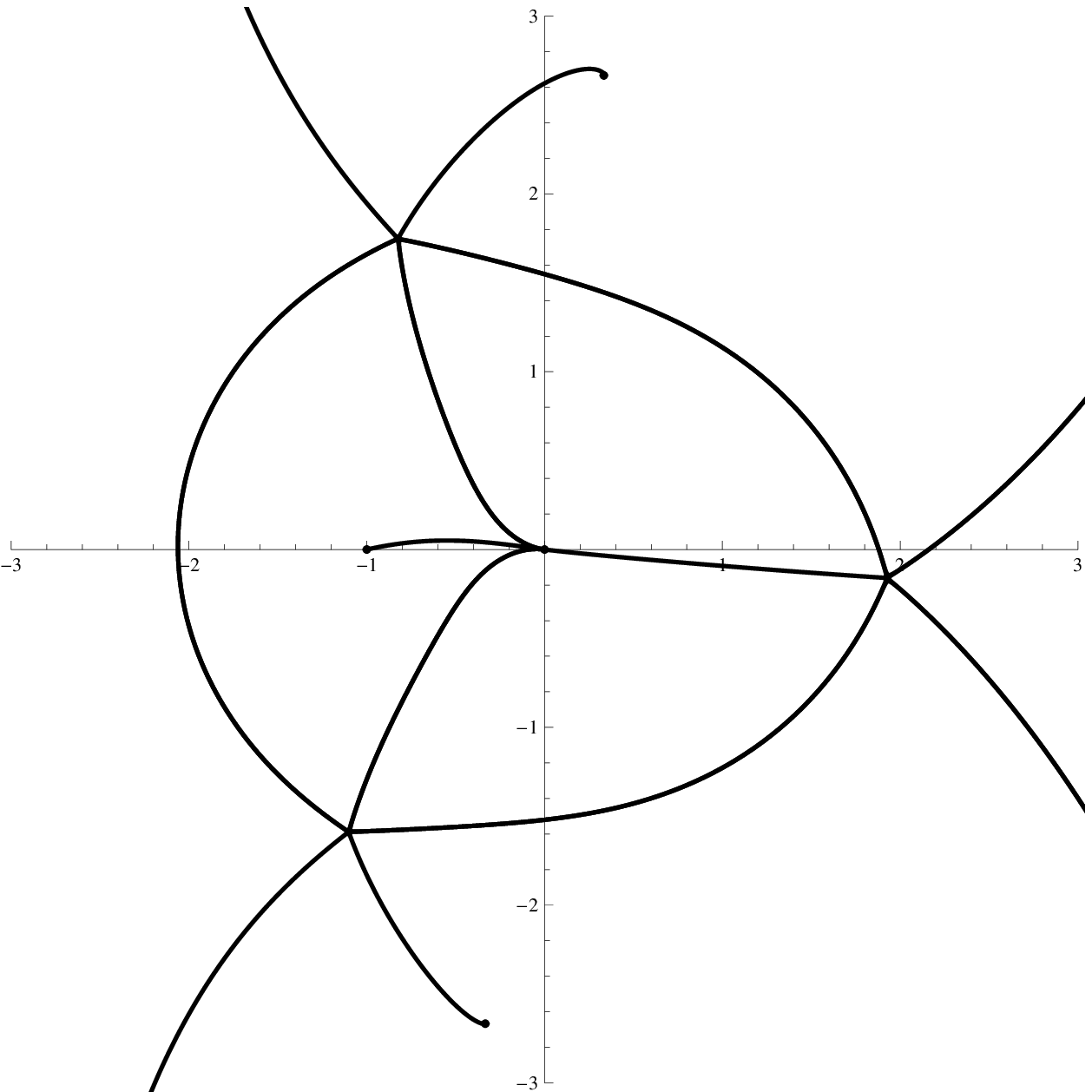}
  \end{center}
  \caption{\small{On ${\cal W}_{4}$ : ${\bf c} = (-2+i,2+i/2)$.}}
  \label{fig:A-5}
  \end{minipage} \hspace{+.3em}
  \begin{minipage}{0.31\hsize}
  \begin{center}
  \includegraphics[width=50mm]
  {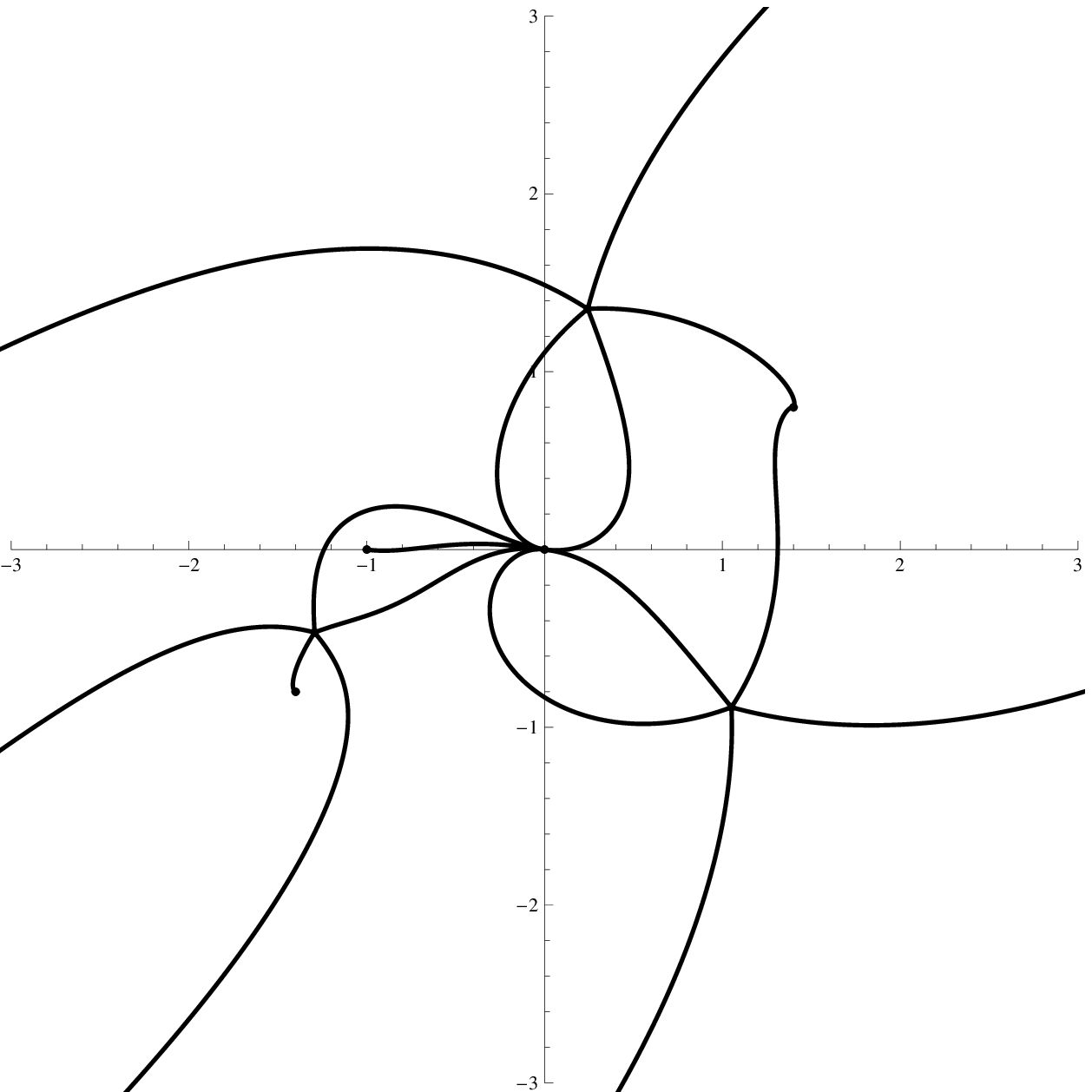}
  \end{center}
  \caption{\small{In $IV$ : ${\bf c} = (-3+i,1+i/2)$}.}
  \label{fig:A-6}
  \end{minipage}
  \end{figure}
  \begin{figure}[h]
  \begin{minipage}{0.31\hsize}
  \begin{center}
  \includegraphics[width=50mm]
  {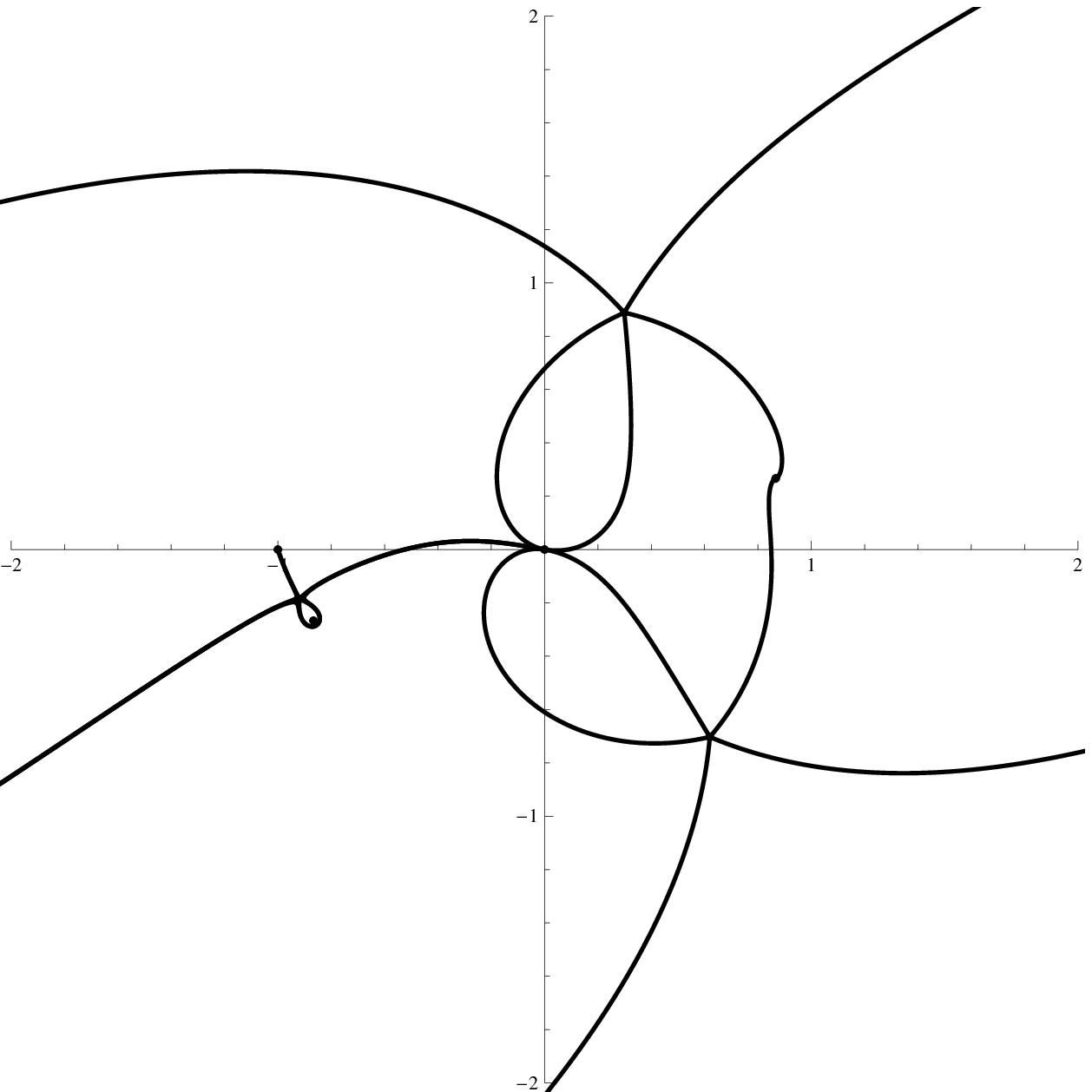}
  \end{center}
  \caption{\small{On ${\cal W}_{5}$ : ${\bf c} = (-3+i,0+i/2)$.}}
  \label{fig:A-7}
  \end{minipage} \hspace{+.5em}
  \begin{minipage}{0.31\hsize}
  \begin{center}
  \includegraphics[width=50mm]
  {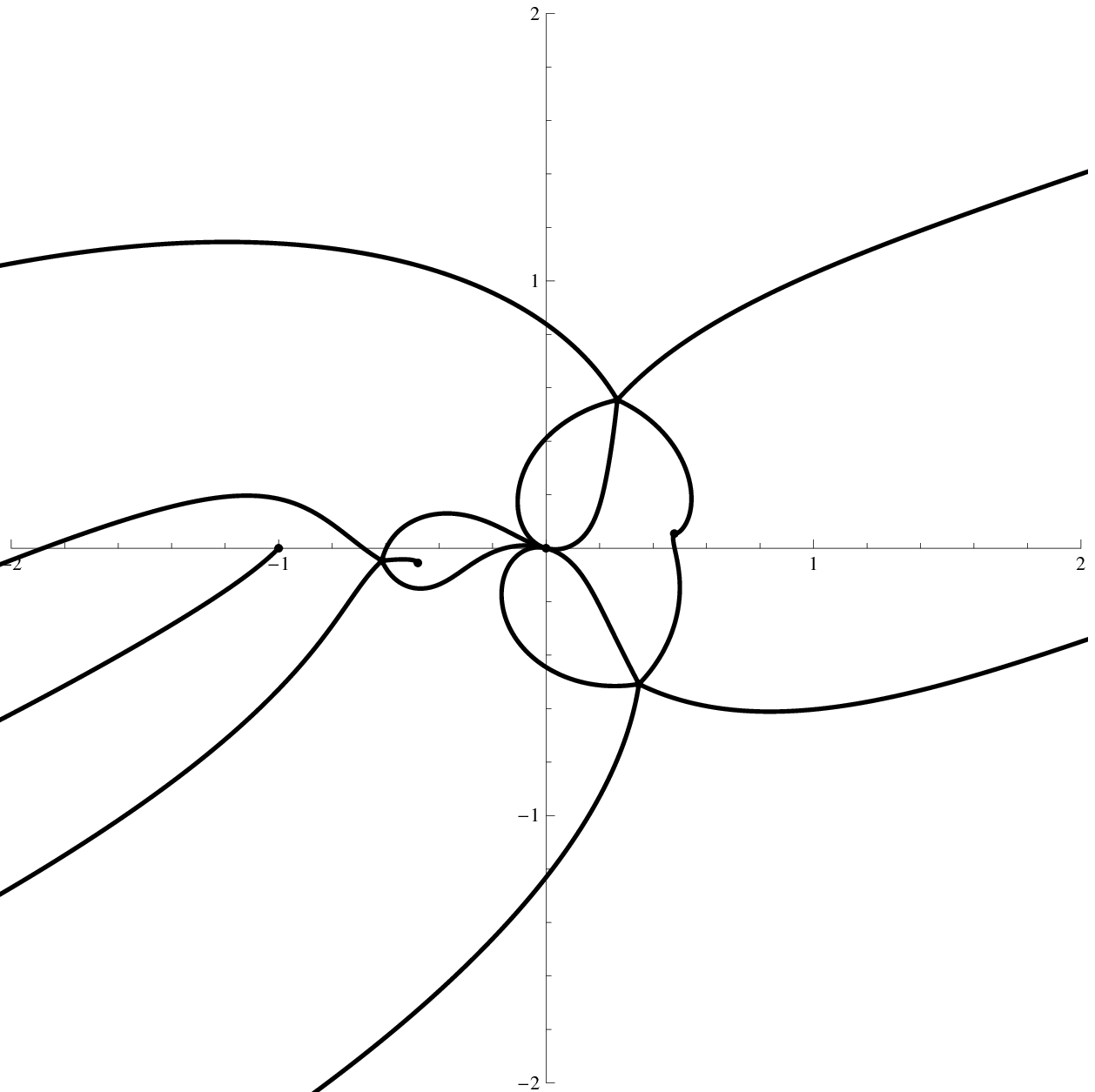}
  \end{center}
  \caption{\small{In $V$ : ${\bf c} = (-3+i,-1+i/2)$.}}
  \label{fig:A-8}
  \end{minipage} \hspace{+.3em}
  \begin{minipage}{0.31\hsize}
  \begin{center}
  \includegraphics[width=50mm]
  {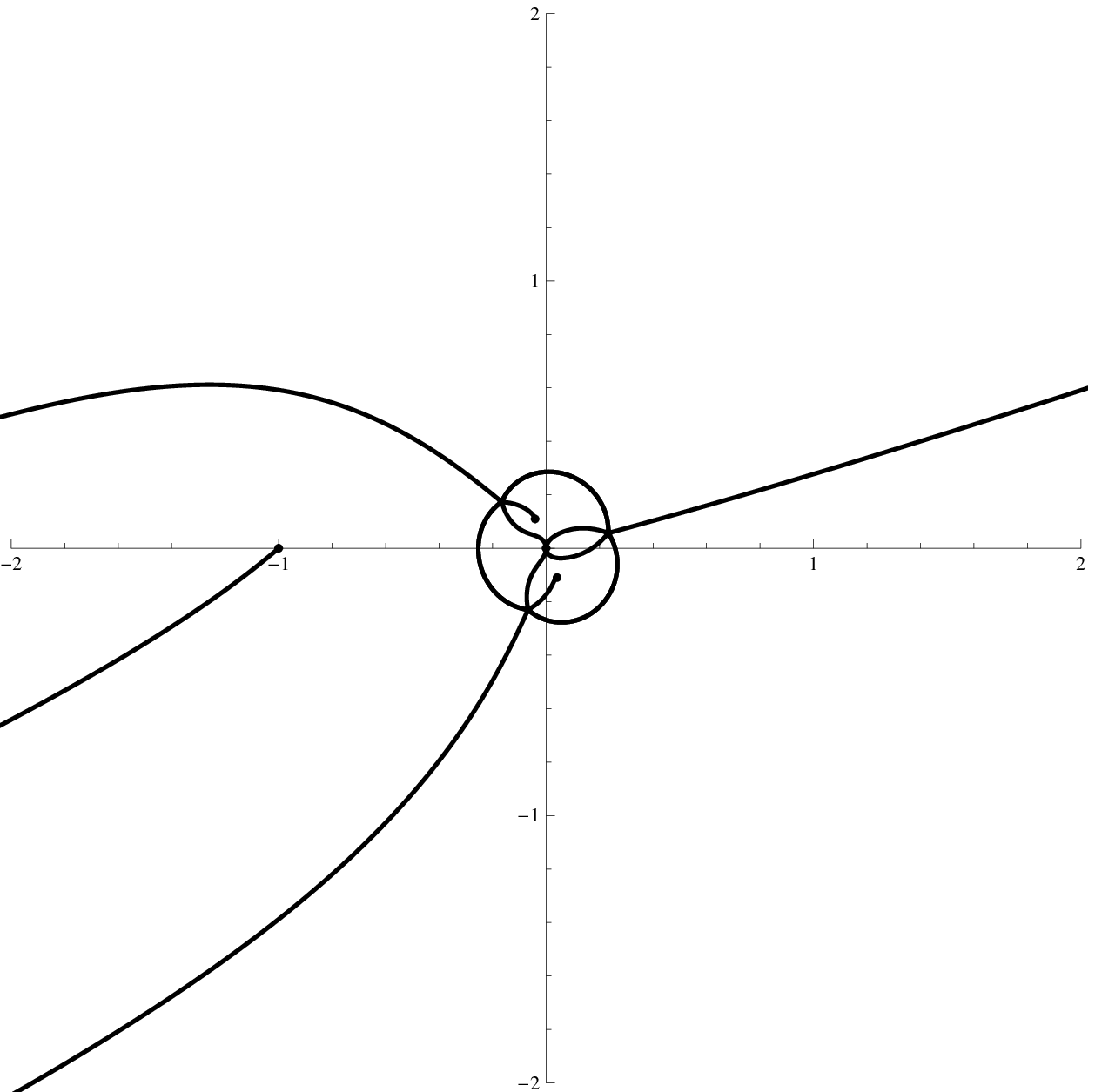}
  \end{center}
  \caption{\small{On ${\cal W}_{6}$ : ${\bf c} = (-2+i,-2+i/2)$}.}
  \label{fig:A-9}
  \end{minipage}
  \end{figure}
  \begin{figure}[h]
  \begin{center}
  \includegraphics[width=50mm]
  {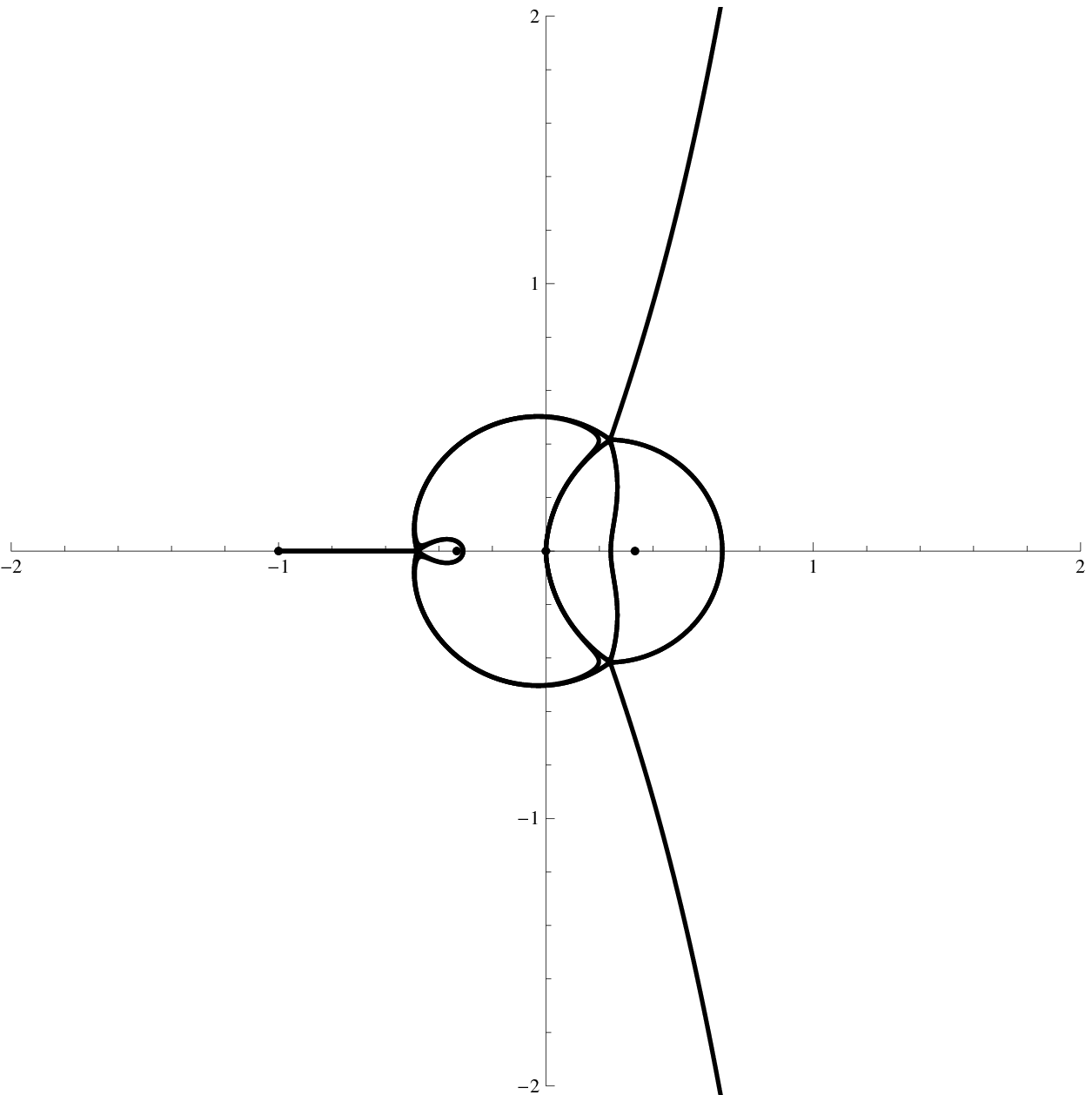}
  \end{center}
  \caption{\small{${\bf c} = (0+i,0+i/2)$.}}
  \label{fig:A-10}
  \end{figure}

\section{Homogeneity}
\label{Appendix:Homogeneity}

As is explained in the beginning of Section 
\ref{section:transseries solutions and P-Stokes geometry}, 
quantities appearing in this paper have a homogeneity 
with respect to the following scaling operation: 
\[
(t, c_{\infty}, c_{0}, \eta) \mapsto 
(r^{-2}t, r^{-1}c_{\infty}, r^{-1}c_{0}, r \eta) 
\hspace{+1.em} (r > 0).
\] 
For example, the homogenious degree of 
$\lambda_0(t,{\bf c}) = \lambda_{0}(t, c_{\infty}, c_{0})$ 
which is the algebraic function defined by 
$F(\lambda_{0},t, {\bf c}) = 0$ is $-1$; 
that is, 
\[
\lambda_0(r^{-2}t, r^{-1}c_{\infty}, r^{-1}c_{0}) 
= 
r^{-1} \lambda_0(t, c_{\infty}, c_{0}).
\]

We list the homogenious degrees of quantities below. 
We regard the free parameter $\alpha$ as 
a constant with homogenous degree $0$.

\begin{eqnarray}
\lambda_{0}(t,{\bf c}) 
& : & -1. \hspace{+7.em}
\label{eq:degree of lambda0}  
\\[+.3em]
\mu_{0}(t,{\bf c}) 
& : & 0.
\label{eq:degree of mu0}    
\\[+.3em]
\Delta(t,{\bf c}) 
& : &  + 2. 
\label{eq:degree of Delta}    
\\[+.3em]
\text{a turning point} 
\hspace{+1.em} \tau({\bf c}) 
& : & - 2. 
\label{eq:degree of turning points}
\\[+.3em]
R(t,{\bf c},\eta) 
& : & +2.
\label{eq:degree of R}    
\\[+.3em]
R_{\rm odd}(t,{\bf c},\eta) 
& : & +2.
\label{eq:degree of Rodd}    
\\[+.3em]
\phi(t,{\bf c}) 
& : &  -1.
\label{eq:degree of phi}    
\\[+.3em]
\lambda(t,{\bf c},\eta;\alpha) 
& : & -1.   
\label{eq:degree of lambda}  
\\[+.3em]
\mu(t,{\bf c},\eta;\alpha) 
& : & 0 . 
\label{eq:degree of mu}    
\\[+.3em]
W_{\tau,\infty}({\bf c},\eta) 
& : & 0 .
\\[+.3em]
\label{eq:degree of W}
W_{\tau,0_{c_{\ast}}}({\bf c},\eta) 
& : & 0 \hspace{+.7em} (\ast = \infty, 0) . 
\label{eq:degree of W-0}
\end{eqnarray}

\section{Asymptotics of coefficients}
\label{Appendix:asymptotics}

As is noted in Section \ref{section:proof of the Main theorems}, 
asymptotic behaviors of $\lambda^{(0)}(t,{\bf c},\eta)$, 
$\mu^{(0)}(t,{\bf c},\eta)$ and $R_{\pm}(t,{\bf c},\eta)$
when $t$ tends to singular points $\infty$ and $0$
is important in derivations of difference equations 
satisfied by Voros coefficients.
Here we summarize all asymptotic behaviors of them. 
(We only show the case $t \rightarrow \infty_{j,+}$ 
($1 \le j \le 4$) and $t \rightarrow 0_{c_{\ast,+}}$ 
($\ast = \infty$ or $0$). 
Behaviors when 
$t \rightarrow \infty_{j,-}$ and 
$t \rightarrow 0_{c_{\ast,-}}$ 
can be obtained 
from these list by replacing $R_{-1} \mapsto - R_{-1}$  
and $R_{\pm} \mapsto R_{\mp}$.)

\vspace{+.5em}
\noindent
$\bullet$ \underline{The case $t \rightarrow \infty_{1,+}$:}
\begin{eqnarray}
\lambda_{0}(t,{\bf c}) & = & t^{1/2} 
+ \frac{c_{\infty} - c_{0}}{4} 
+ \frac{(c_{\infty} - c_{0})(3c_{\infty} + c_{0})}{32} t^{-1/2}
+ O(t^{-1}),
\label{eq:behavior of lambda0 infinity-plus} \\[+.5em]
\mu_{0}(t,{\bf c}) & = & 
\frac{c_{\infty}+c_{0}}{4}t^{-1/2}
+ O(t^{-3/2}),
\label{eq:behavior of mu0 infinity-plus} \\[+.5em]
\lambda^{(0)}(t,{\bf c},\eta) & = & t^{1/2} 
+ \frac{c_{\infty} - c_{0}}{4} 
+ \frac{(c_{\infty} - c_{0})(3c_{\infty} + c_{0})}{32} t^{-1/2}
+ O(t^{-1}),
\label{eq:behavior of lambda(0) infinity-plus} \\[+.5em]
\mu^{(0)}(t,{\bf c},\eta) & = & 
\frac{c_{\infty}+c_{0}-\eta^{-1}}{4}t^{-1/2}
+ O(t^{-3/2}),
\label{eq:behavior of mu(0) infinity-plus} \\[+.5em]
R_{-1}(t,{\bf c}) & = & 2 t^{-1/2} 
- \frac{c_{\infty}+c_{0}}{4} t^{-1}
+ \frac{(3c_{\infty}-c_{0})(c_{\infty}-3c_{0})}{64}
t^{-3/2} + O(t^{-2}), \nonumber \\
\label{eq:behavior of R-1 infinity-plus} \\[+.5em]
R_{\pm}(t,{\bf c},\eta) & = & \pm 2 \eta t^{-1/2} 
- \frac{(\pm c_{\infty}\eta \pm c_{0}\eta - 1)}{4} t^{-1} 
\nonumber \\
&  & 
+ \frac{\pm 3c_{\infty}^{2} \eta^{2}
\mp 10 c_{\infty} c_{0} \eta^{2} \pm 3c_{0}^{2} \eta^{2}  
- 10 c_{\infty} \eta + 6 c_{0} \eta \mp 1 }{64\eta}
t^{-3/2} + O(t^{-2}). \nonumber \\
&  &
\label{eq:behavior of R infinity-plus} 
\end{eqnarray}

\vspace{+.5em}
\noindent
$\bullet$ \underline{The case $t \rightarrow \infty_{2,+}$:}
\begin{eqnarray}
\lambda_{0}(t,{\bf c}) & = & - t^{1/2} 
+ \frac{c_{\infty} - c_{0}}{4} 
- \frac{(c_{\infty} - c_{0})(3c_{\infty} + c_{0})}{32} t^{-1/2}
+ O(t^{-1}),
\label{eq:behavior of lambda0 infinity-plus2} \\[+.5em]
\mu_{0}(t,{\bf c}) & = & 
- \frac{c_{\infty}+c_{0}}{4}t^{-1/2}
+ O(t^{-3/2}),
\label{eq:behavior of mu0 infinity-plus2} \\[+.5em]
\lambda^{(0)}(t,{\bf c},\eta) & = & - t^{1/2} 
+ \frac{c_{\infty} - c_{0}}{4} 
- \frac{(c_{\infty} - c_{0})(3c_{\infty} + c_{0})}{32} t^{-1/2}
+ O(t^{-1}),
\label{eq:behavior of lambda(0) infinity-plus2} \\[+.5em]
\mu^{(0)}(t,{\bf c},\eta) & = &  
- \frac{c_{\infty}+c_{0}-\eta^{-1}}{4}t^{-1/2}
+ O(t^{-3/2}),
\label{eq:behavior of mu(0) infinity-plus2} \\[+.5em]
R_{-1}(t,{\bf c}) & = & - 2 t^{-1/2} 
- \frac{c_{\infty}+c_{0}}{4} t^{-1}
- \frac{(3c_{\infty}-c_{0})(c_{\infty}-3c_{0})}{64}
t^{-3/2} + O(t^{-2}), \nonumber \\
\label{eq:behavior of R-1 infinity-plus2} \\[+.5em]
R_{\pm}(t,{\bf c},\eta) & = & \mp 2 \eta t^{-1/2} 
- \frac{(\pm c_{\infty}\eta \mp c_{0}\eta - 1)}{4} 
t^{-1} \nonumber \\
&  & 
- \frac{( \pm 3c_{\infty}^{2} \eta^{2}
\mp 10 c_{\infty} c_{0} \eta^{2}  \pm 3c_{0}^{2} \eta^{2}  
- 10 c_{\infty} \eta + 6 c_{0} \eta  \mp 1 )}{64\eta}
t^{-3/2} + O(t^{-2}). \nonumber \\
&  &
\label{eq:behavior of R infinity-plus2} 
\end{eqnarray}

\vspace{+.5em}
\noindent
$\bullet$ \underline{The case $t \rightarrow \infty_{3,+}$:}
\begin{eqnarray}
\lambda_{0}(t,{\bf c}) & = & i \hspace{+.1em} t^{1/2} 
+ \frac{c_{\infty} + c_{0}}{4} 
- \frac{i(c_{\infty} + c_{0})(3c_{\infty} - c_{0})}{32} t^{-1/2}
+ O(t^{-1}),
\label{eq:behavior of lambda0 infinity-minus} \\[+.5em]
\mu_{0}(t,{\bf c}) & = & 1 
+ \frac{i(c_{\infty}-c_{0})}{4}t^{-1/2}
+ O(t^{-3/2}),
\label{eq:behavior of mu0 infinity-minus} \\[+.5em]
\lambda^{(0)}(t,{\bf c},\eta) & = & i \hspace{+.1em} t^{1/2} 
+ \frac{c_{\infty} + c_{0}}{4} 
- \frac{i(c_{\infty} + c_{0})(3c_{\infty} - c_{0})}{32} t^{-1/2}
+ O(t^{-1}),
\label{eq:behavior of lambda(0) infinity-minus} \\[+.5em]
\mu^{(0)}(t,{\bf c},\eta) & = & 1 
+ \frac{i(c_{\infty}-c_{0}+\eta^{-1})}{4}t^{-1/2}
+ O(t^{-3/2}),
\label{eq:behavior of mu(0) infinity-minus} \\[+.5em]
R_{-1}(t,{\bf c}) & = & 2i t^{-1/2} 
- \frac{c_{\infty}-c_{0}}{4} t^{-1}
- \frac{i(3c_{\infty}+c_{0})(c_{\infty}+3c_{0})}{64}
t^{-3/2} + O(t^{-2}), \nonumber \\
\label{eq:behavior of R-1 infinity-minus} \\[+.5em]
R_{\pm}(t,{\bf c},\eta) & = & \pm 2i \eta t^{-1/2} 
+ \frac{(\mp c_{\infty}\eta \pm c_{0}\eta+1)}{4} t^{-1} \nonumber \\
&  & 
+ \frac{i(\mp 3c_{\infty}^{2} \eta^{2}
\mp 10 c_{\infty} c_{0} \eta^{2} \mp 3c_{0}^{2} \eta^{2}  
+ 10 c_{\infty} \eta - 6 c_{0} \eta \pm 1 )}{64\eta}
t^{-3/2} + O(t^{-2}). \nonumber \\
&  &
\label{eq:behavior of R infinity-minus} 
\end{eqnarray}

\vspace{+.5em}
\noindent
$\bullet$ \underline{The case $t \rightarrow \infty_{4,+}$:}
\begin{eqnarray}
\lambda_{0}(t,{\bf c}) & = & - i \hspace{+.1em} t^{1/2} 
+ \frac{c_{\infty} + c_{0}}{4} 
+ \frac{i(c_{\infty} + c_{0})(3c_{\infty} - c_{0})}{32} t^{-1/2}
+ O(t^{-1}),
\label{eq:behavior of lambda0 infinity-minus2} \\[+.5em]
\mu_{0}(t,{\bf c}) & = & 1 
- \frac{i(c_{\infty}-c_{0})}{4}t^{-1/2}
+ O(t^{-3/2}),
\label{eq:behavior of mu0 infinity-minus2} \\[+.5em]
\lambda^{(0)}(t,{\bf c},\eta) & = & - i \hspace{+.1em} t^{1/2} 
+ \frac{c_{\infty} + c_{0}}{4} 
+ \frac{i(c_{\infty} + c_{0})(3c_{\infty} - c_{0})}{32} t^{-1/2}
+ O(t^{-1}),
\label{eq:behavior of lambda(0) infinity-minus2} \\[+.5em]
\mu^{(0)}(t,{\bf c},\eta) & = & 1 
- \frac{i(c_{\infty}-c_{0}+\eta^{-1})}{4}t^{-1/2}
+ O(t^{-3/2}),
\label{eq:behavior of mu(0) infinity-minus2} \\[+.5em]
R_{-1}(t,{\bf c}) & = & - 2i t^{-1/2} 
- \frac{c_{\infty}-c_{0}}{4} t^{-1}
+ \frac{i(3c_{\infty}+c_{0})(c_{\infty}+3c_{0})}{64}
t^{-3/2} + O(t^{-2}), \nonumber \\
\label{eq:behavior of R-1 infinity-minus2} \\[+.5em]
R_{\pm}(t,{\bf c},\eta) & = & \mp 2i \eta t^{-1/2} 
+ \frac{(\mp c_{\infty}\eta \pm c_{0}\eta+1)}{4} t^{-1} \nonumber \\
&  & 
- \frac{i( \mp 3c_{\infty}^{2} \eta^{2}
\mp 10 c_{\infty} c_{0} \eta^{2}  \mp 3c_{0}^{2} \eta^{2}  
+ 10 c_{\infty} \eta + 6 c_{0} \eta  \pm 1 )}{64\eta}
t^{-3/2} + O(t^{-2}). \nonumber \\
&  &
\label{eq:behavior of R infinity-minus2} 
\end{eqnarray}

Next we summarize asymptotic behaviors when 
$t$ tends to double-poles.

\vspace{+.5em}
\noindent
$\bullet$ 
\underline{The case $t \rightarrow 0_{c_{\infty},+}$:}
\begin{eqnarray}
\lambda_{0}(t,{\bf c}) & = & 
c_{\infty} - \frac{c_{0}}{c_{\infty}^{2}} \hspace{+.1em} t 
+ \frac{c_{\infty}^{2} - 2c_{0}^{2}}{c_{\infty}^{5}} 
\hspace{+.1em} t^{2} + O(t^{3}),
\label{eq:behavior of lambda0 0-infinity} \\[+.5em]
\mu_{0}(t,{\bf c}) & = & 
\frac{c_{\infty}+c_{0}}{2c_{\infty}}
- \frac{c_{\infty}^{2} - c_{0}^{2}}{2c_{\infty}^{4}}
\hspace{+.1em} t 
- \frac{3c_{0}(c_{\infty}^{2} - c_{0}^{2})} 
{2 c_{\infty}^{7}} \hspace{+.1em} t^{2}
+ O(t^{3}),
\label{eq:behavior of mu0 0-infinity} \\[+.5em]
\lambda^{(0)}(t,{\bf c},\eta) & = & c_{\infty} 
- \frac{c_{0}}{c_{\infty}^{2} - \eta^{-2}} 
\hspace{+.1em} t 
+ \frac{c_{\infty}^{4} - 2 c_{\infty}^{2} c_{0}^{2} 
-2 c_{\infty}^{2} \eta^{-2} 
+ c_{0}^{2} \eta^{-2} + \eta^{-4}}
{c_{\infty}(c_{\infty}^{2}-4\eta^{-2})
(c_{\infty}^{2}-\eta^{-2})^{2}} 
\hspace{+.1em} t^{2}
+ O(t^{3}),
\label{eq:behavior of lambda(0) 0-infinity} \\[+.5em]
\mu^{(0)}(t,{\bf c},\eta) & = & 
\frac{c_{\infty}+c_{0}-\eta^{-1}}{2c_{\infty}}
- \frac{c_{\infty}^{2} - (c_{0} - \eta^{-1})^{2}}
{2 c_{\infty}^{2} (c_{\infty}^{2} - \eta^{-2})} 
\hspace{+.1em} t \nonumber \\[+.5em]
&   & 
- \frac{3 ( c_{\infty}^{2} c_{0} - c_{0}^{3} 
- c_{\infty}^{2} \eta^{-1} + 3 c_{0}^{2} \eta^{-1}   
- 3c_{0}\eta^{-2} + \eta^{-3})}
{2c_{\infty}^{3}(c_{\infty}^{2} - 4\eta^{-2})
(c_{\infty}^{2} - \eta^{-2})}
\hspace{+.1em} t^{2} + O(t^{3}),
\label{eq:behavior of mu(0) 0-infinity} \\[+.5em]
R_{-1}(t,{\bf c}) & = & c_{\infty} t^{-1} 
- \frac{2c_{0}}{c_{\infty}^{2}} 
+ \frac{5 c_{\infty}^{2} - 9c_{0}^{2}}
{2 c_{\infty}^{5}} \hspace{+.1em} t + O(t^{2}), \nonumber \\
\label{eq:behavior of R-1 0-infinity} \\[+.5em]
R_{\pm}(t,{\bf c},\eta) & = & \pm c_{\infty} \eta t^{-1} 
\mp \frac{2 c_{0} \eta}{c_{\infty}^{2} - \eta^{-2}} 
+ \frac{r_{\pm}({\bf c},\eta)}
{2c_{\infty}^{2}(c_{\infty} - \eta^{-1})^{3}
(c_{\infty} + \eta^{-1})^{2}(c_{\infty}^{2} - 4 \eta^{-2})}
\hspace{+.1em} t + O(t^{2}). \nonumber \\[+.3em]
&  & 
\label{eq:behavior of R 0-infinity} \\[+.3em]
r_{\pm}({\bf c},\eta) & = & 
\pm \eta (- 5 c_{\infty}^{6} + 9 c_{\infty}^{4} c_{0}^{2})
+ ( 4 c_{\infty}^{5} - 6 c_{\infty}^{3} c_{0}^{2} )
\pm \eta^{-1} (14 c_{\infty}^{4} + c_{\infty}^{2} c_{0}^{2})
\nonumber \\
&  & 
+ \eta^{-2} (- 8 c_{\infty}^{3} -12 c_{\infty} c_{0}^{2})
\pm \eta^{-3}(-13 c_{\infty}^{2} - 4 c_{0}^{2}) 
+ 4 \eta^{-4} c_{\infty}
\pm 4 \eta^{-5} . \nonumber
\end{eqnarray}

\vspace{+.5em}
\noindent
$\bullet$ \underline{The case $t \rightarrow 0_{c_{0},+}$:}
\begin{eqnarray}
\lambda_{0}(t,{\bf c}) & = & 
\frac{1}{c_{0}} \hspace{+.1em} t 
+ \frac{c_{\infty}}{c_{0}^{4}} \hspace{+.1em} t^{2} 
+ \frac{3c_{\infty}^{2} - c_{0}^{2}}{c_{0}^{7}} 
\hspace{+.1em} t^{3} + O(t^{4}),
\label{eq:behavior of lambda0 0-0} \\[+.5em]
\mu_{0}(t,{\bf c}) & = & 
\frac{c_{\infty}+c_{0}}{2c_{0}}
+ \frac{c_{\infty}^{2} - c_{0}^{2}}{2c_{0}^{4}}
\hspace{+.1em} t 
+ O(t^{2}),
\label{eq:behavior of mu0 0-0} \\[+.5em]
\lambda^{(0)}(t,{\bf c},\eta) & = & 
\frac{1}{c_{0}} t
+ \frac{c_{\infty}}{c_{0}^{2}(c_{0}^{2} - \eta^{-2})} 
\hspace{+.1em} t^{2} 
+ \frac{3c_{\infty}^{2} - c_{0}^{2} + \eta^{-2}}
{c_{0}^{3}(c_{0}^{2} - 4\eta^{-2})
(c_{0}^{2} - \eta^{-2})} 
\hspace{+.1em} t^{3}
+ O(t^{4}),
\label{eq:behavior of lambda(0) 0-0} \\[+.5em]
\mu^{(0)}(t,{\bf c},\eta) & = & 
\frac{c_{\infty}+c_{0}-\eta^{-1}}{2(c_{0} - \eta^{-1})}
+ \frac{c_{\infty}^{2} - (c_{0} - \eta^{-1})^{2}}
{2 c_{0} (c_{0} - 2 \eta^{-1}) (c_{0} - \eta^{-1})^{2}} 
\hspace{+.1em} t + O(t^{2}),
\label{eq:behavior of mu(0) 0-0} \\[+.5em]
R_{-1}(t,{\bf c}) & = & c_{0} t^{-1} 
- \frac{2c_{\infty}}{c_{0}^{2}} 
+ \frac{5 c_{0}^{2} - 9c_{\infty}^{2}}
{2 c_{0}^{5}} \hspace{+.1em} t + O(t^{2}), 
\label{eq:behavior of R-1 0-0} \\[+.5em]
R_{\pm}(t,{\bf c},\eta) & = & (\pm c_{0} \eta + 1) t^{-1} 
\mp \frac{2 c_{\infty} \eta}{c_{0}(c_{0} + \eta^{-1})} 
+ \frac{r_{\pm}({\bf c},\eta)}
{2c_{0}^{2}(c_{\infty} - \eta^{-1})^{3}
(c_{\infty} + \eta^{-1})(c_{\infty}^{2} - 2 \eta^{-2})}
\hspace{+.1em} t + O(t^{2}). \nonumber \\[+.3em]
&  & 
\label{eq:behavior of R 0-0} \\[+.3em]
r_{\pm}({\bf c},\eta) & = & 
\pm \eta (5 c_{0}^{4} - 9 c_{\infty}^{2} c_{0}^{2})
+ ( 11 c_{0}^{3} - 13 c_{\infty}^{2} c_{0} )
\pm \eta^{-1} (-2c_{\infty}^{2} + c_{0}^{2})
- 11 \eta^{-2} c_{0} \mp 6 \eta^{-3} . \nonumber
\end{eqnarray}

\section{The Voros coefficients of the third Painlev\'e
equation of the type $D_{7}$}
\label{Appendix:P3-D7}

We also compute the Voros coefficients of
the degenerate third Painlev\'e equation 
of the type $D_{7}$:
\begin{equation}
(P_{\rm III'})_{D_{7}} : 
\frac{d^{2}\lambda}{dt^{2}} = 
\frac{1}{\lambda} \Bigl( \frac{d \lambda}{dt} \Bigr)^{2}
-
\frac{1}{t} \frac{d \lambda}{dt} 
+
\eta^{2} \Bigl( - \frac{2 \lambda^{2}}{t^{2}}
+ \frac{c}{t} - \frac{1}{\lambda}
\Bigr). \label{eq:P3-D7-appendix}
\end{equation}
Here we assume that 
\begin{equation}
c \ne 0.
\label{eq:condition for D7}
\end{equation} 
Let $F_{D_{7}}(\lambda,t,c)$ be the coefficient 
of $\eta^{2}$ in the right-hand side of 
$(P_{\rm III'})_{D_{7}}$, and $\lambda_{0} = \lambda_{0}(t,c)$
be an algebraic function defined by 
\begin{equation}
F_{D_{7}}(\lambda_{0},t,c) = 
- \frac{2 \lambda_{0}^{2}}{t^{2}}
+ \frac{c}{t} - \frac{1}{\lambda_{0}} = 0.
\end{equation}
Turning points, simple-poles and Stokes curves 
of $(P_{\rm III'})_{D_{7}}$ are defined 
in the same way in Section \ref{section:P-Stokes geometry}
in terms of $\lambda_{0}$. Note that $\lambda_{0}$ has 
the following asymptotic behaviors
\begin{equation}
\lambda_{0}(t,c) = (-2)^{-1/3} \omega^{j} 
t^{2/3} (1 + O(t^{-1/3}))  \hspace{+1.em} 
\text{as $t \rightarrow \infty_{j}$ 
($j = 1,2,3$)}, 
\label{eq:asymptotics of infinity D7}
\end{equation}
where $\omega = e^{2\pi i /3}$, when $t$ tends to infinity, and 
\begin{eqnarray}
\lambda_{0}(t,c) & = & \pm (c/2)^{1/2} t^{1/2} 
(1 + t^{1/2}) \hspace{+1.em} 
\text{as $t \rightarrow \tau_{sp}$}, 
\label{eq:D7 simple-pole} \\[+.5em]
\lambda_{0}(t,c) & = & t/c + O(t^{2}) \hspace{+1.em}
\text{as $t \rightarrow 0_{c}$}, 
\label{eq:D7 double-pole}
\end{eqnarray}
when $t$ tends to 0.
Here we used the symbol $\infty_{j}$ ($j = 1,2,3$), 
$\tau_{sp}$ (which corresponds to the simple-pole 
$u = 0$ of \eqref{eq:quadratic differential P3D7})
and $0_{c}$ (which corresponds to the double-pole
$u = c$ of \eqref{eq:quadratic differential P3D7})
to distinguish the branch of $\lambda_{0}$.
We consider the lift of them onto the Riemann surface of 
$\lambda_{0}$ by taking a new variable $u$ given by 
\begin{equation}
u = 1/\mu_{0}, \hspace{+1.em}
\mu_{0} = \frac{c \lambda_{0} - t}{2\lambda_{0}^{2}}, 
\end{equation}
and the quadratic differential determining the 
Stokes geometry is written as 
\begin{equation}
\partial_{\lambda} F_{D_{7}}(\lambda_{0},t,c) dt^{2} = 
\frac{(3u - 2c)^{3}}{u(u-c)^{2}} du^{2}. 
\label{eq:quadratic differential P3D7}
\end{equation}
On the $u$-plane we have one turning point 
at $u = 2c/3$, and one simple-pole 
at $u = 0$ (which corresponds to \eqref{eq:D7 simple-pole}), 
and a double pole at $u = c$ 
(which corresponds to \eqref{eq:D7 double-pole}).
The following figures (Figure \ref{fig:P3D7,-epsilon} 
$\sim$ \ref{fig:P3D7,+epsilon}) describe 
the $P$-Stokes geometries lifted on the $u$-plane 
near $\arg c = \pi/2$. We can observe that 
Stokes geometry admit a loop-type degeneration 
when $\arg c = \pi/2$. The loop turn around the 
double-pole $u = c$.

  \begin{figure}[h]
  \begin{minipage}{0.31\hsize}
  \begin{center}
  \includegraphics[width=50mm]
  {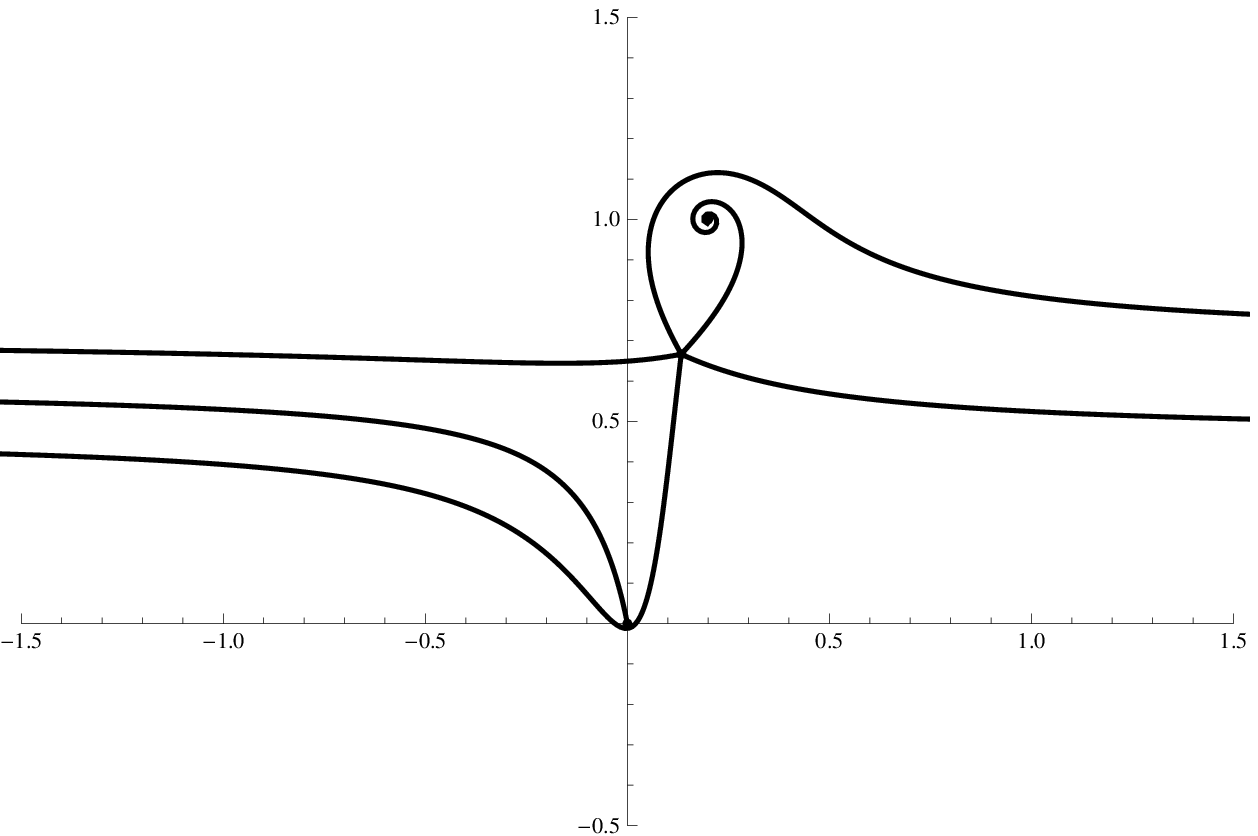}
  \end{center}
  \caption{\small{$c = +0.2+i$.}}
  \label{fig:P3D7,-epsilon}
  \end{minipage} \hspace{+.5em}
  \begin{minipage}{0.31\hsize}
  \begin{center}
  \includegraphics[width=50mm]
  {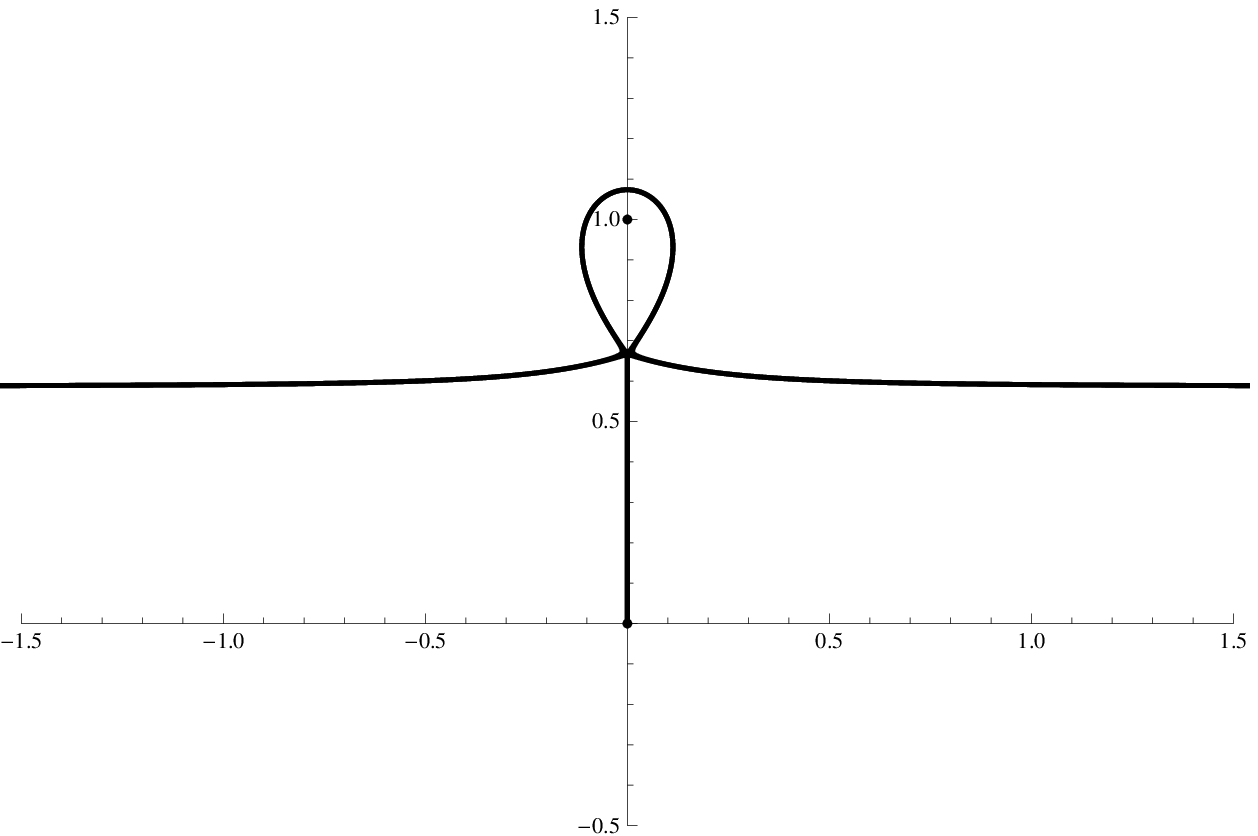}
  \end{center}
  \caption{\small{$c = i$.}}
  \label{fig:P3D7,0}
  \end{minipage} \hspace{+.3em}
  \begin{minipage}{0.31\hsize}
  \begin{center}
  \includegraphics[width=50mm]
  {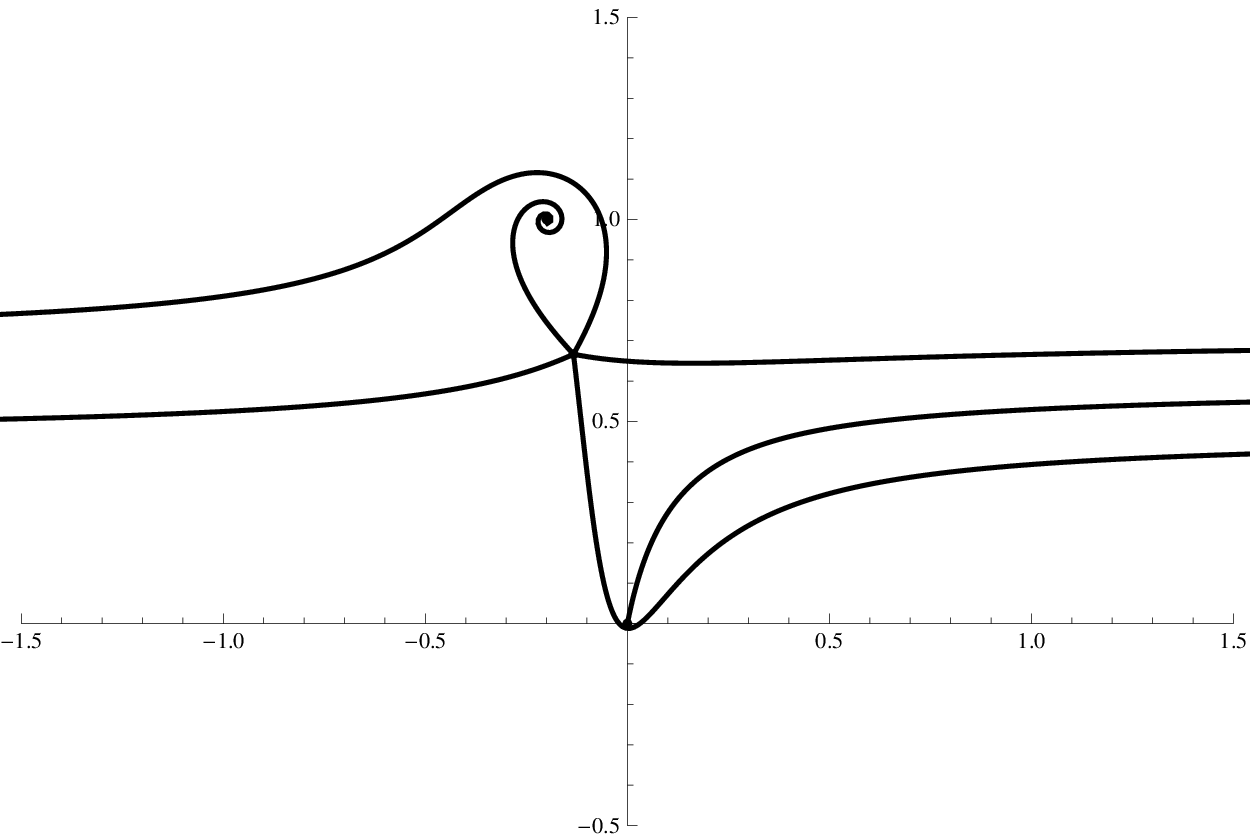}
  \end{center}
  \caption{\small{$c = -0.2+i$}.}
  \label{fig:P3D7,+epsilon}
  \end{minipage}
  \end{figure}

Here we define Voros coefficients for 
$(P_{\rm III'})_{D_{7}}$ as follows. 
Let 
\[
\lambda^{(0)}(t,c,\eta) = \sum_{\ell \ge 0}
\eta^{-\ell} \lambda^{(0)}_{\ell}(t,c)
\] 
be a 0-parameter solution of 
$(P_{\rm III'})_{D_{7}}$, and 
\[
R_{\rm odd}(t,c,\eta) = \sum_{\ell \ge 0} 
\eta^{1-2\ell} R_{2\ell - 1}(t,c)
\]
be the odd part of the formal power series solution 
$R = R(t,c,\eta)$ of the following Riccati equation:
\begin{equation}
R^{2} + \frac{dR}{dt} = 
\Bigl( \frac{2}{\lambda^{(0)}} 
\frac{d \lambda^{(0)}}{dt} - \frac{1}{t} 
\Bigr) R + \eta^{2} \biggl\{ 
\frac{\partial F_{D_{7}}}{\partial \lambda}
\bigl( \lambda^{(0)}, t, c \bigr) 
- 
\eta^{-2} \Bigl( \frac{1}{\lambda^{(0)}}
\frac{d \lambda^{(0)}}{dt} \Bigr)^{2} \biggr\}.
\end{equation}
In order to fix the square root 
$R_{-1}(t,c) = 
\sqrt{\partial_{\lambda}F_{D_{7}}(\lambda_{0},t,c)}$ 
near the infinity and the double-pole,
we use further symbol $\infty_{j, \pm}$ ($j = 1,2,3$)
and $0_{c, \pm}$ such that the following holds:
\begin{equation}
\lambda_{0} = (-2)^{-1/3} \omega^{j} 
t^{2/3} \bigl( 1 + O(t^{-1/3}) \bigr), \hspace{+.5em}
R_{-1} = \pm \bigl\{ 6^{1/2} (-2)^{-1/6} 
\omega^{j/2} t^{1/3} \bigl( 1 + O(t^{-1/3}) \bigr)
\bigr\} \nonumber 
\end{equation}
\begin{equation} 
\text{as $t \rightarrow \infty_{j, \pm}$ ($j = 1,2,3$)}, 
\end{equation}
and 
\begin{equation}
\lambda_{0} = t/c + O(t^{2}), \hspace{+.5em} 
R_{-1} = \pm \bigl\{ c \hspace{+.1em} t^{-1} 
\bigl( 1 + O(t) \bigr) \bigr\} \hspace{+1.em}
\text{as $t \rightarrow 0_{c, \pm}$}.
\end{equation}
%

\begin{defi}\normalfont
Let $\tau$ be the turning point or 
the simple-pole of $(P_{\rm III'})_{D_{7}}$, 
and $\ast = \infty_{j,\pm}$ or $0_{c, \pm}$.
For a path $\Gamma(\tau, \ast)$ from $\tau$ 
to $\ast$ on the $t$-plane,
{\it the Voros coefficient of $(P_{\rm III'})_{D_{7}}$
for the path $\Gamma(\tau, \ast)$} is defined by 
\begin{equation}
W_{\ast}(c,\eta) = 
\int_{\Gamma(\tau, \ast)} \bigl( 
R_{\rm odd}(t,c,\eta) - \eta R_{-1}(t,c)
\bigr)dt.
\end{equation}
\end{defi}

Then we have the following list of Voros coefficients.
\begin{theo} \label{theorem:D7-Voros}
The Voros coefficients are represented 
explicitly as follows: 
\begin{eqnarray}
W_{\infty_{j,\pm}}(c,\eta) & = & 0 
\hspace{+.5em}(j = 1,2,3), \\[+.3em]
W_{0_{c,\pm}}(c,\eta) & = & 
\mp 3 \hspace{+.1em} {\cal G}(c,\eta).
\end{eqnarray}
Here ${\cal G}(c,\eta)$ is the formal power series 
given by \eqref{eq:Voros-coeff-G}.
\end{theo}

This theorem can be proved by the completely same manner 
used in Section \ref{section:proof of the Main theorems}. 
Here we only show the B\"acklund transformation 
for the Hamiltonian system $(H_{\rm III'})_{D_{7}}$
which is equivalent to $(P_{\rm III'})_{D_{7}}$.
\begin{equation}
(H_{\rm III'})_{D_{7}} : \hspace{+.3em}
\frac{d \lambda}{dt} = \eta 
\frac{\partial {\cal H}}{\partial \mu}, \hspace{+1.em}
\frac{d \mu}{dt} = - \eta 
\frac{\partial {\cal H}}{\partial \lambda}, 
\end{equation}
\begin{equation}
t \hspace{+.1em} {\cal H} = 
\lambda^{2} \mu^{2} - (c - \eta^{-1}) \lambda \mu 
+ t \mu + \lambda.
\end{equation}
%
\begin{prop}[e.g., \cite{OKSO}]
Let $(\lambda, \mu)$ be a solution of 
$(H_{\rm III'})_{D_{7}}$. Then, 
$(\Lambda, M) = (\Lambda(\lambda,\mu), M(\lambda,\mu))$ 
defined by 
\begin{eqnarray}
\begin{cases}
\displaystyle \Lambda = - t \mu  
+ \frac{c \hspace{+.1em} t}{\lambda} 
- \frac{t^{2}}{\lambda^{2}}, \\[+.5em]
\displaystyle M = \frac{\lambda}{t},
\end{cases} 
\end{eqnarray}
is a solution of the following equation:
\begin{equation}
\frac{d \Lambda}{dt} = \eta 
\frac{\partial {H}}{\partial M}, \hspace{+1.em}
\frac{d M}{dt} = - \eta 
\frac{\partial {H}}{\partial \Lambda},
\end{equation}
where the Hamiltonian $H$ is given by 
\begin{equation}
t H = \Lambda^{2} M^{2} - c \Lambda M 
+ t M + \Lambda.
\end{equation}
\end{prop}

Using the B\"acklund transformation, we can confirm that 
$W_{\infty_{j, \pm}}(c,\eta)$ satisfies 
\[
W_{\infty_{j, \pm}}(c + \eta^{-1},\eta) 
- W_{\infty_{j, \pm}}(c,\eta) = 0,  
\]
and $\mp W_{0_{c, \pm}}(c,\eta)/3$ satisfies
\eqref{eq:difference equation for G}. 
Thus we can prove the Theorem \ref{theorem:D7-Voros}.

\begin{rem} \normalfont
In parallel with the discussion presented in  
Section \ref{subsection:connection of loop-type}, 
we can conclude that, if the independent variable $t$ 
lies outside of the loop in Figure \ref{fig:P3D7,0}, 
then the parametric Stokes phenomena never occur to 
trannsseries solutions of $(P_{\rm III'})_{D_{7}}$ 
because the Voros coefficients for $\infty$ 
are trivial in this case. However, due to the same 
reason as the case of $(P_{\rm III'})_{D_{6}}$, 
connection formula when $t$ lies inside the loop 
is remains to be analyzed.
\end{rem}


\end{document}